\documentclass[
11pt,                          
english                        
]{article}

\usepackage{titling}
\usepackage{blindtext}
\usepackage{amssymb,latexsym}
\usepackage{amsmath}
\usepackage{graphicx}
\usepackage{amsthm}
\usepackage{relsize}
\usepackage[margin=1in]{geometry}
\usepackage{mathrsfs}  
\usepackage{a4wide}
\usepackage[latin1]{inputenc}
\usepackage{fancyhdr}
\usepackage{pgfplots}
\usepackage{hyperref}
\usepackage{graphicx}
\usepackage{relsize}
\usepackage{newlfont}
\usepackage{mathtools}
\usepackage{tikz}
\usepackage{tikz-cd}
\usetikzlibrary{arrows.meta,
                positioning, shapes.geometric,shapes, backgrounds
                }
\usepackage{enumerate,enumitem}
\usepackage{authblk} 
\usepackage{framed}
\usepackage{comment}

\renewcommand{\b}{\mathfrak{b} }
\newcommand{\g}{\mathfrak{g} }

\newcommand{\K}{\mathbb{K}}
\newcommand{\U}{\mathrm{U}}
\newcommand{\F}{\mathfrak{F}}
\newcommand{\C}{\mathscr{C}}
\newcommand{\D}{\mathscr{D}}
\newcommand{\B}{\mathscr{B}}
\newcommand{\Y}{\mathscr{Y}}

\renewcommand{\L}{\mathscr{L}}
\newcommand{\A}{\mathscr{A}}
\newcommand{\M}{\mathscr{M}}

\newcommand{\LC}{\mathscr{LC}}
\newcommand{\ten}{\otimes  }

\newcommand{\op}{\mathrm{op}}
\newcommand{\cop}{\mathrm{cop}}

\newcommand{\Hom}{\mathrm{Hom}}
\newcommand{\id}{\mathrm{id}}

\newcommand{\ev}{\mathrm{ev}}

\newcommand{\Se}{\check{S}}

\newtheorem {lemma} {Lemma}[section]
\newtheorem {proposition} [lemma] {Proposition}
\newtheorem {theorem} [lemma] {Theorem}
\newtheorem {corollary} [lemma] {Corollary}
\newtheorem {definition}[lemma] {Definition}
\newtheorem {example}[lemma] {Example}

\theoremstyle{definition}
\newtheorem{remark}[lemma]{Remark}

\newcounter{wordlabel}

\renewcommand{\thewordlabel}{\Roman{wordlabel}}

\newcommand{\labelword}[2]{%
  \refstepcounter{wordlabel}%
  #1\textsuperscript{[\scriptsize\thewordlabel]}%
  \label{#2}%
}
\tikzset{
  box/.style={
    draw,
    thick,
    rounded corners,
    align=center,
    text width=5cm,
    inner sep=6pt
  },
  arrow/.style={
    -{Stealth[length=2mm]},
    thick
  },
  alabel/.style={
    font=\small\itshape,
    fill=white,
    inner sep=1pt
  }
}

\begin{document}
\setlength{\parindent}{0pt} 
\title {\textbf{On the Quantization-Dequantization Correspondence for (co)Poisson Hopf Algebras}}
\author[a]{Andrea Rivezzi}
\author[b]{Jonas Schnitzer}

\affil[a]{{\sl
{Mathematical Institute of Charles University}}

{\sl Sokolovsk\'a 49/83, 186 75 Prague 8 , Czech Republic}

{~}}

\affil[b]{{\sl
{Dipartimento di Matematica ``Felice Casorati", Universit\`a degli Studi di Pavia}}
	
{\sl Via Ferrata 5, 27100 Pavia, Italy}}

\date{}

\maketitle
\vspace{-1cm}
\begin{abstract}
\noindent
In this paper, we construct a functorial quantization of (co)Poisson Hopf algebras within a broad categorical framework. We further introduce categories naturally associated with (co)Poisson Hopf algebras, namely \emph{Drinfeld-Yetter modules}. These categories provide a canonical setting in which we define explicit dequantization functors that are inverse to the quantization functors. Using this framework, we also establish functorial (de)quantization results for the corresponding module categories. Finally, we recover the classical results of Etingof and Kazhdan as special cases of our construction and discuss applications to deformation quantization \`a la Tamarkin. 
\end{abstract}
\tableofcontents

\section*{Introduction}
\subsection*{Background and motivation}

Lie bialgebras were introduced by Drinfeld \cite{Dri82,Dri86} as the classical limits of quantum groups, and the question of whether \emph{every} Lie bialgebra admits a quantization, functorially and over any field of characteristic zero, was raised in his list of unsolved problems \cite{DriUnsolved}. Etingof and Kazhdan answered this question positively in their celebrated series of papers \cite{EK1,EK2,EK3,EK4,EK5,EK6}: using a Drinfeld associator \cite{Dri90}, they produced a quantization functor from Lie bialgebras to quantized universal enveloping algebras, and showed that it is an equivalence of categories, with the inverse providing a \emph{dequantization} \cite{EK2}; the invertibility of arbitrary quantization functors was later established by Enriquez and Etingof \cite{EnrEt}, see also \cite{Enr05}. These results have been extended in several directions, most notably by Appel and Toledano Laredo \cite{ATL18,ATL24}, who in particular proved that quantization induces an equivalence between (admissible) Drinfeld--Yetter module categories of a Lie bialgebra and of its quantization.

A conceptually different and remarkably economical approach to quantization was proposed by \v{S}evera \cite{Sev16}. Its starting point is the observation that a cocommutative comonoid $M$ in a braided monoidal category $\B$, together with an \emph{$M$-adapted} comonoidal functor $F\colon \B\to\C$, produces a Hopf monoid $F(M\otimes M)$ in $\C$. Quantization of a Lie bialgebra $\mathfrak b$ then amounts to deforming the (symmetric Cartier) category of Drinfeld--Yetter modules of $\mathfrak b$ by means of a Drinfeld associator, following Cartier \cite{Car93}, and applying this construction to a canonical adapted functor. Pulmann and \v{S}evera subsequently applied these techniques to (infinitesimally braided) Poisson Hopf algebras \cite{PulSev}.

Despite this progress, two fundamental aspects of the quantization--dequantization correspondence have remained open. First, in all the approaches above, \emph{dequantization is only known to exist}: it is obtained either as the inverse of an equivalence of categories or by an inductive argument of Hensel-lemma type \cite{EK2,EnrEt,Zar}, and no construction of the dequantized object in terms of the quantum one is available. This is not a merely aesthetic deficiency. In Tamarkin's proof of Deligne's conjecture and of Kontsevich's formality theorem \cite{tamarkin,Kon97,hinich}, the key step is precisely the dequantization of a Hopf algebra, and since only the existence of the dequantization is known, the resulting $G_\infty$-structure on Hochschild cochains is not written down. Second, the existing theory is essentially confined to universal enveloping algebras of Lie bialgebras in the category of topologically free $\mathbb K[[\hbar]]$-modules, while many Hopf algebras of interest, such as functions on Poisson--Lie groups, brace Hopf algebras, or Hopf algebras in differential graded or topological categories, do not fit this framework.

\subsection*{Main results}

The purpose of this paper is to resolve both issues simultaneously. We establish an \emph{explicit and functorial} quantization--dequantization correspondence for filtered (co)Poisson Hopf monoids in an arbitrary complete and separated filtered symmetric monoidal category $\C$ with enough (co)kernels. More precisely, fixing a Drinfeld associator $\Phi$, we construct four functors
\begin{equation*}
\begin{aligned}
\mathcal{Q}^\Phi_-&\colon \mathrm{qCoPoissH}(\C)\to \mathrm{dqHopf}(\C), &\quad
\mathcal{D}^\Phi_-&\colon \mathrm{dqHopf}(\C)\to \mathrm{qCoPoissH}(\C),\\
\mathcal{Q}^\Phi_+&\colon \mathrm{cqPoissH}(\C)\to \mathrm{cdqHopf}(\C), &\quad
\mathcal{D}^\Phi_+&\colon \mathrm{cdqHopf}(\C)\to \mathrm{cqPoissH}(\C),
\end{aligned}
\end{equation*}
where the prefixes encode the positive-filtration conditions imposed on the relevant structure maps (the cobracket, the bracket, the failure of cocommutativity, and the failure of commutativity, respectively), and we prove:

\begin{itemize}
\item[\textbf{(A)}] \emph{Quantization and dequantization are mutually inverse} (Theorem~\ref{Q-DQ-Eq}): the compositions $\mathcal{Q}^\Phi_\pm\circ\mathcal{D}^\Phi_\pm$ and $\mathcal{D}^\Phi_\pm\circ\mathcal{Q}^\Phi_\pm$ are naturally isomorphic to the identity functors. All four functors leave the underlying object of $\C$ unchanged and act as the identity on morphisms (Theorem~\ref{Thm: QandDQfunctors}).
\item[\textbf{(B)}] \emph{Dequantization is given by an explicit construction} and not merely by the existence of an inverse. The dequantization $\mathcal{D}^\Phi_-(H)$ of a Hopf monoid $H$ is obtained by a sequence of functorial constructions, see Section~\ref{sec-quantdequant-functors}, starting from the Yetter--Drinfeld category of $H$; the dual functor $\mathcal{D}^\Phi_+$ is obtained by the coadapted counterpart of the same procedure.
\item[\textbf{(C)}] \emph{The correspondence lifts to module categories} (Theorem~\ref{Thm: RepCatsEq}): quantization and dequantization induce strongly monoidal equivalences between the categories $\mathfrak{M}_\pm$ of \emph{admissible} Drinfeld--Yetter modules of a (co)Poisson Hopf monoid and the categories $\M_\pm$ of \emph{admissible} Yetter--Drinfeld modules of its (co)quantization.
\end{itemize}

The mechanism underlying these results relies on three ingredients which, we believe, are of independent interest beyond the quantization problem.

\medskip\noindent
\textbf{Drinfeld--Yetter modules for (co)Poisson Hopf monoids.}
The first ingredient is the introduction of a category attached to an arbitrary coPoisson Hopf monoid $C$ (resp.\ Poisson Hopf monoid $P$), whose objects we call \emph{Drinfeld--Yetter modules} and denote by $\mathfrak{DY}(C,\C)$ (resp.\ $\mathfrak{DY}(P,\C)$), see Definition~\ref{def-DY}. A Drinfeld--Yetter module over $C$ is a left $C$-module $V$ together with a ``Lie-type'' coaction $\delta_V\colon V\to V\otimes C$ satisfying a compatibility with the coPoisson cobracket which is the infinitesimal counterpart of the Yetter--Drinfeld condition. We prove that $\mathfrak{DY}(C,\C)$ and $\mathfrak{DY}(P,\C)$ are \emph{symmetric Cartier categories}, i.e.\ symmetric monoidal categories equipped with an infinitesimal braiding in the sense of Cartier \cite{Car93}, whose infinitesimal braiding is built out of the actions and coactions (Theorem~\ref{Thm: DYinf}). These categories carry a distinguished (co)adjoint object $C_-$ (resp.\ $P_+$), and the functor of coinvariants with respect to the action, $\F_-\colon \mathfrak{DY}(C,\C)\to\C$ (resp.\ with respect to the coaction, $\F_+\colon\mathfrak{DY}(P,\C)\to\C$), is adapted (resp.\ coadapted) and recovers $C$ (resp.\ $P$), see Theorem~\ref{thm: (co)adaptedPoiss}. For the universal enveloping algebra $\U(\mathfrak b)$ of a Lie bialgebra, $\mathfrak{DY}(\U(\mathfrak b))$ is isomorphic as a symmetric Cartier category to the usual category of Drinfeld--Yetter modules of $\mathfrak b$ (Lemma~\ref{lemma_isom_DY_to_DY}), so that our notion genuinely extends the classical one; a pictorial definition of these objects appears, under the name \emph{dimodules}, in \cite[Sect.\ 7.2]{PulSev}. We also provide a conceptual interpretation: Drinfeld--Yetter modules are the first-order approximations of the Yetter--Drinfeld modules of a first-order deformation of a (co)commutative Hopf monoid, and the only axiom not accounted for by this interpretation is precisely the obstruction to extending the deformation to second order (Remark~\ref{remark-secord}). With these categories at hand, \v{S}evera's construction admits an \emph{infinitesimal analogue}: an infinitesimally cocommutative comonoid $M$ in a symmetric Cartier category together with an $M$-adapted infinitesimally braided functor $F$ produces a \emph{coPoisson} Hopf monoid structure on $F(M\otimes M)$, with an explicit cobracket (Theorem~\ref{Thm: constrPoissonHopf}); dually, one obtains Poisson Hopf monoids from coadapted functors.

\medskip\noindent
\textbf{Functoriality of Yetter--Drinfeld and Drinfeld--Yetter modules.}
The second ingredient, which should be of interest to the Hopf algebra community independently of any quantization problem, is that \emph{taking Yetter--Drinfeld modules is functorial in the Hopf monoid}. Yetter--Drinfeld categories \cite{LamRad,RadTow,Majid} are usually studied for a fixed Hopf algebra; here we show that a morphism of Hopf monoids $f\colon H\to H'$ induces a braided comonoidal functor $\Y\D(H,\C)\to\Y\D(H',\C)$ (given by a reflexive coequalizer $V\mapsto H'\otimes_H V$), compatible with the adjoint comonoids $H_-$, $H'_-$ and with the coinvariant functors $\mathcal{F}_-$, and dually a braided monoidal functor $\Y\D(H',\C)\to\Y\D(H,\C)$ built from an equalizer. In other words, the assignments
\begin{equation*}
\begin{aligned}
B_-&\colon \mathrm{Hopf}(\C)\longrightarrow \mathrm{bMonCat}^3(\C), &\qquad H&\longmapsto \big(\Y\D(H,\C),H_-,\mathcal{F}_-\big),\\
B_+&\colon \mathrm{Hopf}(\C)\longrightarrow \mathrm{bMonCat}_3(\C), &\qquad H&\longmapsto \big(\Y\D(H,\C),H_+,\mathcal{F}_+\big),
\end{aligned}
\end{equation*}
into suitable categories of triples (braided monoidal category, cocommutative comonoid, adapted functor) and their coadapted duals, extend to a functor and a contravariant functor, respectively (Theorem~\ref{theorem-Brochier}). Conversely, \v{S}evera's construction $(\B,M,F)\mapsto F(M\otimes M)$ is itself functorial (Theorem~\ref{theorem-Severa-functor}), and the composition of the two is naturally isomorphic to the identity of $\mathrm{Hopf}(\C)$ (Proposition~\ref{Lem:counit}). The same holds verbatim for Drinfeld--Yetter modules of (co)Poisson Hopf monoids and for the infinitesimal construction of Theorem~\ref{Thm: constrPoissonHopf} (Theorems~\ref{theorem-Severa-coPoisson} and~\ref{theorem-Brochier-coPoisson}), and the resulting functors are compatible with the deformations of categories discussed below (Theorem~\ref{thm: DA-GT}). This functoriality is what allows us to prove that quantization and dequantization are functors, and that their compositions are naturally isomorphic to the identity, rather than merely isomorphic on objects.

\medskip\noindent
\textbf{Drinfeld associators and the Grothendieck--Teichm\"uller semigroup as inverse equivalences.}
The third ingredient concerns the deformation theory of braided monoidal categories. It is classical, after Cartier \cite{Car93} and Drinfeld \cite{Dri90}, that a Drinfeld associator $\Phi$ deforms a symmetric Cartier category $\B$ into a braided monoidal category $\B_\Phi$. We prove the converse: the Grothendieck--Teichm\"uller semigroup \cite{Dri90} acts on \emph{quasi-symmetric} braided monoidal categories, and the algebraic one-parameter semigroup $g\colon\mathbb K\to\overline{GT}(\mathbb K)$ attached to $\Phi$ provides, by differentiation at $0$, a canonical infinitesimal braiding on the symmetrization $\C_{g(0)}$ of a quasi-symmetric category $\C$ (Theorem~\ref{thm: qasisygivesinfitesimal braiding}). Our main structural result in this direction is that the Drinfeld-associator construction and the Grothendieck--Teichm\"uller construction are \emph{inverse equivalences} between the category of quasi-trivial symmetric Cartier categories and the category of quasi-symmetric braided monoidal categories (Theorem~\ref{Thm:Quasisym=infbraided}), compatible with (co)adapted functors, (co)commutative (co)monoids, and natural transformations (Theorems~\ref{prop: Severa}, \ref{Thm: applGT}, and~\ref{Thm: curveassociators}). This is the core of dequantization: starting from a dequantizable Hopf monoid $H$, the Grothendieck--Teichm\"uller construction equips a suitable subcategory of $\Y\D(H,\C)$ with a symmetric Cartier structure, the infinitesimal \v{S}evera construction turns the coinvariant functor into a coPoisson Hopf monoid, and this is, by definition, $\mathcal{D}^\Phi_-(H)$. We stress that all the structure maps involved, including the infinitesimal braiding, are universal formulas which depend only on $\Phi$ and on the Hopf structure of $H$.
It seems that the authors of \cite{EK2} were already aware of the results concerning the Grothendieck-Teichm\"uller semi-group, but did not state or prove the above theorem. 
 
\medskip
The following diagram summarizes the two directions of the construction. Here $\C$ is a fixed symmetric monoidal category; the restrictions needed for the horizontal equivalences are specified in Section~\ref{Sec: PHAandQuant}.

\begin{center}
\begin{tikzpicture}[node distance=5cm, auto]

\node[box] (SC) {
\textbf{Symmetric Cartier category $\B$}\\[0.3em]
\begin{minipage}[t]{\linewidth}
\begin{itemize}[leftmargin=1.2em,labelsep=0.5em,itemsep=0.2em]
\item $\mathrm{i}$-cocommutative comonoid $M$
\item $M$-adapted, $\mathrm{i}$-braided functor
      $F\colon\B\to\C$
\end{itemize}
\end{minipage}
};

\node[box, right=of SC] (BC) {
\textbf{Braided monoidal category $\B$}\\[0.3em]
\begin{minipage}[t]{\linewidth}
\begin{itemize}[leftmargin=1.2em,labelsep=0.5em,itemsep=0.2em]
\item cocommutative comonoid $M$
\item $M$-adapted braided functor $F\colon\B\to\C$
\end{itemize}
\end{minipage}
};

\node[box, below=3cm of SC] (CPH) {
coPoisson Hopf monoid in $\C$
};

\node[box, right=of CPH] (HA) {
Hopf monoid in $\C$
};

\draw[arrow]
  ([yshift=0.2cm]SC.east) --
  ([yshift=0.2cm]BC.west)
  node[midway, above, alabel]
  {\labelword{Drinfeld associator}{lbl:DA}};

\draw[arrow]
  ([yshift=-0.2cm]BC.west) --
  ([yshift=-0.2cm]SC.east)
  node[midway, below, alabel]
  {\labelword{Grothendieck--Teichm\"uller}{lbl:GT}};

\draw[arrow]
  ([yshift=0.2cm]CPH.east) --
  ([yshift=0.2cm]HA.west)
  node[midway, above, alabel]
  {\labelword{Quantization}{lbl:Q}};

\draw[arrow]
  ([yshift=-0.2cm]HA.west) --
  ([yshift=-0.2cm]CPH.east)
  node[midway, below, alabel]
  {\labelword{Dequantization}{lbl:DQ}};

\draw[arrow]
  ([xshift=-0.2cm]CPH.north) --
  ([xshift=-0.2cm]SC.south)
  node[midway, left, alabel]
  {\labelword{Drinfeld--Yetter}{lbl:DY}};

\draw[arrow]
  ([xshift=0.2cm]SC.south) --
  ([xshift=0.2cm]CPH.north)
  node[midway, right, alabel]
  {\labelword{Cartier--\v{S}evera}{lbl:gS}};

\draw[arrow]
  ([xshift=-0.2cm]HA.north) --
  ([xshift=-0.2cm]BC.south)
  node[midway, left, alabel]
  {\labelword{Yetter--Drinfeld}{lbl:YD}};

\draw[arrow]
  ([xshift=0.2cm]BC.south) --
  ([xshift=0.2cm]HA.north)
  node[midway, right, alabel]
  {\labelword{\v{S}evera}{lbl:S}};

\end{tikzpicture}
\end{center}

All eight arrows are functorial. Arrows~\ref{lbl:DA} and~\ref{lbl:GT} are the inverse equivalences of Theorem~\ref{Thm:Quasisym=infbraided}; arrows~\ref{lbl:YD} and~\ref{lbl:S} are the functors $B_-$ and $\Se_-$ of Theorems~\ref{theorem-Severa-functor} and~\ref{theorem-Brochier}, and arrows~\ref{lbl:DY} and~\ref{lbl:gS} are their (co)Poisson counterparts of Theorems~\ref{theorem-Severa-coPoisson} and~\ref{theorem-Brochier-coPoisson}. Up to natural isomorphism, the quantization functor $\mathcal{Q}^\Phi_-$ is the composite VIII$\circ$I$\circ$V and the dequantization functor $\mathcal{D}^\Phi_-$ is the composite VI$\circ$II$\circ$VII (Section~\ref{sec-quantdequant-functors}); the functors $\mathcal{Q}^\Phi_+$ and $\mathcal{D}^\Phi_+$ arise from the corresponding diagram for Poisson Hopf monoids, coadapted functors, and (infinitesimally) commutative monoids.

\subsection*{Applications}

Besides the generality of the framework, the main gain of our approach is explicitness, and Section~\ref{Sec:Appl} illustrates this on three classical topics.

\medskip\noindent
\textbf{Lie bialgebras.} When specialized to universal enveloping algebras of topologically free Lie bialgebras, our construction recovers the quantization of Etingof--Kazhdan \cite{EK1,EK2} in the form given by \v{S}evera \cite{Sev16}, and we obtain the following concatenation of equivalences of categories
\begin{equation*}
\mathscr{LBA}^0(\mathrm{tfMod})
\xrightarrow{\U}
\mathrm{qCoPoissH}_{\mathrm{qcn}}(\mathrm{tfMod})
\xrightarrow{\mathcal{Q}^\Phi_-}
\mathrm{dqHopf}_{\mathrm{qcn}}(\mathrm{tfMod}),
\end{equation*}
where $\mathrm{tfMod}=\mathrm{tfMod}_{\mathbb K[[\hbar]]}$ is the category of topologically free $\mathbb K[[\hbar]]$-modules, the subscript qcn stands for quasi-conilpotent, and the inverse equivalences are given by taking primitive elements and by $\mathcal{D}^\Phi_-$, respectively. In particular, $\mathcal{D}^\Phi_-$ provides a new, explicit dequantization of quantized universal enveloping algebras (Section~\ref{SubSec: DQLiebialgebras}), in contrast with the existence proofs of \cite{EK2,EnrEt}. On the level of modules, Theorem~\ref{Thm: RepCatsEq} yields the equivalence between the category of admissible Drinfeld--Yetter modules of a Lie bialgebra and that of its quantization, which was established by Appel and Toledano Laredo \cite{ATL18} by completely different techniques.

\medskip\noindent
\textbf{Universal enveloping coalgebras.} The coadapted, Poisson side of our theory produces a result which, to the best of our knowledge, has no counterpart in the literature: a quantization--dequantization correspondence for \emph{universal enveloping coalgebras} of (quasi-conilpotent) Lie bialgebras, together with the corresponding equivalence of module categories (Section~\ref{SubSec: UECoalgebras}). Along the way we establish a dual Milnor--Moore theorem for conilpotent commutative Hopf algebras (Theorem~\ref{Thm: coMM}) and, in Appendix~\ref{App: CoPBW}, a proof of the dual Poincar\'e--Birkhoff--Witt theorem in the conilpotent case, a statement we were not able to locate in the literature in this generality (cf.\ \cite{Mich2}).

\medskip\noindent
\textbf{Tamarkin's proof of Deligne's conjecture.} Tamarkin's $G_\infty$-structure on the Hochschild cochains $C^\bullet(A)$ of an associative algebra $A$ is obtained by dequantizing the brace Hopf algebra $T^c(C^\bullet(A)[1])$. In \cite{tamarkin} this dequantization is the one provided by the Etingof--Kazhdan equivalence between quantized universal enveloping algebras and Lie bialgebras \cite{EK2}, and the structure is transported to $C^\bullet(A)$ by a lifting argument through the operad $B_\infty$, so that the resulting operations are not written down. The functor $\mathcal{D}^\Phi_+$ constructed here replaces both steps: it applies directly to the codequantizable Hopf algebra $T^c(C^\bullet(A)[1])$ in the category of complete filtered (dg) vector spaces, without passing through formal power series, and produces a Poisson Hopf algebra structure whose bracket restricts to the Gerstenhaber bracket. Since the brace Hopf algebra is defined for every brace algebra, i.e.\ for every algebra over the braces operad $\mathrm{Br}$, and since the structure maps of $\mathcal{D}^\Phi_+$ are universal expressions in the Hopf structure of the input, the dequantized structure is a morphism of dg operads $\mathcal{T}\colon G_\infty\to\mathrm{Br}$ whose components are explicit linear combinations of brace trees. We thereby recover Tamarkin's theorem (Theorem~\ref{Thm: Tamarkin}) in a form in which the $G_\infty$-operations can, in principle, be computed.

\medskip
We conclude in Section~\ref{section-Outlook} with an outlook on further applications of the framework, among them the (co)quantization of functions on Poisson--Lie groups, the extension to infinitesimally braided Poisson Hopf algebras in the sense of \cite{PulSev}, Lie bialgebroids \cite{Xu}, and the dequantization of Yetter--Drinfeld modules as a tool to construct homotopy calculus structures on Hochschild (co)chains \cite{Tsygan}.

\subsection*{Organization of the paper}

In Section~\ref{section-preliminaries} we recall adapted and coadapted functors and Yetter--Drinfeld modules over a Hopf monoid in a symmetric monoidal category, and we establish the categorical tools used throughout the paper; in particular, Section~\ref{Subsubsec: FunctorialityI} contains the functoriality of Yetter--Drinfeld modules and of \v{S}evera's construction. Section~\ref{Sec: DA and GT} is devoted to Drinfeld associators and to the Grothen\-dieck--Teich\-m\"ul\-ler semigroup: we construct the infinitesimal braiding attached to a quasi-symmetric category and prove that the two deformation procedures are inverse equivalences of categories. In Section~\ref{Sec: PHAandQuant} we introduce Drinfeld--Yetter modules for (co)Poisson Hopf monoids, study their Cartier structure and their functoriality, and prove our main quantization, dequantization, and module-category equivalences. Section~\ref{Sec:Appl} specializes these results to Lie bialgebras and universal enveloping coalgebras, and gives our version of Tamarkin's argument. We conclude with an outlook in Section~\ref{section-Outlook}. Appendix~\ref{App: CoPBW} contains the dual of the Poincar\'e--Birkhoff--Witt theorem in the conilpotent case.

\subsection*{Acknowledgements}
Our greatest thanks go to Martin Bordemann, who provided us with a complete proof of Theorem \ref{Thm: Bordemann} and, moreover, supported us throughout almost the entire writing process of this paper through fruitful discussions and by pointing us to relevant literature. We would also like to thank Andrea Appel for several helpful discussions, Adrien Brochier for comments that contributed to the development of Section \ref{Subsubsec: FunctorialityI}, and Jan Pulmann and Pavol \v{S}evera for pointing out the existence of the notion of Drinfeld-Yetter modules in their unpublished work \cite{PulmannSevv2}, and for giving hints on the connection with their techniques. Moreover, we are thankful to Boris Tsygan for pointing out that our construction can be used to get a better understanding of Tamarkin's approach to formality. 
A.R. is supported by GA\v{C}R/NCN grant Quantum Geometric Representation Theory and Noncommutative Fibrations 24-11728K and is a member
of the Gruppo Nazionale per le Strutture Algebriche, Geometriche e le loro Applicazioni
(GNSAGA) of the Istituto Nazionale di Alta Matematica (INdAM). This publication is
based upon work from COST Action CaLISTA CA21109 supported by COST (European
Cooperation in Science and Technology). www.cost.eu.
J.S. would also like to thank  Nocte Obducta for the inspiration and atmosphere their music provided during the writing of this work.

\section{Preliminaries}
\label{section-preliminaries}
\subsection{Cartier categories}
We start our discussion by introducing braided monoidal categories \cite{MacLane63}, \cite{JS93}.
\subsubsection{Notations on braided monoidal categories}
We shall denote braided monoidal categories by $ (\C,\otimes , I, a, \ell, r,\sigma)$, where $I$ is the monoidal unit, and $a, \ell, r,\sigma$ denote respectively the associativity, left unit, right unit constraints, and the braiding. We shall denote monoidal functors by $(F,F_2,F_0)$ and comonoidal functors by $(F,F^2,F^0)$. For a braided monoidal category $\C$, we use the term \emph{monoid, comonoid}, and \emph{Hopf monoid} for the usual notion of algebra, coalgebra, and Hopf algebra viewed in $\C$ (we always suppose that all the antipodes are invertible, and denote by $\mu^{(n)}$, resp. $\Delta^{(n)}$ the $n$-th iterated multiplication, resp. comultiplication). It is well-known that (co)monoidal functors send (co)monoids into (co)monoids. In view of Mac Lane coherence (see e.g. \cite[Sect. 2.8]{EGNO}), we shall sometimes strictify formulas, even though all our results hold in the non-strict setting. We denote the tensor flip by $\tau$. If $\sigma$ is the braiding of a braided monoidal category, we denote by $\sigma^2_{X,Y} := \sigma_{Y,X} \circ \sigma_{X,Y}$ and by $\beta_{X,Y,Z,W}$ the middle four interchange map $\beta_{X,Y,Z,W}: (X \ten Y) \ten (Z \ten W) \to (X \ten Z) \ten (Y \ten W)$. If $\sigma$ is symmetric, we shall also employ the notation $\sigma_{\gamma}$ for any permutation $\gamma$ for the corresponding isomorphism. Moreover, we will consistently use the one-line notation for permutations, i.e. for the permutation $\gamma\colon \{1,\dots,n\}\to \{1,\dots,n\}$ we will only write $(\gamma(1),\dots,\gamma(n))$.

Finally, recall that if $F,G : \B\ \to \C$ are two braided comonoidal functors, a natural transformation $n: F \implies G$ is compatible with the braided comonoidal structures if for any pair of objects $X,Y$ in $\B$ one has $(n_X \ten n_Y) \circ F^2(X,Y) = G^2(X,Y) \circ n_{X \ten Y}$ and $F^0 = G^0 \circ n_I$. We always consider fields of characteristic zero, even though some results hold in any characteristic.

 \subsubsection{Infinitesimal braidings and Cartier categories}
A special class of braided monoidal categories, which will play a central role in the following, are the so-called (pre-)Cartier categories \cite{Car93}, \cite{ABSW}, \cite{ERSW}:
\begin{definition}
A pre-additive braided monoidal category $\C$ is said to be pre-Cartier if there is a natural transformation $t: \ten \implies \ten$, called an infinitesimal braiding, satisfying the \emph{infinitesimal hexagon relations} for all objects $X,Y,Z$:
\begin{align}	
t_{X, Y\otimes   Z} &= a_{X,Y,Z} \circ (t_{X,Y} \otimes   \mathrm{id}_Z) \circ a_{X,Y,Z}^{-1} \nonumber \\	& \ + a_{X,Y,Z} \circ (\sigma^{-1}_{X,Y} \otimes   \mathrm{id}_Z) \circ a^{-1}_{Y,X,Z} \circ (\mathrm{id}_Y \otimes   t_{X,Z}) \circ a_{Y,X,Z} \circ (\sigma_{X,Y} \otimes   \mathrm{id}_Z) \circ a^{-1}_{X,Y,Z} \label{eq:pre-cartier-one} \\
t_{X \otimes   Y, Z} &= a^{-1}_{X,Y,Z} \circ (\mathrm{id}_X \otimes   t_{Y,Z}) \circ a_{X,Y,Z} \nonumber \\	& \ +a_{X,Y,Z}^{-1} \circ (\mathrm{id}_X \otimes   \sigma^{-1}_{Y,Z}) \circ a_{X,Z,Y} \circ (t_{X,Z} \otimes   \mathrm{id}_Y) \circ a^{-1}_{X,Z,Y} \circ (\mathrm{id}_X \otimes   \sigma_{Y,Z}) \circ a_{X,Y,Z} \label{eq:pre-cartier-two}
\end{align}
If moreover
\begin{equation}
\label{eq:cartier-category}
\sigma_{X,Y} \circ t_{X,Y} = t_{Y,X} \circ \sigma_{X,Y}
\end{equation}
is satisfied for all objects $X,Y$ we say that $\C$ is a Cartier category. 
\end{definition}
We will mainly consider symmetric Cartier categories, for which the conditions \eqref{eq:pre-cartier-one} and \eqref{eq:pre-cartier-two} are equivalent.
\begin{remark}
Any braided monoidal category can be thought of as a Cartier category with the trivial infinitesimal braiding $t=0$.
\end{remark}
A braided (co)monoidal functor between pre-Cartier categories is said to be infinitesimally braided (co)monoidal if $t_{F(X),F(Y)} \circ F^2(X,Y)   =F^2(X,Y) \circ  F(t_{X,Y})$ \big(resp. $ F_2(X,Y) \circ t_{F(X),F(Y)} = F(t_{X,Y}) \circ F_2(X,Y)\big)$.
A comonoid $(M,\Delta,\varepsilon)$ in a pre-Cartier category $\C$ is said to be \emph{infinitesimally cocommutative} ($i$-cocommutative) if $t \circ \Delta =0$ and $ \sigma \circ \Delta = \Delta$. Similarly, a monoid $(A,\mu,\eta)$ in a pre-Cartier category $\C$ is said to be \emph{infinitesimally commutative} ($i$-commutative) if $\mu \circ t = 0$ and $\mu \circ \sigma = \mu$. 
\subsubsection{The opposite-reversed category of a braided monoidal category}
In the next sections, especially for proofs, it is useful to sometimes consider the opposite and reversed category of a braided monoidal category $ (\C,
  \otimes  , I, a, \ell, r,\sigma)$. In the literature, there are two notions: opposite (see e.g. \cite[Sect. II.2]{MacLaneWorking}) and reversed (in which one considers the opposite tensor product and associativity constraint). We will combine the two at the same time: we denote by $ (\C^{\vee},
  \otimes^{\vee}  , I^{\vee}, a^{\vee}, \ell^{\vee}, r^{\vee},\sigma^{\vee})$ the braided monoidal category which has the same objects as 
  $\C$ and with
  \begin{align*}
\Hom_{\C^{\vee}}(X,Y)&=\Hom_{\C}(Y,X), \qquad 
X\ten^{\vee}Y=Y\ten X, \qquad 
I^{\vee}=I\\
a^{\vee}_{X,Y,Z} &= a_{Z,Y,X}, \qquad \ell^{\vee}_X=r^{-1}_X, \qquad
r^{\vee}_X=\ell^{-1}_X, \qquad \sigma^{\vee}_{X,Y}=\sigma_{X,Y}.
  \end{align*}
If moreover $\C$ has an infinitesimal braiding $t$, then $t^\vee_{X,Y}=t_{Y,X}$
is an infinitesimal braiding for $\C^\vee$. 
It is easy to show the following
\begin{proposition}
Let $\B,\C$ be two braided monoidal categories. Then
\begin{enumerate}
\item If $(F,F^2,F^0)\colon \B\to \C$ is a braided comonoidal functor, then $(F^{\vee},F_{2}^{\vee},F_{0}^{\vee})\colon \B^{\vee}\to \C^{\vee}$,
where $F^{\vee}=F$, $F_{0}^{\vee}=F^0$ and $F_{2}^{\vee}(X,Y)=F^2(Y,X)$ is a braided monoidal functor.
\item  If $(F,F_2,F_0)\colon \B\to \C$ is a braided monoidal functor, then $(F^{\vee},F^{2,\vee},F^{0,\vee})\colon \B^{\vee}\to \C^{\vee}$,
where $F^{\vee}=F$, $F^{0,\vee}=F_0$ and $F^{2,\vee}(X,Y)=F_2(Y,X)$ is a braided comonoidal functor.
\end{enumerate}
\end{proposition}  

\begin{remark}
It is  clear by construction that    $(\C^{\vee})^{\vee}=\C$. 
\end{remark}
\subsection{Adapted and coadapted functors}
In this section, we recall P. \v{S}evera's notion of a functor adapted to a comonoid \cite{Sev16}, and we dualize this notion by defining coadapted functors. We refer to \cite{Sev16} and \cite{Riv} for proofs regarding adapted functors, while we omit the proofs regarding the coadapted ones, since it suffices to pass to the opposite-reversed categories. 
\subsubsection{$M$-adapted and $A$-coadapted functors}
\begin{definition}
Let $\B,\C$ be braided monoidal categories.
\begin{enumerate}
\item Let $(M, \Delta, \varepsilon)$ be a comonoid in $\B$ and $(F, F^2,F^0): \B \to \C$ be a braided comonoidal functor. We say that $F$ is $M$--adapted if for any objects $X,Y$ in $\B$ the morphisms\footnote{In \v{S}evera's paper \cite{Sev16} the map $\gamma_{X,Y}^M$ is denoted by $\tau_{X,Y}^{(M)}$.}
\begin{align}
\chi_M &\coloneqq F^0 \circ F(\varepsilon) \label{m-adapted-condition-one} \\
\gamma_{X,Y}^M & \coloneqq F^2(X \ten M, M \ten Y) \circ F(\alpha_{X,M,M,Y})\circ F\big((\id_X \ten \Delta) \ten \id_Y\big) \label{m-adapted-condition-two}
\end{align}
are invertible, where $\alpha_{X,Y,Z,W}: (X \ten (Y \ten Z)) \ten W  \to (X \ten Y) \ten (Z \ten W)$ is a natural isomorphism made of composition of associativity constraints.
\item Let $(A, \mu, \eta)$ be a monoid in $\B$ and $(F, F_2,F_0): \B \to \C$ be a braided monoidal functor. We say that $F$ is $A$--coadapted if for any objects $X,Y$ in $\B$ the morphisms
\begin{align}
\vartheta_A &\coloneqq F(\eta) \circ F_0 \label{a-adapted-condition-one} \\
\psi_{X,Y}^A & \coloneqq F\big(\id_X \ten (\mu_A \ten \id_Y)\big) \circ F(\alpha^\vee_{Y,A,A,X}) \circ F_2(X \ten A, A \ten Y) \label{a-adapted-condition-two}
\end{align}
are invertible.
\end{enumerate}
\end{definition}
The following result gives an equivalent way to define (co)adapted functors.
\begin{proposition}
\label{proposition-equivalent-condition-M-adapted}
Let $\C,\B$ be two braided monoidal categories.
\begin{enumerate}
\item For a comonoid $M\in \B$, the functor $M \ten - : \B \to \B$ is comonoidal. Moreover, a comonoidal functor $(F,F^2,F^0): \B \to \C$ is $M$-adapted if and only if 
\begin{equation*}
\begin{tikzcd}
\B \arrow[r, "M \ten -"] & \B \arrow[r, "F"] & \C.
\end{tikzcd}
\end{equation*}
is strongly comonoidal.
\item For a monoid $A\in \B$, the functor $-\ten A : \B \to \B$ is monoidal. Moreover, a monoidal functor $(F,F_2,F_0): \B \to \C$ is $A$-coadapted if and only if 
\begin{equation*}
\begin{tikzcd}
\B \arrow[r, "-\ten A"] & \B \arrow[r, "F"] & \C.
\end{tikzcd}
\end{equation*}
is strongly monoidal.
\end{enumerate}
\end{proposition}
As a direct consequence of the last result, we get the following
\begin{corollary}
\label{corollary-M-adapted}
Let $\B, \B',\C$ be three braided monoidal categories, $(M, \Delta,\varepsilon)$ be a comonoid object in $\B$, and $(F,F^2,F^0): \B \to \B'$ and $(G,G^2,G^0): \B'\to \C$ be two braided comonoidal functors such that $F$ is strongly comonoidal and $G$ is $F(M)$-adapted. Then $G \circ F$ is $M$-adapted. 
\end{corollary}

\subsubsection{Construction of Hopf algebras}
The following theorem is the main result regarding (co)adapted functors, i.e., the construction of a Hopf monoid. 
\begin{theorem}
\label{theorem-hopf-algebra}
Let $\B, \C$ be two braided monoidal categories.
\begin{enumerate}
\item Let $(M, \Delta_M,\varepsilon_M)$ be a cocommutative comonoid in $\B$ and $(F, F^2, F^0) : \B \to \C$ be a $M$--adapted functor. Then $F(M \ten M)$ is a Hopf monoid with structure maps
\begin{align}
\mu &= F(r_M \ten \id_M) \circ F((\id_M \ten \varepsilon_M) \ten \id_M) \circ (\gamma_{M,M}^M)^{-1} \\
\eta &= F(\Delta_M) \circ (\chi_M)^{-1} \\
\Delta &= F^2(M \ten M, M \ten M) \circ F(\beta_{M,M,M,M}) \circ F(\Delta_M \ten \Delta_M) \\
\varepsilon &= F^0 \circ F(r_I) \circ F(\varepsilon_M \ten \varepsilon_M)\\
S &= F(\sigma_{M,M}).
\end{align}
Moreover, if $\B$ is symmetric, then the Hopf monoid $F(M \ten M)$ is cocommutative. 
     
\item Let $(A, \mu_A,\eta_A)$ be a commutative monoid in $\B$ and $(F, F_2, F_0) : \B \to \C$ be an $A$--coadapted functor. Then $F(A \ten A)$ is a Hopf monoid with structure maps
\begin{align}
\mu &= F(\mu_A \ten \mu_A) \circ F(\beta_{A,A,A,A}) \circ F_2(A \ten A, A \ten A) \\
\eta &= F(\eta_A \ten \eta_A) \circ F(r^{-1}_I) \circ F_0 \\
\Delta &= (\psi^A_{A,A})^{-1} \circ F((\id_A \ten \eta_A) \ten \id_A) \circ F(r^{-1}_A \ten \id_A) \\
\varepsilon &= (\vartheta_A)^{-1} \circ F(\mu_A)\\
S &= F(\sigma_{A,A}).
\end{align}
Moreover, if $\B$ is symmetric, then the Hopf monoid $F(A \ten A)$ is commutative. 
\end{enumerate}
\end{theorem}
\begin{remark}
Note that if one uses the inverse of the braiding in the previous statement, then one obtains respectively $F(M \ten M)^{\mathrm{cop}}$ and $F(A \ten A)^{\mathrm{op}}$.
\end{remark}

We conclude this paragraph with an interesting observation that we will need throughout the entire paper. 
\begin{proposition}
\label{Prop:nattrafoHopf}
Let $F,G\colon \B\to\C$ be two adapted functors with respect to the same cocommutative comonoid $M\in \B$ and let $n\colon F\implies G$  be a natural transformation of braided comonoidal functors. Then $n_{M\ten M}\colon F(M\ten M)\to G(M\ten M)$ is a morphism of Hopf algebras in $\C$. Similarly, the dual statement holds in the monoidal setting by passing to the opposite-reversed setting. 
\end{proposition}
\begin{proof}
From Equation \eqref{m-adapted-condition-two}, we get, by using that $n$ is a natural transformation of comonoidal functors, that $\gamma_{X,Y}^M\circ (n_{M\ten X}\ten n_{M\ten Y})=n_{X\ten M \ten Y}\circ \gamma_{X,Y}^M,$
for all $X,Y\in \B$. This means in particular, using the naturality of $n$, that $\mu_{G(M \ten M)}\circ (n_{M\ten M} \ten n_{M\ten M})= n_{M\ten M}\circ \mu_{F(M \ten M)}.$
which means that it is a morphism of algebras $n_{M\ten M}\colon F(M\ten M)\to G(M\ten M)$. For the rest of the structure maps the arguments are exactly the same.  
\end{proof}

\subsection{Yetter-Drinfeld modules over a Hopf monoid}
In this section, we recall the category of Yetter-Drinfeld modules over a Hopf monoid in a symmetric monoidal category. In the framework of vector spaces, the reader can find more details in \cite{LamRad}, \cite{RadTow} and in \cite{Kass} (in Kassel's book, this notion is called \emph{crossed bimodule}).
\subsubsection{Definition and braided monoidal structures}
\begin{definition}
Let $H$ be a Hopf monoid in a symmetric monoidal category $\C$. A (left-right) Yetter-Drinfeld module over $H$ is a triple $(V, \mu_V, \Delta_V)$, where $V$ is an object in $\C$ and $\mu_V: H \ten V \to V$,  $\Delta_V: V \to V \ten H$ are respectively a left action and a right coaction satisfying 
\begin{align}
(\mu_V \ten \mu)\circ \sigma_{(1324)} \circ (\Delta \ten \Delta_V) &= (\id \ten \mu) \circ (\Delta_V \ten \id) \circ \sigma \circ (\id \ten \mu_V) \circ (\Delta \ten \id). \label{eq:YD-three}
\end{align}
\end{definition}
A morphism of Yetter-Drinfeld modules is a morphism that is both a morphism of left modules and of right comodules. We shall denote the category of Yetter-Drinfeld modules by $\Y \D(H,\C)$ (we drop the reference to the category if it is clear in which category the Yetter-Drinfeld modules are considered). 
An easy computation shows that the compatibility condition \eqref{eq:YD-three} is equivalent to
\begin{equation}
\label{eq:compatibility_YD_antipode}
\Delta_V \circ \mu_V = (\mu_V \ten \mu^{(2)}) \circ (\id_4 \ten S^{-1}) \circ \sigma_{(51324)} \circ (\Delta^{(2)} \ten \Delta_V).
\end{equation}

In the next Proposition, we exhibit the two different braided monoidal structures on the category $\Y \D(H,\C)$ \cite{RadTow}.
\begin{proposition}
\label{Prop: YDbraidedmonstructure}
Let $H$ be a Hopf monoid object in a symmetric monoidal category $\C$.
\begin{enumerate}
\item  For any $(V, \mu_V, \Delta_V)$, $(W, \mu_W, \Delta_W)$ in $\Y \D(H,\C)$, the following maps endow $V \ten W$ with the structure of a Yetter-Drinfeld module over $H$:
\begin{align}
\mu_{V \ten W} &= (\mu_V \ten \mu_W) \circ \sigma_{(1324)} \circ (\Delta \ten \id \ten \id)\\ 
\Delta_{V \ten W} &= (\id \ten \id \ten \mu^{\mathrm{op}})  \circ \sigma_{(1324)} \circ (\Delta_V \ten \Delta_W). 
\end{align}
This association endows $\Y \D(H,\C)$ with a braided monoidal structure, where the braiding and the inverse braiding are respectively given by 
\begin{align}
\sigma^{\Y \D}_{V,W} &= \sigma_{V,W} \circ (\id \ten \mu_W)\circ ( \id \ten S \ten \id) \circ (\Delta_V \ten \id)  \label{eq:overline-braiding-YD-two} 
\\\
(\sigma_{V,W}^{\Y \D})^{-1} &= (\id \ten \mu_W) \circ (\Delta_V\ten \id) \circ \sigma_{W,V}. \label{eq:overline-braiding-YD-one}
\end{align}
We shall denote this braided monoidal category by $\underline{\Y \D}(H,\C)$.
\item  For any $(V, \mu_V, \Delta_V)$, $(W, \mu_W, \Delta_W)$ in $\Y \D(H,\C)$, the following maps endow $V \ten W$ with the structure of a Yetter-Drinfeld module over $H$:
\begin{align}
\mu_{V \ten W} &= (\mu_V \ten \mu_W) \circ \sigma_{(1324)} \circ (\Delta^{\mathrm{cop}} \ten \id \ten \id)\\ 
\Delta_{V \ten W} &= (\id \ten \id \ten \mu)  \circ \sigma_{(1324)}  \circ (\Delta_V \ten \Delta_W). \label{eq:coaction-DY-tensor-product}
\end{align}
This association endows $\Y \D(H,\C)$ with a braided monoidal structure, where the braiding and the inverse braiding are respectively given by 
\begin{align}\sigma^{\Y \D}_{V,W} &= (\id \ten \mu_V)\circ ( \id \ten S \ten \id) \circ (\Delta_W \ten \id) \circ \sigma_{V,W} \label{eq:braiding-YD-four} \\
(\sigma_{V,W}^{\Y \D})^{-1} &= \sigma_{W,V}\circ(\id \ten \mu_V)\circ (\Delta_W \ten \id)\label{eq:braiding-YD-three} 
\end{align}
We shall denote this braided monoidal category by $\overline{\Y \D}(H,\C)$.
\end{enumerate}
\end{proposition}
It is easy to prove the following
\begin{proposition}
\label{Prop: dualHopf}
Let $(H,\mu,\eta,\Delta,\varepsilon,S)$ be a Hopf monoid in a symmetric monoidal category $\C$. Then $(H^{\vee}\mu^{\vee},\eta^\vee,\Delta^{\vee},\varepsilon^{\vee},S^{\vee})$ is a Hopf monoid in $\C^{\vee}$, where 
\[
\mu^{\vee}=\Delta, \quad\Delta^\vee=\mu, \quad 
 \eta^\vee=\varepsilon, \quad \varepsilon^\vee=\eta, \quad  S^\vee=S.
\]
Moreover, the following functor is a braided monoidal equivalence $$\big( \overline{\Y \D}(H,\C)\big)^\vee\to\  \underline{\Y \D}({H^\vee},{\C^\vee}), \qquad (V,\mu_V,\Delta_V) \mapsto (V,\Delta_V,\mu_V), \qquad f \mapsto f. $$ 
\end{proposition}

\subsubsection{The adjoint and coadjoint Yetter-Drinfeld modules}
\label{section-adj-coadj-YD}
Let $H$ be a Hopf monoid in a symmetric monoidal category. Then there are two distinguished Yetter-Drinfeld modules over $H$, called the adjoint representation and the coadjoint corepresentation, see e.g. \cite{Majid}. We resume their properties in the following 
\begin{proposition}
\label{proposition-Hplusminus}
Let $(H,\mu,\eta, \Delta,\varepsilon,S)$ be a Hopf monoid in a symmetric monoidal category $\C$. 
\begin{enumerate}
\item The triple $H_- = (H, \mu_- ,\Delta_-)$ is an object of $\Y \D(H,\C)$, where 
\begin{align}
\mu_- &=  \mu\\
\Delta_- &= \sigma \circ (\mu^\op \ten \id) \circ (S^{-1} \ten \id \ten \id) \circ \sigma_{(132)} \circ \Delta^{(2)}. 
\end{align}
Moreover, $(H_-,\Delta)$ is a cocommutative comonoid in $\underline{\Y \D}(H, \C)$. In particular, if $(H,\Delta,\varepsilon)$ is cocommutative in $\C$ then $\Delta_-=\id\ten \eta$.
\item The triple $H_+ = (H, \mu_+, \Delta_+)$ is an object of $\Y \D(H, \C)$, where 
\begin{align}
\mu_+ &= \mu^{(2)} \circ \sigma_{(132)} \circ (\id \ten S^{-1} \ten \id) \circ (\Delta^{\cop} \ten \id)  \\
\Delta_+ &= \Delta  .
\end{align}
Moreover, $(H_+,\mu)$ is a commutative monoid in $\overline{\Y \D}(H,\C)$. In particular, if  $(H,\mu,\eta)$ is commutative in $\C$ then $\mu_+=\varepsilon\ten \id$. 
\end{enumerate}
\end{proposition}

In general, our constructions require categories with additional structure. We refer to such categories as \emph{good enough}, or more precisely:

\begin{definition}
\label{Def:goodenough}
    A symmetric monoidal category $\C$ is called 
    \begin{enumerate}
        \item $k$-good enough, if it has all kernels and  the functor $Z\ten-\colon \C \to \C $ preserves kernels for all $Z\in \C$.
        \item $c$-good enough, if it has all cokernels and the functor $Z\ten-\colon \C \to \C $ preserves cokernels for all $Z\in \C$ 
    \end{enumerate}
\end{definition}

In the following, we will often only use the term \emph{good enough}, since it will be clear from the context whether we need kernels or cokernels, see Remark  \ref{remark_good_enough}. Recall that monoidal pre-abelian categories are both $k$-good enough and $c$-good enough. 

\begin{theorem}
\label{Thm: (co)adaptedHopffunctors}
Let $H$ be a Hopf monoid in a good enough category $\C$ and consider the functors
\begin{align}
&\mathcal{F}_- : \ \underline{\Y \D}(H,\C) \to \C,  \quad (V, \mu_V, \Delta_V) \mapsto \mathrm{coker} (\mu_V - \varepsilon \ten \id), \quad f \mapsto f|_{\mathrm{coker}} \label{eq:F-minus} \\
&\mathcal{F}_+ : \ \overline{\Y \D}(H, \C) \to \C,  \quad (V, \mu_V, \Delta_V) \mapsto \mathrm{ker}  (\Delta_V -  \id\ten \eta), \quad \ \ f \mapsto f|_{\ker} \label{eq:F-plus} 
\end{align}
together with their braided monoidal (resp. monoidal) structure 
\begin{align}
&\mathcal{F}_{-}^2 (V,W):  \mathrm{coker} (\mu_{V\ten W} - \varepsilon \ten \id)\to \mathrm{coker} (\mu_V - \varepsilon \ten \id)\ten \mathrm{coker} (\mu_W - \varepsilon \ten \id) \\
&\mathcal{F}_{+,2}(V,W) : \mathrm{ker}  (\Delta_V -  \id\ten \eta)\ten \mathrm{ker}  (\Delta_W-  \id\ten \eta) \to \mathrm{ker}  (\Delta_{V\ten W} -  \id\ten \eta)
\end{align}
which are induced by the universal properties of kernels and cokernels, see e.g. \cite[Sect. VIII]{MacLaneWorking}.
\begin{enumerate}
\item The functor $\mathcal{F}_-(H_-\ten - )$ is comonoidally naturally isomorphic to the forgetful functor, and therefore $\mathcal{F}_-$ is $H_-$-adapted.
\item The functor $\mathcal{F}_+(-\ten H_+)$ is monoidally naturally isomorphic to the forgetful functor, and therefore $\mathcal{F}_+$ is $H_+$-coadapted.
\item As Hopf monoids we have $\mathcal{F}_+(H_+ \ten H_+) \cong \mathcal{F}_-(H_- \ten H_-) \cong H.$
\end{enumerate}
\end{theorem}
\begin{proof}
$(ii):$ Let us check that we can corestrict the canonical map 
\[\mathrm{ker}  (\Delta_V -  \id\ten \eta)\ten \mathrm{ker}  (\Delta_W-  \id\ten \eta)\to V\ten W\] to the kernel of $\Delta_{V\ten W}-\id\ten \eta$.
Consider two objects $(V,\mu_V,\Delta_V), (W,\mu_W,\Delta_W)$ in $\overline{\Y\D}(H,\C)$. Then $\mathrm{ker}  (\Delta_V -  \id\ten \eta)$ (resp. $\mathrm{ker}  (\Delta_W -  \id\ten \eta)$) is the equalizer of the maps $\Delta_V$  and $ \id\ten \eta$  (resp. $\Delta_W$  and $ \id\ten \eta$). 
This means that, using the formula \eqref{eq:coaction-DY-tensor-product}, the concatenation 
    \begin{align*}
        \mathrm{ker}  (\Delta_V -  \id\ten \eta)\ten \mathrm{ker}  (\Delta_W-  \id\ten \eta)\to V\ten W \overset{\Delta_{V\ten W}}\longrightarrow V\ten W\ten H 
    \end{align*}
is equal to $\id\ten \eta$. By the universal property of the kernel, we get a unique map 
\begin{align*}
    \mathrm{ker}  (\Delta_V -  \id\ten \eta)\ten \mathrm{ker}  (\Delta_W-  \id\ten \eta) \to \mathrm{ker}  (\Delta_{V\ten W} -  \id\ten \eta),
\end{align*}
which is exactly $\mathcal{F}_{+,2}(V,W)$. It is now easy to check that $\mathcal{F}_+$ is monoidal using the universal property of the kernels.    
Next, we show that $\mathcal{F}_+$ is braided. 
Recall that the braiding in $\overline{\Y\D}(H,\C)$ is given by Equation \eqref{eq:braiding-YD-four}, and using that $\mathrm{ker}  (\Delta_V -  \id\ten \eta)$ (resp. $\mathrm{ker}  (\Delta_W -  \id\ten \eta)$) is the equalizer of the maps $\Delta_V$  and $ \id\ten \eta$ (resp. $\Delta_W$  and $ \id\ten \eta$), one checks that the map 
\begin{align*}
\mathrm{ker}  (\Delta_V -  \id\ten \eta)\ten \mathrm{ker}  (\Delta_W-  \id\ten \eta)\to V\ten W \overset{\sigma_{V,W}^{\Y \D}}\longrightarrow W\ten V
    \end{align*}
is equal to 
\begin{align*}
    \mathrm{ker}  (\Delta_V -  \id\ten \eta)\ten \mathrm{ker}  (\Delta_W-  \id\ten \eta)\overset{\sigma}\longrightarrow\mathrm{ker}  (\Delta_W -  \id\ten \eta)\ten \mathrm{ker}  (\Delta_V-  \id\ten \eta)\to W\ten V, 
\end{align*}
which shows that $\mathcal{F}_+$ is braided. 
Finally, to show that $\mathcal{F}_+(-\ten H_+)$ is isomorphic to the identity, it suffices to consider the composition \[\mathcal{F}_+(X\ten H_+)= \ker(\Delta_{X\ten H_+}-\id\ten \eta)\to X\ten H_+ 
        \overset{\id \ten \varepsilon }\longrightarrow X\]
whose inverse is given by $(\id\ten S)\circ\Delta_V \colon X\to X\ten H_+$
which clearly restricts to the kernel of $\Delta_{X\ten H_+}-\id\ten \eta$. Note that here it is crucial that the antipode is invertible. 
The proof of the functoriality of $\mathcal{F}_-$ follows the same lines, or one can alternatively use Proposition \ref{Prop: dualHopf} and see that $\mathcal{F}_+^\vee=\mathcal{F}_-$. Part $(iii)$ of the statement is a straightforward computation using the explicit isomorphism $\mathcal{F}_+(H_+\ten H_+)\cong H$
on the level of objects. 
\end{proof}

\begin{remark}
\label{remark_good_enough}
Note that, in the case of $\mathcal{F}_+$ we only need kernels, and in the situation of $\mathcal{F}_-$, we only need cokernels. This means in particular that we in fact assume that our Hopf monoid $H$ is either in a $k$-good enough or in a 
$c$-good enough category, depending on the context. This also holds true for all the further proofs.
\end{remark}

\subsubsection{Construction of Yetter-Drinfeld  modules}
We now provide a way to construct Yetter-Drinfeld modules starting from an adapted or a coadapted functor.
\begin{theorem}
\label{Thm: Bordemann}
Let $\B$ be a braided monoidal category and $\C$ be a symmetric monoidal category.
\begin{enumerate}
\item Let $(M,\Delta,\varepsilon)$ be a cocommutative comonoid in $\B$ and $(F,F^2,F^0): \B \to \C$ be a $M$-adapted functor.  
\begin{enumerate}
\item For any object $X$ in $\B$, the morphisms 
\begin{align*}
\mu_{F(M \ten X)} & = F(r_M \ten \id_X) \circ F((\id_M \ten \varepsilon_M) \ten \id_X) \circ (\gamma_{M,X}^M)^{-1} \\
\Delta_{F(M \ten X)} &= F^2(M \ten X, M \ten M) \circ F( \id \ten \sigma_{M,M} \ten \id)\circ F(\id \ten \sigma^{-1}_{X,M})  \circ F(\Delta^{(2)} \ten \id)
\end{align*}
endow $F(M \ten X)$ with the structure of a Yetter-Drinfeld module over $F(M \ten M)$.
\item The assignment 
\begin{align}
\label{eq:Gamma}
\Gamma_F\colon \B \to\underline{\Y \D}\big({F(M \ten M)},\C\big), \quad X \mapsto F(M \ten X), \quad f \mapsto F(\id_M \ten f)
\end{align}
is a strongly braided comonoidal functor.
\item The comonoids $\Gamma_F(M)$ and $F(M\ten M)_-$ coincide in $\underline{\Y \D}({F(M \ten M)},\C)$.
\item If $\B$ is symmetric, then $\Delta_{F(M\ten X)}$ is trivial for all $X\in \B$. 
\end{enumerate}
\item
Let  $(A,\mu,\eta)$ be a commutative monoid in $\B$ and $(F,F_2,F_0): \B \to \C$ be an $A$-coadapted functor. Then 
\begin{enumerate}
\item For any object $X$ in $\B$, the morphisms 
\begin{align*}
\Delta_{F(X\ten A)} &= (\psi_{X,A}^A)^{-1}\circ F(\id\ten(\eta \ten \id)) \circ F(\id\ten \ell_A^{-1})\\
\mu_{F(X \ten A)} &= F(\id\ten\mu^{(2)})\circ F(\sigma_{X,A}^{-1}\ten \id)\circ F(\id\ten \sigma_{A,X}\ten \id)\circ F_2(A\ten A, X\ten A)
\end{align*}
endow $F(X\ten A)$ with the structure of a Yetter-Drinfeld module over $F(A \ten A)$.
\item The assignment
\begin{align}
\label{eq:Xi}
\Xi_F\colon \B \to\ \overline{\Y \D}\big({F(A \ten A)}, \C\big), \quad X \mapsto F(X\ten A) , \quad f \mapsto F(f \ten \id_A)
\end{align}
is a strongly braided monoidal functor.
\item The monoids $\Xi_F(A)$ and $F(A\ten A)_+$ coincide in $\overline{\Y \D}({F(A \ten A)}, \C)$.
\item If $\B$ is symmetric, then $\mu_{F(X\ten A)}$ is trivial for all $X\in \B$. 
\end{enumerate}
\end{enumerate}
\end{theorem}
\begin{proof}
The proof of statement $(a)$ is contained in \cite[Sect. 7.9]{Riv}. Statement $(b)$ is claimed in \cite{Sev16}, and we are deeply grateful to M. Bordemann for providing us a complete and detailed proof, which we do not insert here. The proof itself is a rather straightforward, but long diagram chase. 
In order to prove statement $(c)$, note first that, as objects in $\C$, one clearly has $\Gamma_F(M)= F(M \ten M)_- = F(M \ten M)$. In order to show that they also coincide in $\underline{\Y \D}({F(M \ten M)},\C)$, we use the following trick: if $H$ is a Hopf monoid and we consider the  left action $\mu_H : H \ten H \to H$ given by the multiplication, then one can show (by using Equation \eqref{eq:compatibility_YD_antipode}) that there exists a unique right coaction $\Delta_H : H \to H \ten H$ satisfying $\Delta_H(1) = 1 \ten 1$, such that $(H,\mu_H,\Delta_H)$ is a Yetter-Drinfeld module. The two Yetter-Drinfeld $F(M \ten M)$-modules $\Gamma_F(M)$ and $ F(M \ten M)_-$ clearly have the same action, and both coactions satisfy the above property, hence they must coincide. Finally, part $(d)$ follows by a straightforward computation. Part $(ii)$ follows the same lines as $(i)$ by passing to the opposite-reversed category.
\end{proof}

Next, we collect some facts about the induced functors $\Gamma$ and $\Xi$ in the following
\begin{theorem}
\label{Prop:natTraf}
Let $\B, \B'$ be braided monoidal categories and $\C$ be a symmetric monoidal category.
\begin{enumerate}
\item Let $H\in\C$ be a Hopf monoid, assume that $\C$ is good enough and let $\mathcal{F}_-\colon \underline{\Y\D}(H,\C)\to \C$ and $\mathcal{F}_+\colon \overline{\Y\D}(H,\C)\to \C$ be the functors defined in Equations \eqref{eq:F-minus}-\eqref{eq:F-plus}. Then
\begin{align}
&\Gamma_{\mathcal{F}_-}\colon \underline{\Y\D}(H,\C)\to \underline{\Y\D}\big(\mathcal{F}_-(H_-\otimes H_-),\C\big) \\
&\Xi_{\mathcal{F}_+}\colon \overline{\Y\D}(H,\C)\to \overline{\Y\D}\big(\mathcal{F}_+(H_+\otimes H_+),\C\big)
\end{align}
are equivalences of categories.
\item Let $F,G\colon\B\to\C$ be two braided comonoidal functors which are $M$-adapted with respect to a cocommutative comonoid $M$, and let $n\colon F\implies G$ be a natural comonoidal transformation such that $n_{M\ten X}\colon F(M\ten X)\to G(M\ten X)$ is an isomorphism for all $X$ in $\B$. Then $n_{M\ten M}\colon F(M\ten M)\to G(M\ten M)$ is an isomorphism of Hopf algebras 
and therefore it induces an invertible braided strongly comonoidal functor $\mathcal{I}_{n_{M\ten M}}\colon \underline{\Y\D}(F(M\ten M),\C)\to \underline{\Y\D}(G(M\ten M),\C)$. Furthermore,
$n$ induces a natural isomorphism $\mathcal{I}_{n_{M\ten M}}\circ\Gamma_F\implies \Gamma_G$. Moreover, the dual counterpart of the statement also holds.
\item Let $G\colon \B'\to \B$ be a braided strongly comonoidal functor, $M\in \D$ be a  cocommutative comonoid, and $F\colon \B\to \C$ be a braided comonoidal functor which is $G(M)$-adapted. Then there is a natural isomorphism $\Gamma_{F\circ G}\implies \Gamma_F\circ G$. Moreover, the dual counterpart of the statement (which involves the functors $\Xi$'s) also holds. 
\item
Let $M\in \B$ be a cocommutative comonoid, let $\C$ be good enough and let $F\colon \B\to \C$ be $M$-adapted.
There is a natural comonoidal transformation $n^-\colon \mathcal{F}_-\circ \Gamma_F\implies F$ such that for all $X\in \B$ the maps $n^-_{{M\ten X}}\colon (\mathcal{F}_-\circ \Gamma_F)(M\ten X)\to F(M\ten X)$ are isomorphisms. Dually, let $A\in \B$ be a commutative monoid and let $F\colon \B\to \C$ be $A$-coadapted.
There is a natural monoidal transformation $n^+\colon F\implies \mathcal{F}_+\circ \Xi_F$ such that for all $X\in \B$ the maps $n^+_{X\ten A}\colon F(X\ten A)\to \mathcal{F}_+\circ \Xi_F(X\ten A)$ are isomorphisms.
\end{enumerate}
\end{theorem}

\begin{proof}
The first statement follows by the fact that the identification $H\cong \mathcal{F}_+(H_+\ten H_+)$ (resp. $H\cong\mathcal{F}_-(H_-\ten H_-)$) implies that any Yetter-Drinfeld module over $H$, $X$, is isomorphic to $\mathcal{F_+}(X\ten H_+)$ (resp. $\mathcal{F}_-(H_-\ten X)$).
The second statement follows similar lines. One can check that 
$F(M\ten X)\cong G(M\ten X)$ as Yetter-Drinfeld modules using the identification $F(M\ten M)\cong G(M\ten M)$ as Hopf monoids. 
The third statement also follows the same line by using that $G$ is strongly monoidal and therefore $G(M\ten M)\cong G(M)\ten G(M)$ as comonoids and 
$G(M\ten X)\cong G(M)\ten G(X)$.
Next, we prove the first part of statement $(iv)$, as the dual part follows the same lines.
Let $X\in \B$. Recall from Theorem \ref{Thm: Bordemann} that the action $\mu_{F(M\ten X)}$ of $F(M\ten M)$ on $F_M(X)=F(M\ten X)$ is given by
\begin{center}
\begin{tikzcd}
F(M\ten M)\ten F(M\ten X) \arrow[rr, "\mu_{F(M\ten X)}"] \arrow[rd, "{(\gamma^{M}_{M,X})^{-1}}"'] &  & F(M\ten X) \\                  & F((M\ten M)\ten X) \arrow[ru, "F((\id\ten\varepsilon)\ten \id )"'] &           
\end{tikzcd}.
\end{center}
On the other hand, the following diagram also commutes
\begin{center}
\begin{tikzcd}
F(M\ten M)\ten F(M\ten X) \arrow[rr, "\varepsilon_{F(M\ten M)}\ten \id"] \arrow[rd, "{(\gamma_{M,X}^M)^{-1}}"'] &                                                                     & F(M\ten X) \\
& F((M\ten M)\ten X) \arrow[ru, "F((\varepsilon\ten \id)\ten \id )"'] &           
\end{tikzcd}.
\end{center}
This means in particular that the morphism $F(\varepsilon \ten \id): F(M \ten X) \to F(X)$
maps the image of $\mu_{F(M\ten X)}-(\varepsilon_{F(M\ten M)}\ten \id)$ to zero and defines a map
\[ n^-_X:\mathcal{F}_-(F(M\ten X))=\mathrm{coker}\big(\mu_{F(M\ten X)}-(\varepsilon_{F(M\ten M)}\ten \id)\big)\to F(X)\]
which is natural on $X$ by definition. The fact that $n^-$ is compatible with the comonoidal structures follows from the commutativity of the following diagram 
\begin{center}
\begin{tikzcd}
F(M\otimes(X\otimes Y)) \arrow[rr, "F(\varepsilon \ten \id_{X \ten Y})"] \arrow[d, "{F(\beta_{M,M,X,Y}\circ \Delta \ten \id_{X \ten Y})}"'] &  & F(X \ten Y) \arrow[d, "\id"]        \\
F((M \ten X) \ten (M \ten Y)) \arrow[d, "{F^2(M \ten X, M \ten Y)}"']                                                                       &  & F(X \ten Y) \arrow[d, "{F^2(X,Y)}"] \\
F(M \ten X) \ten F(M \ten Y) \arrow[rr, "F(\varepsilon \ten \id) \ten F(\varepsilon \ten \id)"]                                             &  & F(X) \ten F(Y)                     
\end{tikzcd}.
\end{center}
Finally, one can check that the map 
\begin{align*}
F(M\ten X)&\overset{\Delta\ten \id}{\longrightarrow}F((M\ten M)\ten X)\to F(M\ten(M \ten X))
\\&\to\mathrm{coker}(\mu_{F(M\ten(M \ten X))}-\varepsilon_{F(M\ten(M \ten X))}\ten \id)=\mathcal{F}_- \circ\Gamma_F(M\ten X)
\end{align*}
is an inverse for $n^-_{M\ten X}$. 
We will not prove the second part of statement $(iv)$ since it follows from the dual reasoning: here one first consider the map $F(\id\ten \eta)\colon F(X)\to F(X\ten A)$, and then see that we get unique maps $n^+_X\colon F(X)\to \ker\big(\Delta_{F(X\ten A)}-(\id\ten \eta_{F(A\ten A)})\big)$
which we define as $n^+_X\colon F(X)\to \mathcal{F}_+(\Xi_F(X))$. 
\end{proof}

\subsubsection{Functoriality}
\label{Subsubsec: FunctorialityI}
In this subsection, we briefly discuss that the constructions corresponding to arrows \ref{lbl:YD} and \ref{lbl:S} (see the introduction), i.e., constructing a Hopf algebra out of a braided monoidal category and taking the Yetter-Drinfeld modules, can be made functorial.
Since the subsequent results are of minor importance for the rest of the paper, we only sketch most of the constructions. Nevertheless, we hope that the following discussion may shed some light on the nature of the quantization/dequantization of (co)Poisson Hopf algebras and we believe that it may serve as a foundation of future research directions. 
For a fixed good enough category $\C$, we denote by $\mathrm{Hopf}(\C)$ the category of Hopf monoids in $\C$, and we define the categories $\mathrm{bMonCat}^3(\C)$ and 
$\mathrm{bMonCat}_3(\C)$ as follows
    \begin{enumerate}
        \item Objects of $\mathrm{bMonCat}^3(\C)$ are triples $(\B,M,F)$, where $\B$ is a braided monoidal category, $M\in \B$ is a cocommutative comonoid and $F\colon \B\to \C$ is an $M$-adapted comonoidal functor. Morphisms in $\mathrm{bMonCat}^3(\C)$ are triples $(\phi, \psi,n)\colon (\B,M,F)\to (\B',M',F')$, where $\phi=(\phi,\phi^2,\phi^0):\B\to \B'$ is a braided comonoidal functor, $\psi:\phi(M)\to M'$ is a morphism of comonoids, and $n\colon F\implies F'\circ \phi$ is a natural  comonoidal transformation. 
         \item Objects of $\mathrm{bMonCat}_3(\C)$ are triples $(\B,A,F)$, where $\B$ is a braided monoidal category, $A\in \B$ is a commutative monoid, and $F\colon \B\to \C$ is an $A$-coadapted monoidal functor. Morphisms in $\mathrm{bMonCat}_3(\C)$ are triples $(\varphi,\chi,m)\colon (\B,A,F)\to (\B',A',F')$, where $\varphi = (\varphi, \varphi_2,\varphi_0):\B\to \B'$ is a braided monoidal functor, $\chi: A' \to \varphi(A)$ is a morphism of monoids, and $m\colon F'\circ \phi\implies F$ is a natural monoidal transformation. 
    \end{enumerate}
Apart from set-theoretic issues, it is clear that both define categories by just concatenating functors and natural transformations. Note that the natural transformations in Theorem \ref{Prop:natTraf} (iv) are going in the opposite directions, so Theorem \ref{Prop:natTraf} does not provide a morphism in the respective categories. We will address this issue later. 

\begin{theorem}
\label{theorem-Severa-functor}
The assignments 
\begin{align*}
\Se_-\colon\mathrm{bMonCat}^3(\C)&\to \mathrm{Hopf}(\C)\\
\Se_+\colon\mathrm{bMonCat}_3(\C)&\to \mathrm{Hopf}(\C)
\end{align*}
defined on objects by $\Se_-(\B,M,F)=F(M\ten M)$ (resp. $\Se_+(\B,A,F)=F(A\ten A)$) by means of Theorem \ref{theorem-hopf-algebra} and on morphisms by 
\begin{align*}
\Se_-(\phi,\psi,n) &:= F'(\psi \ten \psi) \circ F' (\phi^2(M,M)) \circ n_{M \ten M} : F(M \ten M) \to F'(M'\ten M') \\
\Se_+(\varphi,\chi,m) &:=  m_{A \ten A}  \circ F' (\varphi_2(A,A)) \circ F'(\chi \ten \chi)  : F'(A' \ten A') \to F(A\ten A)
\end{align*}
define respectively a functor and a contravariant functor.
\end{theorem}
\begin{proof}
    The only thing that we have to prove is that the assignment of the morphism gives a morphism of Hopf monoids, which is a straightforward computation. 
\end{proof}

Next, we want to establish functors\footnote{We decided to name such functors "$B_\pm$" to thank A. Brochier for his comments, which enlightened us on how to show functoriality of arrows \ref{lbl:YD} and \ref{lbl:S}.}
\begin{align*}
B_-&\colon \mathrm{Hopf}(\C)\to \mathrm{bMonCat}^3(\C) \\
B_+&\colon \mathrm{Hopf}(\C)\to \mathrm{bMonCat}_3(\C)
\end{align*}
which extend the assignments 
\begin{align*}
B_-(H)&=(\underline{\Y\D}(H,\C),H_-,\mathcal{F}_-)\\
B_+(H)&=(\overline{\Y\D}(H,\C),H_+,\mathcal{F}_+)
\end{align*}
to morphisms. We first consider the case of $B_-$: let $H,H'\in \mathrm{Hopf}(\C)$ and $f\colon H\to H'$ be a morphism of Hopf monoids. Our aim is to define a triple $B_-(f) = (\phi_f, \psi_f,n_f)$, where $\phi_f = (\phi_f,\phi^2_f,\phi_f^0) : \underline{\Y\D}(H,\C) \to \underline{\Y\D}(H',\C)$ is a braided comonoidal functor, $\psi_f : \phi_f(H_-) \to H'_-$ is an isomorphism of comonoids, and $n_f:\mathcal{F}_- \to \mathcal{F}'_- \circ \phi_f$ is a natural comonoidal transformation. 
Let $(V,\mu_V,\Delta_V)\in \underline{\Y\D}(H,\C)$. It is easy to see that the map $\Delta_V' := (\id \ten f)\circ\Delta_V\colon V\to V\ten H'$ endows
$V$ with a right $H'$-comodule structure.
Moreover, the object $H'\ten V$ can also be endowed with a right $H'$-comodule structure with the following coaction
    \begin{align*}
        \Delta_{H'\ten V}:=(\id_{H'\ten V}\ten \mu^{(2)}) \circ (\id^{\ten 4} \ten S^{-1} ) \circ \sigma_{(51324)} \circ (\Delta^{(2)}\ten \Delta_V').
    \end{align*}
It is easy to check that the left action $ \mu_{H'\ten V}=\mu\ten \id\colon H'\ten H'\ten V\to H'\ten V$ makes the triple $(H'\ten V,\mu_{H'\ten V},\Delta_{H'\ten V})$ an object in $\underline{\Y\D}(H',\C)$.
Moreover, the assignment 
\[\underline{\Y\D}(H,\C) \to \underline{\Y\D}(H',\C), \quad(V,\mu_V,\Delta_V)\mapsto (H'\ten V, \mu_{H'\ten V},\Delta_{H'\ten V})\]
is functorial.
However, it fails to assign $H_-$ to $H_-'$. In order to fix that, we denote by $q\colon H\ten V\to H'\ten_H V$ the coequalizer (or equivalently, the cokernel of the difference) of the maps 
\begin{align*}
(\mu\ten\id)\circ (\id \ten f\ten \id)\colon H'\ten H \ten V &\to H'\ten V\\
\id\ten \mu_V\colon H'\ten H \ten V &\to H'\ten V.
\end{align*}
Note that, using the unit of $H$, these two maps have a common section $H'\ten V\to H'\ten H\ten V$ (this is what is called a reflexive coequalizer). 
Moreover, let us consider the maps 
\begin{align*}
    H'\ten V \overset{\Delta_{H'\ten V}}{\longrightarrow} H'\ten V\ten H'&\overset{q\ten \id}\longrightarrow (H'\ten_H V)\ten H' \\
     H'\ten H'\ten V \overset{\mu_{H'\ten V}}{\longrightarrow} H'\ten V &\overset{q}\to H'\ten_H V.
\end{align*}
One checks that the first maps coequalizes $(\mu\ten\id)\circ (\id \ten f\ten \id)$ and $\id\ten \mu_V$  and therefore descends to a map $\Delta_{H'\ten_HV}\colon (H'\ten_HV)\to (H'\ten_HV)\ten H'. $
The second map coequalizes $\id \ten [(\mu\ten\id)\circ (\id \ten f\ten \id)]$  and $\id \ten (\id\ten \mu_V)$, and therefore descends to a map from their coequalizer,
which is given by $H'\ten (H'\ten_HV)$. This means that we obtain a map $\mu_{H'\ten_HV}\colon H'\ten (H'\ten_HV) \to H'\ten_HV$.
Note that, since $(H'\ten V, \mu_{H'\ten V},\Delta_{H'\ten V})$ is a Drinfeld-Yetter module and tensoring preserves coequalizers, it follows that $(H'\ten_H V, \mu_{H'\ten_H V},\Delta_{H'\ten_H V})$ is also a Yetter-Drinfeld module. This assignment 
defines a functor 
    \begin{align*}
        \phi_f \colon \underline{\Y\D}(H,\C) \to \underline{\Y\D}(H',\C), \quad  (V,\mu_V,\Delta_V)\mapsto (H'\ten_HV,\mu_{H'\ten_H V},\Delta_{H'\ten_H V}).
    \end{align*}
Let us now define its braided comonoidal structure. In order to define the natural transformation $\phi_f^2(V,W)\colon \phi_f(V\ten W)\to \phi_f(V)\ten \phi_f(W)$, we consider the map 
    \begin{align*}
        H'\ten V\ten W \overset{\Delta\ten \id}\to H'\ten H'\ten V\ten W \overset{\sigma_{(1324)}}\to H'\ten V\ten H'\ten W \overset{q\ten q}\to (H'\ten_HV)\ten (H'\ten_HW) 
    \end{align*}
and note that it coequalizes the maps \[(\mu\ten\id)\circ (\id \ten f\ten \id), \id\ten \mu_{V\ten W}\colon H'\ten H\ten V\ten W \to H' \ten V \ten W.\] Therefore, it descends to a map $ \phi_f^2\colon H'\ten_H(V\ten W) \to (H'\ten_HV)\ten (H'\ten_HW).$
With the same line of thought we define the map $\phi_f^0\colon H'\ten_HI\to I$, and one can show that the triple $(\phi_f,\phi_f^2,\phi_f^0)\colon \underline{\Y\D}(H,\C)\to \underline{\Y\D}(H',\C)$ defines a braided comonoidal functor. Moreover, the morphism
we have 
    \begin{align*}
        H'\overset{\id\ten \eta}\to H'\ten H\overset{q}\to H'\ten_HH,
    \end{align*}
is an isomorphism, whose  inverse is constructed via
    \begin{align*}
        H'\ten H \overset{\id\ten f}\to H'\ten H' \overset{\mu}\to H'
    \end{align*}
which can be easily checked to descend to the coequalizer and is an inverse using the universal property of (reflexive) coequalizers. One can also check that this isomorphism (in $\C$) is in fact a morphism between the Yetter-Drinfeld modules 
$H_-'$ and $\phi_f(H_-)$. It is even an isomorphism of comonoids, which we choose as our $\psi_f$. 
Next, we construct the natural comonoidal transformation $n_f:\mathcal{F}_- \Rightarrow \mathcal{F}'_- \circ \phi_f$. For $(V,\mu_V,\Delta_V)\in \underline{\Y\D}(H,\C)$, we define the following  morphism (in $\C$) 
\begin{align*}
   (\tilde{n}_f)_V\colon  V\to H'\ten V\to H'\ten_H V
\end{align*}
where the first map is given by the unit of $H'$. Since it satisfies $(\tilde{n}_f)_V\circ\mu_V= \mu_{H'\ten_H V}\circ (f\ten (\tilde{n}_f)_V) $, we have that $\tilde{n}_f$ descends to a morphism $(n_f)_V\colon \mathcal{F}_-(V)\to \mathcal{F}_-'(H'\ten_H V)$. It is clear by construction that it is a natural comonoidal transformation.

Let us sketch how to define the triple $B_+(f) = (\varphi_f, \chi_f,m_f)$ making $B_+$ a (contravariant) functor. Let $H,H' \in \mathrm{Hopf}(\C)$, $f: H' \to H$ be a morphism of Hopf monoids, $(V',\mu_{V'},\Delta_{V'})\in \overline{\Y\D}(H',\C)$, and consider the map $\mu'_{V'}=\mu_{V'}\circ (f\ten \id)\colon H\ten V' \to V',$
which is obviously a left $H$-action on $V'$. We endow $V'\ten H$ with the following left $H$-module structure
    \begin{align*}
        \mu_{V'\ten H}=    (\id\ten \mu^{(3)})\circ(\id^{\ten 3}\ten S^{-1})\circ\sigma_{(4213)} \circ (\Delta \ten \id_{V'\ten H})
    \end{align*}
which, together with $\Delta_{V'\ten H}=\id\ten \Delta\colon V'\ten H\to V'\ten H\ten H$, turns $V'\ten H$ into an $H$-Yetter-Drinfeld module. Following the same lines as above, one shows that this structure induces a Yetter-Drinfeld module on the equalizer of the two maps 
    \begin{align*}
        \Delta_V\colon V\ten H \to H\ten H'\ten H \\
        (\id\ten f\ten \id)\circ (\id\ten \Delta)\colon V\ten H \to H\ten H'\ten H 
    \end{align*}
which we denote by $V'\ten^{H'}H$ (here we need that tensoring preserves equalizers). We denote such $H$-Yetter-Drinfeld module by $(V\ten^{H'}H,\mu_{V\ten^{H'}H},\Delta_{V\ten^{H'}H})$. This assignment provides a functor 
\[ \varphi_f \colon \overline{\Y\D}(H',\C) \to \overline{\Y\D}(H,\C), \quad  (V',\mu_{V'},\Delta_{V'})\mapsto (V\ten^{H'}H,\mu_{V\ten^{H'}H},\Delta_{V\ten^{H'}H}).\]
Next, we define the natural transformation $\varphi_{f,2}(V',W')\colon \varphi_f(V')\ten \varphi_f(W')\to \varphi_f(V'\ten W')$ as the corestriction of 
$$(V'\ten^{H'}H)\ten (W'\ten^{H'}H)\to V'\ten H\ten W'\ten H\to V'\ten W'\ten H\ten H\to (V'\ten W') \ten H$$
to the equalizer. Similarly, the morphism $\varphi_{f,0}\colon I\to \varphi(I)$, is given by the corestriction of $I\to I\ten H$ to the equalizer.
One can check that $(\varphi_f,\varphi_{f,2}, \varphi_{f,0}) : \overline{\Y\D}(H',\C)\to \overline{\Y\D}(H,\C)$ is a braided monoidal functor inducing an isomorphism of monoids $\chi_f : \varphi_f(H'_+)\to H_+$.
Finally, the natural transformation $m_f\colon \mathcal{F}'_+\implies \mathcal{F}_+\circ \varphi_f $ is given by the corestriction to the equalizer of the morphism $\mathcal{F}_+'(V')\to V' \to V'\ten H$ given by tensoring with the unit. 
Let us summarize our findings in the following 
\begin{theorem}
\label{theorem-Brochier}
Let $\C$ be a good enough category. Then the assignments
\begin{align*}
B_-\colon \mathrm{Hopf}(\C)\to \mathrm{bMonCat}^3(\C)\\
B_+\colon \mathrm{Hopf}(\C)\to \mathrm{bMonCat}_3(\C)
\end{align*}
defined on objects by $B_-(H)=(\underline{\Y\D}(H,\C),H_-,\mathcal{F}_-)$ and $B_+(H)=(\overline{\Y\D}(H,\C),H_+,\mathcal{F}_+)$, and on morphisms by $B_-(f) = (\phi_f, \psi_f,n_f)$ and $B_+(f)= (\varphi_f,\chi_f,m_f)$ define respectively a functor and a contravariant functor.
\end{theorem}

The obvious question to ask now is the relationship between $\Se_\pm$ and $B_\pm$, and at least one concatenation gives us a very satisfying answer:
\begin{proposition}
\label{Lem:counit}
The endofunctors $\Se_-\circ B_-\colon \mathrm{Hopf}(\C)\to \mathrm{Hopf}(\C)$ and $\Se_+\circ B_+\colon\mathrm{Hopf}(\C)\to \mathrm{Hopf}(\C)$ are isomorphic to the identity. 
\end{proposition}
\begin{proof}
From Theorem \ref{Thm: (co)adaptedHopffunctors} we know that $\Se_\pm(B_\pm(H))=\mathcal{F}_\pm(H_\pm\ten H_\pm)\cong H$
 is an isomorphism of Hopf algebras, and one easily checks that this isomorphism is natural in morphisms of Hopf monoids. 
\end{proof}
Let us now focus on the compositions
\begin{align*}
&B_-\circ\Se_-\colon \mathrm{bMonCat}^3(\C)\to \mathrm{bMonCat}^3(\C)\\
&B_+\circ\Se_+\colon \mathrm{bMonCat}_3(\C)\to \mathrm{bMonCat}_3(\C).
\end{align*} 
At first sight, there is no way to find natural transformations between them and their respective identities. In order to see that, let us consider the \emph{minus} case. Recall that, by Theorem \ref{Thm: Bordemann}, for every $(\B,M,F)\in \mathrm{bMonCat}^3(\C)$ one can obtain a braided comonoidal functor $\Gamma_F\colon \B\to \underline{\D\Y}(F(M\ten M),\C)$, such that $\Gamma_F(M)=F(M\ten M)_-$. Furthermore, by Theorem \ref{Prop:natTraf} (iv), we obtain a natural transformation $n^-\colon \mathcal{F}_-\circ \Gamma_F\implies F$.
If the morphism $n^-_X\colon  \mathcal{F}_-\circ \Gamma_F(X)\to F(X)$ would be invertible for every $X$, then $$(\Gamma_F,\id,(n^-)^{-1})\colon (\B,M,F) \to (\underline{\D\Y}(F(M\ten M),\C), F(M\ten M)_-,\mathcal{F}_-)=B_-(\Se_-(\B,M,F))$$ would be a morphism in $\mathrm{bMonCat}^3(\C)$. Note that for 
the image of $B_-$ this is always the case, so we can restrict to the full subcategory of $\mathrm{bMonCat}^3(\C)$ where this natural transformation is invertible and call it $\mathrm{bMonCat}^3(\C)'$. It is not clear whether this is enough to show that, when (co-)restricting to this subcategory, there is an adjunction between $\Se_\pm$ and $B_\pm$. Since this is not relevant for the present paper, we postpone these investigations to future work.

\subsection{$\mathbb{K}[[\hbar]]$-modules and filtered categories}
We now introduce topologically free modules (for which we refer to \cite[$XVI$]{Kass}, \cite[App. A]{SakSev} for more details) and filtered braided monoidal categories. 
\subsubsection{Topologically free modules}
Let $\mathbb{K}$ be a field and $\hbar$ be a formal parameter. We denote by $\mathbb{K}[[\hbar]]$ the usual ring of formal power series with coefficients in $\mathbb{K}$. It is well-known that $\mathbb{K}[[\hbar]]$ possesses a topology induced by the inverse system of rings $\mathbb{K}[\hbar] / (\hbar^n)$, which is called the inverse limit topology, or the $\hbar$-adic topology. 
\begin{definition}
\label{def-top-free-modules}
A $\mathbb{K}[[\hbar]]$-module $M$ is called 
\begin{enumerate}
\item complete, if the canonical map $M\to \lim_{\infty\leftarrow n} \frac{M}{\hbar^n M}$ is surjective;
\item separated, if $\bigcap_{n\geq 1}\hbar^n M=\{0\}$;
\item torsion free if the map $M\ni m\to \hbar m\in M$ is injective;
\item topologically free if it is torsion-free, complete, and separated.
\end{enumerate}
\end{definition}
Another well-known way to interpret the above properties of modules over $\mathbb{K}[[\hbar]]$ is through the map $o\colon M\to \mathbb{N}\cup \{\infty\}$ defined by $o(m)=\sup\{n\in \mathbb{N}\ | \ m\in \hbar^n M\}$. Namely, if one considers the map $d\colon M\times M\to \mathbb{R}$ defined by $d(m_1,m_2)=2^{-o(m_1-m_2)}$ (setting $2^{-\infty}=0$), then it is possible to show that $d$ is a metric if and only if $M$ is separated, whence if $M$ is separated, then $M$ is complete if and only if it is complete as a topological space.
Moreover, if $f\colon M\to N$ is a morphism of $\mathbb{K}[[\hbar]]$-modules and $M$  and $N$ are separated, then $f$ is automatically continuous with respect to the respective metrics. 
It is well-known that a $\mathbb{K}[[\hbar]]$-module is topologically free if and only if it is of the form $V[[\hbar]]$ for some vector space $V$.
The category of all topologically free modules is a braided monoidal category, where the constraint and the tensor product are given by the $\hbar$-adic completions of those of the category of vector spaces. We shall denote it by $\mathrm{tfMod}_{\mathbb{K}[[\hbar]]}$ and call Hopf monoids in $\mathrm{tfMod}_{\mathbb{K}[[\hbar]]}$ \emph{topological Hopf algebras}. Moreover, it is worth mentioning that $\mathrm{tfMod}_{\mathbb{K}[[\hbar]]}$ has kernels (i.e., is $k$-good enough in the sense of Definition \ref{Def:goodenough} ), which will be important later on, so let us sketch the idea for the proof here: let $f\colon M\to N$ be a morphism between topologically free $\mathbb{K}[[\hbar]]$-modules. As a submodule of a torsion-free and separated module, $\ker f$ is separated and torsion-free as well. Moreover, since $f$ is continuous, its kernel is a closed subspace and as such also complete with respect to the subspace topology, which comes from the filtration $\ker f \cap \hbar^kM$
Using the torsion-freeness of $N$, one checks 
$\hbar^k\ker f =\ker f \cap \hbar^kM$ for all $k$ and this shows the claim. 
    \\
The category $\mathrm{tfMod}_{\mathbb{K}[[\hbar]]}$ misses however some of the features we will need in this paper, namely it is not $c$-good enough. So it is sometimes convenient to pass to the category of complete and separated $\mathbb{K}[[\hbar]]$-modules, which we denote by $\mathrm{csMod}_{\mathbb{K}[[\hbar]]}$. There is an adjunction
\begin{equation*}
\begin{tikzcd}
{\mathrm{csMod}_{\mathbb{K}[[\hbar]]}} \arrow[rr, "U", bend left] & \top & {\mathrm{Mod}_{\mathbb{K}[[\hbar]]}} \arrow[ll, "\mathrm{Compl}_\hbar", bend left]
\end{tikzcd}
\end{equation*}
where $U\colon \mathrm{csMod}_{\mathbb{K}[[\hbar]]}\to \mathrm{Mod}_{\mathbb{K}[[\hbar]]}$ denotes the forgetful functor and $\mathrm{Compl}_\hbar : \mathrm{Mod}_{\mathbb{K}[[\hbar]]} \to \mathrm{csMod}_{\mathbb{K}[[\hbar]]}$ is defined on objects by $M \mapsto \lim_{\infty\leftarrow n} \frac{M}{\hbar^n M}$.
The latter is a strongly monoidal functor between the algebraic tensor product in $\mathrm{Mod}_{\mathbb{K}[[\hbar]]}$ and the completed tensor product $M\hat{\ten}N:=\mathrm{Compl}_\hbar(U(M)\ten U(N))$. 
Note that the category $\mathrm{csMod}_{\mathbb{K}[[\hbar]]}$ 
is $c$-good enough, which is clear by this adjunction. Nevertheless, kernels of morphisms between complete and separated modules may fail to be complete, and their completion fails to be a kernel, which means that $\mathrm{csMod}_{\mathbb{K}[[\hbar]]}$ is not $k$-good enough. 

\subsubsection{Filtered monoidal categories}
Next, we introduce a more general framework, namely that of filtered braided monoidal categories.

\begin{definition}
\label{def-filtered-stuffs}
Let $\C$ be a $\mathbb{K}$-linear category. We say that $\C$ is filtered if for any pair of objects $X,Y$ the vector space $\mathrm{Hom}(X,Y)$ possesses a filtration 
    \begin{align*}
        \Hom(X,Y)=\mathrm{F}^0\Hom(X,Y)
        \supseteq 
        \mathrm{F}^1\Hom(X,Y)
        \supseteq 
        \mathrm{F}^2\Hom(X,Y)
        \supseteq
        \dots ,  
    \end{align*}
such that $\mathrm{F}^l\Hom(Y,Z)\circ \mathrm{F}^k\Hom(X,Y)\subseteq
\mathrm{F}^{l+k}\Hom(X,Z)$ for all $X,Y,Z\in \C$. If moreover $\C$ is monoidal, we require that the maps $\Hom(X_1,Y_1)\times   \Hom(X_2,Y_2) \to \Hom(X_1\otimes   X_2,Y_1\otimes   Y_2)$ are filtration preserving, i.e. 
$\mathrm{F}^k\Hom(X_1,Y_1)\times   \mathrm{F}^l\Hom(X_2,Y_2) \to \mathrm{F}^{k+l}\Hom(X_1\otimes   X_2,Y_1\otimes   Y_2)$ for all $X_1,X_2,Y_1,Y_2\in \C$. If $\D$ is another filtered category and $F\colon  \C \to \D$ is a functor, we say that $F$ is filtered if $F(\mathrm{F}^k\Hom_\C(X,Y))\subseteq \mathrm{F}^k\Hom_\D(F(X),F(Y))$
for all $X,Y\in \C$ and $k\in \mathbb{N}_0$. In the same spirit of Definition \ref{def-top-free-modules}, we say that $\C$ is separated if for all $X,Y\in \C$, we have 
    \begin{equation*}
        \bigcap_{k\geq 0}\mathrm{F}^k\Hom(X,Y)=\{0\},
    \end{equation*}
and we call $\C$ complete if 
the maps $\Hom(X,Y)\to \lim_{\infty\leftarrow n}\frac{\Hom(X,Y)}{\mathrm{F}^n\Hom(X,Y)}$
are surjective. 
\end{definition}
\begin{example}
\label{Ex: PSareFilt}
The category $\mathrm{Mod}_{\mathbb{K}[[\hbar]]}$ is filtered by
\begin{align*}
\mathrm{F}^k\Hom(M,N)=\{ \phi \in \Hom(M,N)\ | \phi(M)\subseteq \hbar^kN\}.
\end{align*}
This filtration is neither separated nor complete with respect to $\mathrm{Mod}_{\mathbb{K}[[\hbar]]}$, but restricted to the subcategories $\mathrm{csMod}_{\mathbb{K}[[\hbar]]}$ and $\mathrm{tfMod}_{\mathbb{K}[[\hbar]]}$, endowed with the same filtration, are complete and separated.     
\end{example}
Another example of a filtered category, where formal power series are not involved, will be discussed in Section \ref{Sec:Appl}. Before that, the reader can think of $\mathrm{csMod}_{\mathbb{K}[[\hbar]]}$ or $\mathrm{tfMod}_{\mathbb{K}[[\hbar]]}$ if we mention a filtered category. 
\begin{definition}
Let $\C$ be a filtered monoidal category.  
\begin{enumerate}
\item A braiding $\sigma$ on $\C$ is called quasi-symmetric if $\sigma_{X,Y}^2-\id_{X\otimes   Y}\in \mathrm{F}^1\Hom(X\otimes   Y,X\otimes   Y). $
\item An infinitesimal braiding $t$ on $\C$ is called quasi-trivial if $t_{X,Y}\in \mathrm{F}^1\Hom(X\otimes   Y,X\otimes   Y)$.

\end{enumerate}
\end{definition}
In the following, if not stated differently,  if we refer to a \emph{quasi-trivial Cartier category} or to a \emph{quasi-symmetric category}, we shall always suppose that we talk about a complete and separated filtered monoidal category.
\subsubsection{Auxiliary categories}
\label{section-aux-categories}
In this section, we construct some auxiliary categories that we are going to use later on. We will not go into detail, this definition is only meant for fixing notation. Moreover, the statements which are not proven are straightforward to show. 
\begin{definition}
\label{Def: AuxCats}
${}$
\begin{enumerate}
\item Let $\C$ be a preadditive symmetric monoidal category. We denote by $\C_x$ the category with the same objects as $\C$ and with morphisms 
\begin{align*}
\Hom_{\C_x}(X,Y)=\frac{\Hom_{\C}(X,Y)[x]}{x^2\Hom_{\C}(X,Y)[x]}.
\end{align*}
\item Let $\C$ be a filtered category. We denote by $[\C]$ the category which has the same objects as $\C$ and with morphisms 
    \begin{align*}
        \Hom_{[\C]}(X,Y)=\frac{\Hom_\C(X,Y)}{\mathrm{F}^1\Hom_\C(X,Y)}.
    \end{align*} 
We denote by $[-]\colon \C\to [\C]$ the obvious strongly monoidal functor. 
\item Let $\C$ be a complete and separated filtered category. For an indeterminate $\nu$, we denote by $\C_\nu$ the category
which has the same objects as $\C$ and as morphisms
\begin{align*}
\Hom_{\C_\nu}(X,Y)=\lim_{n} \frac{\Hom_{\C}(X,Y)[\nu]}{\mathrm{F}^n\Hom_{\C}(X,Y)[\nu]}
\end{align*}
that have the following separated and complete filtration 
\begin{align*}
\mathrm{F}^k \Hom_{\C_\nu}(X,Y) =\ker 
\Bigg(\Hom_{\C_\nu}(X,Y)\to 
\frac{\Hom_{\C}(X,Y)[\nu]}{\mathrm{F}^k\Hom_{\C}(X,Y)[\nu]}\Bigg).
\end{align*}
One can show that this is the subspace of $\Hom_{\C}(X,Y)[[\nu]]$ which is polynomial up to any filtration degree.  
Moreover, there is a canonical map $\Hom_\mathcal{\C}(X,Y)\hookrightarrow \Hom_{\mathcal{\C}_\nu}(X,Y)$
which embeds morphisms as constants. 
There are also evaluation functors 
$\ev_s\colon \C_\nu\to \C$ which are filtered in the sense of Definition \ref{def-filtered-stuffs}. The derivative with respect to $\nu$ is 
\[ \frac{d}{d\nu}\colon \Hom_{\C}(X,Y)[\nu] \to \Hom_{\C}(X,Y)[\nu], \quad  \sum_{i=0}^k\nu^if_i\mapsto \sum_{i=1}^k i\nu^{i-1}f_i
\]
and it clearly fulfills the Leibniz identity
\begin{align}
\label{eq:Leibniz}
\frac{d}{d\nu} (F\circ G)= \bigg(\frac{d}{d\nu} F\bigg)\circ G + F\circ\bigg( \frac{d}{d\nu} G\bigg).
\end{align}
Since it is filtration preserving, it extends to a map $ \frac{d}{d\nu}\colon \Hom_{\C_\nu}(X,Y)\to \Hom_{\C_\nu}(X,Y)$.
Let now $F\colon \C\to \D$ be a filtered functor. Then for any objects $X,Y$ in $\C$ there is a filtration-preserving map $F\colon \Hom_{\C}(X,Y)[\nu]\to \Hom_{\D}(F(X),F(Y))[\nu]$, and therefore induces a unique map $F_\nu\colon \Hom_{\C_\nu}(X,Y)\to \Hom_{\D_\nu}(F(X),F(Y))$and thus a functor $F_\nu\colon \C_\nu\to\D_\nu$. 

\item Let $\C$ be a preadditive category. We denote by $\C_\hbar$ the category with the same objects of $\C$ and as morphisms $\Hom_{\C_\hbar}(X,Y):= \Hom_\C(X,Y)[[\hbar]]$. It is easy to see that the filtration $\mathrm{F}^k\Hom_{\C_\hbar}(X,Y)= \hbar^k\Hom_\C(X,Y)[[\hbar]]$ turns $\C_\hbar$ into a complete and separated filtered category.

\end{enumerate}
\end{definition}
The following lemma will be helpful throughout the next sections.

\begin{lemma}
\label{Lem: Evareenough}
    Let $\C$ be a complete and separated category and let $f\in \Hom_{\C_\nu}(A,B)$, such that $\mathrm{ev}_s(f)=0$ for all $s\in \mathbb{K}$, then $f=0$. 
\end{lemma}

\begin{proof}
Note that the statement is clear for  $f\in \Hom_\C(A,B)[\nu]$, since the characteristic of $\mathbb{K}$ is $0$. Let us expand $f\in \Hom_{\C_\nu}(A,B)$ as 
    \begin{align*}
        f=\sum_{k=0}^\infty \nu^k f_k
    \end{align*}
where we use that $\Hom_{\C_\nu}(A,B)\subseteq \Hom_{\C}(A,B)[[\nu]]$.
Let now $N\in \mathbb{N}$, then by definition of $\Hom_{\C_\nu}(A,B)$,  there exists a $K\in \mathbb{N}$ such that 
    \begin{align*}
        [f]:=\sum_{k=0}^K \nu^k[f_k],
    \end{align*}
where $[f_k]$ denotes the equivalence class of $f_k$ in $\Hom_{\C}(A,B)/\mathrm{F}^N\Hom_{\C}(A,B)$. Then we know that the evaluations $\tilde{\mathrm{ev}_s}([f])=0\in \Hom_{\C}(A,B)/\mathrm{F}^N\Hom_{C}(A,B)$ and we can use again the statement for polynomials to deduce that $[f_k]=0\in \Hom_{\C}(A,B)/\mathrm{F}^N\Hom_{C}(A,B)$ for all $k\in \mathbb{N}$. Since $N$ was arbitrary, we conclude that $f_k\in \bigcap_{i=0}^\infty \mathrm{F}^i\Hom_\C(A,B)=\{0\}$, since $\C$ is separated.  
\end{proof}

\section{Drinfeld associators and the Grothendieck-Teichm\"uller semigroup}
\label{Sec: DA and GT}
\subsection{Drinfeld associators and deformation of categories \`a la Cartier}
We now briefly discuss Drinfeld associators and how to use them in order to deform Cartier categories.
\subsubsection{Definition and some properties}
\begin{definition}
\label{Def: DA}
A Drinfeld associator with coefficients in $\mathbb{K}$ is a pair $(\Phi,\lambda)$ consisting of a formal power series in two non-commuting variables $\Phi(A,B) \in \mathbb{K} \langle\!\langle A,B\rangle\! \rangle$ and $\lambda\in \mathbb{K}$  having the following properties
\begin{enumerate}
\item $\Phi = 1 + \mathcal{O}(\langle A,B\rangle^2)$.
\item $\Phi (A,B)^{-1} = \Phi (B,A)$.
\item  For any elements $\{A_{ij}\}_{ 1 \leqslant i,j \leqslant n }$ satisfying the infinitesimal braid relations
\begin{align}
[A_{ij}, A_{ji}] &=0 \qquad \text{for all } i,j \text{ with } \#\{i,j \} = 2  \\
[A_{ij} + A_{ik}, A_{ik}]&=0 \qquad \text{for all } i,j,k \text{ with } \#\{i,j,k \} = 3 \\
[A_{ij},A_{kl}] &=0 \qquad \text{for all } i,j,k,l \text{ with } \#\{i,j,k,l \} = 4
\end{align} 
the pentagon equation holds:
\begin{equation}
\Phi(A_{12}, A_{23} + A_{24}) \Phi (A_{13} + A_{23}, A_{34}) = \Phi (A_{23}, A_{34}) \Phi (A_{12} + A_{13}, A_{24} + A_{34}) \Phi (A_{12}, A_{23}).
\end{equation}
\item For any elements $A,B,C$ for which the sum $\Lambda = A+B+C$ satisfies $[\Lambda , A] = [\Lambda , B] = [\Lambda, C] = 0$, the hexagon equation holds:
\begin{equation}
e^{{\lambda}\Lambda} = e^{{\lambda} A}  \Phi(C,A)e^{{\lambda} C} \Phi(B,C) e^{{\lambda} B } \Phi(A,B).
\end{equation}

\item $\Phi$ is a grouplike element with respect to the comultiplication generated by $\Delta(A) = A \ten 1 + 1 \ten A$ and $\Delta(B) = B \ten 1 + 1 \ten B$.
\end{enumerate}
\end{definition}
We shall denote the set of all Drinfeld associators over $\mathbb{K}$ by $\mathrm{Ass}(\mathbb{K})$.
\begin{example}
The most famous example of a Drinfeld associator is the one arising from the monodromy of the Knizhnik-Zamolodchikov connections, see \cite{Dri90} and \cite{BRW} for more details. In particular, it is a complex Drinfeld associator with $\lambda = \frac{1}{2}$. Moreover, it was shown by V. Drinfeld that one can always assume the existence of a Drinfeld associator in any field of characteristic zero, see also \cite{{APS}} and references therein for more details.
\end{example}
We shall need the following result, which follows directly from the definition of Drinfeld associators. 
\begin{proposition}
Let $(\Phi,\lambda)$ be a Drinfeld associator and let $\chi \in \mathbb{K}$. Then $(\Phi_\chi,\chi\lambda)$ is a Drinfeld associator, with $\Phi_\chi(A,B)=\Phi(\chi A, \chi B)$.
\end{proposition}

\subsubsection{Cartier's deformation of categories}
Drinfeld associators can be used to construct deformations of symmetric Cartier categories. This is explained in the following result (where we slightly modify the original formulation by employing the notion of a quasi-trivial Cartier category), due P. Cartier\footnote{A recent generalisation of Cartier's result to the non-symmetric setting is given in \cite{ERSW}.} \cite{Car93}.
\begin{theorem}
\label{Thm: ApplDA}
Let $\C$ be a quasi-trivial Cartier category with infinitesimal braiding $t$ and let $(\Phi,\lambda)$ be a Drinfeld associator. Then we can form a new braided monoidal category $\C_{\Phi}$, where objects, morphisms, and tensor products are the same as $\C$,  and the associativity constraint and the braiding of $\C_{\Phi}$ are 
\begin{equation*}
\begin{split}
a^{\Phi}_{X,Y,Z} &= a_{X,Y,Z} \circ \Phi\big( t_{X,Y} \ten \id_Z, a^{-1}_{X,Y,Z} \circ (\id_X \ten  t_{Y,Z}) \circ a_{X,Y,Z}\big) \\
\sigma^{\Phi}_{X,Y} &= \sigma_{X,Y} \circ e^{\lambda t_{X,Y}}.
\end{split}
\end{equation*}
\end{theorem}
Note that if $\C$ is symmetric, then $\C_{\Phi}$ is quasi-symmetric. There is also the following result, see \cite{Sev16} for a proof of the first three statements, whence the fourth follows from an easy check.
\begin{theorem}
\label{prop: Severa}
Let $\B,\C$ be two symmetric quasi-trivial Cartier categories,  $(\Phi,\lambda)$ be a Drinfeld associator, and $M$ be a comonoid object in $\B_{\Phi}$. Then 
\begin{enumerate}
\item If $M$ is infinitesimally cocommutative in $\B$,  then it is cocommutative in $\B_{\Phi}$.
\item If $(F,F^2,F^0): \B \to \C$ is an infinitesimally braided comonoidal functor, then  $(F_{\Phi}, F^2_{\Phi}, F^0_{\Phi}):=(F,F^2,F^0): \B_{\Phi}\to \C_{\Phi}$ is a braided comonoidal functor.
\item If $(F,F^2,F^0): \B \to \C$ is $M$-adapted then $(F_{\Phi}, F^2_{\Phi}, F^0_{\Phi}): \B_{\Phi}\to \C_{\Phi}$ is $M$-adapted.
\item Let $n\colon F\implies G$ be a comonoidal natural transformation between two infinitesimally braided comonoidal functors $F,G\colon\B\to \C$. Then we have a comonoidal natural transformation of braided comonoidal functors $n_{\Phi}:=n\colon F_{\Phi}\implies G_{\Phi}$. 
\end{enumerate}
Moreover, all the stated results hold for infinitesimally braided monoidal functors and monoids by passing to the opposite-reversed setting.
\end{theorem}
Combining the two last theorems, we thus obtain that any Drinfeld associator $(\Phi,\lambda)$ induces a functor from the category of symmetric quasi-trivial Cartier categories to the category of quasi-symmetric categories. We will show in the next section that such a functor is even an equivalence of categories whenever $\lambda\neq 0$, see Theorem \ref{Thm:Quasisym=infbraided}.

\subsection{The Grothendieck-Teichm\"uller semigroup}
In this section, we introduce the Grothendieck-Teichm\"uller group $\overline{GT}(\mathbb{K})$ associated to a field $\mathbb{K}$, which goes back to V. Drinfeld \cite{Dri90}. There are several ways to define $\overline{GT}(\mathbb{K})$: as a group of elements defining a braided monoidal structure on a quasi-symmetric category; as pairs of elements in $\mathbb{K}^\ast \times \widehat{\mathrm{FG}}(x,y)(\mathbb{K})$ satisfying certain conditions; or by the group of automorphisms of a certain operad satisfying certain properties \cite{BNat}. We shall employ the second approach and refer to \cite{Mer} for a treatment of all the other cited ones. For further details, we refer the reader to \cite{BNat,Sch,Bro}.

\subsubsection{The pro-unipotent completion of a group and the definition of $\overline{GT}(\mathbb{K})$}
In order to define the Grothendieck-Teichm\"uller group of a field $\mathbb{K}$, we need to introduce the pro-unipotent completion (also known as the Malcev completion \cite{Mal}) of some groups. A detailed treatment of the pro-unipotent completion of a group is out of the scope of this paper; the reader can find a complete treatment in \cite{Fre}. In what follows, we shall mainly refer to \cite[Sect. 2]{Mer}.
\begin{definition}
Let $G$ be a finitely generated group and $\mathbb{K}$ be a field. The pro-unipotent completion of $G$ is the group $\widehat{G}(\mathbb{K}) := \mathrm{Grp}(\widehat{\mathbb{K}[G]})$, i.e., the set of grouplike elements of the completion of the group algebra of $G$ with respect to the filtration induced by the ideal $\ker(\varepsilon)$.
\end{definition}
We now introduce the two main examples of our interest.
\begin{example}
Let $\mathbb{K}$ be a field. 
\begin{enumerate}
\item 
Denote by $\mathrm{FG}(x_1,\ldots, x_n)$ the free group generated by $x_1,\ldots, x_n$. Then it is well-known that the association $x_1 \mapsto e^{A_1}, \ldots, x_n \mapsto e^{A_n}$ defines an isomorphism $$\widehat{\mathrm{FG}}(x_1,\ldots, x_n)(\mathbb{K}) \to \mathrm{Grp}(\mathbb{K}\langle\!\langle A_1,\ldots, A_n\rangle\!\rangle)$$ where 
$\mathrm{Grp}(\mathbb{K}\langle\!\langle A_1,\ldots, A_n\rangle\!\rangle)$ denotes the set of grouplike elements of the completed free associative unital algebra generated by $A_1,\ldots, A_n$. In particular, any element in $\widehat{\mathrm{FG}}(x,y)(\mathbb{K})$ can be thought as a Lie series $f=f(x,y) \cong f(e^A,e^B)$, where $A$ and $B$ are the generators of  $\mathbb{K}\langle\!\langle A,B\rangle\!\rangle$. 

\item Recall (see e.g. \cite{Lee}) that for any non-negative integer $n$, the pure braid group on $n$ strands is the group $\mathrm{PB}_n$ generated by elements $\{x_{ij}\}_{1 \leq i<j \leq n}$ with relations 
\begin{align}
x_{ij} x_{rs} &= x_{rs} x_{ij} \quad \text{for } r<s<i<j \ \text{or } \ i<r<s<j\\
x_{ji}x_{ir}x_{rj} &= x_{ir}x_{rj}x_{ji} = x_{rj}x_{ji}x_{ir} \quad \text{for } r<i<j \\
x_{rs}(x_{jr}x_{ji}x_{js}) &= (x_{jr}x_{ji}x_{js})x_{rs} \quad \text{for } r<i<s<j. 
\end{align}
Then, following \cite[Sect. 2.5]{Mer} and \cite[Sect. 3]{Bro} we have that $\widehat{\mathrm{PB}_n}(\mathbb{K})$ is generated by elements $\{A_{ij}\}_{1 \leq i < j \leq n}$ subject to the so-called infinitesimal pure braid relations
\begin{align}
[A_{ij},A_{kl}]&= 1 \qquad \text{for } i<j<k<l \\
[A_{il},A_{jk}]&= 1 \qquad \text{for } i<j<k<l \\
[A_{kl}A_{ik}A_{kl}^{-1},A_{jl}] &= 1\qquad \text{for } i<j<k<l \\
A_{ik}A_{jk}A_{ij} &= A_{jk}A_{ij}A_{ik} = A_{ij}A_{ik}A_{jk}.
\end{align}
\end{enumerate}
\end{example}
We can now introduce the following
\begin{definition}
Let $\mathbb{K}$ be a field. The Grothendieck-Teichm\"uller group of $\mathbb{K}$, denoted by $\overline{GT}(\mathbb{K})$, consists of the set of all pairs $(\lambda, f) \in \mathbb{K}^\ast \times \widehat{\mathrm{FG}}(x,y)(\mathbb{K})$ satisfying the following conditions:
\begin{align}
f(x,y) &= f(y,x)^{-1} \\
f(x_3,x_1) x_3^{\frac{\lambda-1}{2}} f(x_2,x_3) x_2^{\frac{\lambda-1}{2}} f(x_1,x_2) x_1^{\frac{\lambda-1}{2}} &=1 \quad \text{for} \quad x_1x_2x_3=1 \\
f(x_{12}, x_{23}x_{24}) f(x_{13}x_{23}, x_{34}) =f(x_{23},x_{34})f(x_{12}&x_{13}, x_{24}x_{34}) f(x_{12},x_{23}) \quad \text{in} \quad \widehat{\mathrm{PB}_4}(\mathbb{K}).
\end{align}
\end{definition}
The group structure of $\overline{GT}(\mathbb{K})$ is defined by $(\lambda_1,f_1) \cdot (\lambda_2,f_2) = (\lambda,f)$, where $\lambda = \lambda_1 \lambda_2$ and 
\[ f(x,y) = f_1 \big(f_2(x,y) x^{\lambda_2} f_2(x,y)^{-1}, y^{\lambda_2}\big) f_2(x,y)=f_2(x,y) f_1 \big( x^{\lambda_2}, f_2(x,y)^{-1}y^{\lambda_2}f_2(x,y)\big).\]
Note that we can extend $\overline{GT}(\mathbb{K})$ to a semigroup by considering elements in $\mathbb{K} \times \widehat{\mathrm{FG}}(x,y)(\mathbb{K})$.

\subsubsection{On the relationship between $\overline{GT}(\mathbb{K})$ and $\mathrm{Ass}(\mathbb{K})$}
In this section we exhibit some of the main results of V. Drinfeld \cite{Dri90} regarding the semigroup $\overline{GT}(\mathbb{K})$ in relationship with Drinfeld associators. Here, the assumption that the field has characteristic zero is crucial.
\begin{theorem}
\label{Thm: DrinfeldI} \label{Thm: nicecurve} \label{Cor: Firstordervanishes} \label{theorem-action-gt}
Let $\mathbb{K}$ be a field. 
\begin{enumerate}
\item There is a transitive semigroup action
\begin{align*}
\overline{GT}(\mathbb{K}) \times \mathrm{Ass}(\mathbb{K}) &\to \mathrm{Ass}(\mathbb{K}) \\
\big((\chi,f) , (\Phi,\lambda)\big) &\mapsto(\chi,f) \cdot (\Phi,\lambda) :=
(f\bullet_\lambda \Phi,\chi\lambda)
\end{align*}
where 
\begin{align*}
f\bullet_\lambda \Phi(A,B)&=f(\Phi(A,B)e^{2\lambda A}\Phi(A,B)^{-1}, e^{2\lambda B})\Phi(A,B)\\
&=\Phi(A,B)f(e^{2\lambda A}, \Phi(A,B)^{-1}e^{2\lambda B}\Phi(A,B)).
\end{align*}
\item For any $(0,f_0)\in \overline{GT}(\mathbb{K})$ there exists a unique algebraic morphism of semigroups $g : \mathbb{K}\to \overline{GT}(\mathbb{K})$, such that 
$g(0)=(0,f_0)$ and $g(\lambda)=(\lambda, \omega)$ for some $\omega$.    
\item Let $(\Phi,\lambda)$ be a Drinfeld associator with $\lambda\neq 0$. Then there exists a unique algebraic morphism of semigroups $g\colon \mathbb{K}\to \overline{GT}(\mathbb{K})$, such that $g(\chi)\cdot(\Phi,\lambda)=(\Phi_\chi,\chi\lambda)$.
\item The algebraic morphism of semigroups $g\colon \mathbb{K}\to \overline{GT}(\mathbb{K})$ from statement $(iii)$ does not have a linear term.
\end{enumerate}
\end{theorem}
We refer to \cite{Dri90} for the proof of the first three statements. Before proving statement $(iv)$, let us shortly explain what we mean by algebraic morphism of semigroups. Consider the free algebra in two generators $A,B$, which has a canonical filtration 
\begin{align*}
\mathrm{F}^N\mathbb{K}\langle A,B\rangle=
\bigoplus_{i=N}^\infty\mathbb{K}\langle A,B\rangle^i,
\end{align*}
where $\mathbb{K}\langle A,B\rangle^i$ is the span of elements which are generated by exactly $i$ elements. 
It is clear that the completion with respect to this filtration is $\mathbb{K}\langle\!\langle A,B\rangle\! \rangle$, i.e. the set of all formal power series in two non-commuting variables. 
For a fixed indeterminate $\nu$ we set
\begin{align*}
\mathbb{K}\langle\!\langle A,B\rangle\! \rangle(\nu)= \prod_{i\in\mathbb{N}} (\mathbb{K}\langle A,B\rangle^i[\nu]).
\end{align*}
Note that this is nothing else than the completion of $\mathbb{K}\langle\!\langle A,B\rangle\! \rangle[\nu]$ with respect to the induced filtration as already used in Definition \ref{Def: AuxCats} (iii).
For any $s \in \mathbb{K}$ there is an evaluation morphism $\ev_s\colon \mathbb{K}\langle\!\langle A,B\rangle\! \rangle(\nu)\to \mathbb{K}\langle\!\langle A,B\rangle\! \rangle$.
A morphism $g\colon \mathbb{K}\to \overline{GT}(\mathbb{K})$ is said to be algebraic if there exists an element $F\in \mathbb{K}\langle\!\langle A,B\rangle\! \rangle(\nu)$ such that $g(\chi)=(\chi,\ev_\chi(F))$
for $\chi\in \mathbb{K}$. \\
Next, we take a closer look at the action of elements of $\overline{GT}(\mathbb{K})$ on $\mathrm{Ass}(\mathbb{K})$ explained in statement $(i)$ of the above theorem. 
Let  $G(A,B)\in \mathbb{K}\langle\!\langle A,B\rangle\!\rangle$ be a grouplike element. Then the algebra morphism $\mathbb{K}\langle A,B\rangle\to \mathbb{K}\langle\!\langle A,B\rangle\!\rangle$ defined on generators by $A\mapsto A$ and \ $B\mapsto G(A,B)^{-1}BG(A,B)$ is compatible with filtrations, and therefore it extends to an algebra isomorphism $T_{G}\colon \mathbb{K}\langle\!\langle A,B\rangle\!\rangle\to \mathbb{K}\langle\!\langle A,B\rangle\!\rangle$. 
Furthermore, since $G$ is grouplike, we have that $T_G$ is also compatible with the antipodes and the comultiplications, i.e. it is a morphism of Hopf algebras, and thus induces, by considering the restriction to grouplike elements, a group morphism. 
Moreover, it satisfies $T_{G_1}\circ T_{G_2}=T_{G_1T_{G_1}(G_2)}$, which makes the map $(G_1,G_2)\mapsto G_1\diamond G_2:= G_1T_{G_1}(G_2)$ associative. 
One can show that $1\in \mathbb{K}\langle\!\langle A,B\rangle\!\rangle$ is a unit for this product and that for every grouplike element $G$ its inverse with respect to $\diamond $  is given by 
$\mathrm{inv}_\diamond(G)=T_G^{-1}(G^{-1})$. This turns grouplike elements into a group with respect to $\diamond$. Let now $(\Phi,\frac{1}{2})$ be a Drinfeld associator and $(\chi,f)\in \overline{GT}(\mathbb{K})$. Then we have 
\begin{align*}
(\chi,f)\cdot \bigg(\Phi,\frac{1}{2}\bigg)=\bigg(\Phi\diamond f,\frac{\chi}{2}\bigg).
\end{align*}
We now show part $(iv)$ of the above theorem.
\begin{proof}
Set $g(\chi)=(\chi,f_\chi)$.
By the discussion above, we have $f_\chi = \mathrm{inv}_\diamond(\Phi)\diamond \Phi_\chi$. In view of property $(i)$ of Definition \ref{Def: DA}, we know that 
$\Phi_\chi$ does not have a linear term in $\chi$, and by the explicit formula of $\diamond$ this implies that also $f_\chi$ does not have it. In particular, we get 
\begin{align*}
\frac{d}{d\chi}\bigg|_{\chi=0}f_\chi=0.
\end{align*}
\end{proof}

\subsubsection{Deformation of quasi-symmetric categories}
The main property of $\overline{GT}(\mathbb{K})$ we are interested in is described by the following
\begin{theorem}
\label{theorem-gt-deformation}
Let $\C$ be a quasi-symmetric braided monoidal category and let $(\lambda,f) \in \overline{GT}(\mathbb{K})$. Then the morphisms
\begin{align}
a^{\lambda,f}_{X,Y,Z} &= a_{X,Y,Z} \circ f\big(\sigma^2_{X,Y}\ten \id_Z, a^{-1}_{X,Y,Z}\circ(\id_X \ten\sigma^2_{Y,Z})\circ a_{X,Y,Z}\big) \label{eq:def-GT-ass-constr} \\
\sigma^{\lambda,f}_{X,Y} &= \sigma_{X,Y} \circ (\sigma^2_{X,Y})^{\frac{\lambda-1}{2}}=\sigma_{X,Y}\circ \sum_{k=0}^\infty\frac{(\lambda-1)^k}{2^kk!}\log(\sigma_{X,Y}^2)^k \label{eq:def-GT-braiding}
\end{align}
define a new braided monoidal structure on $\C$, which we denote by $\C_{(\lambda,f)}$.
\end{theorem}
Note that if $\C$ is symmetric then $\C_{(\lambda,f)}=\C$. Furthermore, $\C_{(0,f)}$ is always a symmetric monoidal category. The previous result extends to functors and natural transformations via the following
\begin{theorem}
\label{Thm: applGT}
Let $\B,\C$ be two quasi-symmetric braided monoidal categories, $F,G: \B \to \C$ be two braided comonoidal functors together with a natural comonoidal transformation $n\colon F\implies G$, and let $(\lambda,f) \in \overline{GT}(\mathbb{K})$. 
Then there are braided comonoidal functors $F_{(\lambda,f)}, G_{(\lambda,f)}: \B_{(\lambda,f)} \to  \C_{(\lambda,f)}$ which are defined as $F,G$ on objects and on morphisms, and $n_{(\lambda,f)}\colon F_{(\lambda,f)}\implies G_{(\lambda,f)}$ is a natural comonoidal transformation. Furthermore, if $M\in \B$ is a cocommutative comonoid and $F$ is $M$-adapted, then $M$ is a cocommutative comonoid in $\B_{(\lambda,f)}$ and  $F_{(\lambda,f)}$ is $M$-adapted.\\
Moreover, all the stated results hold in the corresponding framework of monoidal functors and monoids.

\end{theorem}

\begin{proof}
 Note that for every odd power of $\sigma_{X,Y}$, the diagram 
    \begin{center}
        \begin{tikzcd}
F(X\ten Y) \arrow[r, "{F^2(X,Y)}"] \arrow[d, "{F((\sigma_{X,Y})^{2k+1})}"'] & F(X)\ten F(Y) \arrow[d, "{(\sigma_{F(X),F(Y)})^{2k+1}}"] \\
F(Y\ten X) \arrow[r, "{F^2(Y,X)}"]                                          & F(Y)\ten F(X)                                           
\end{tikzcd}
    \end{center}
commutes, which also implies the commutativity for all the sums. 
Since the functors respect the filtrations, this also holds for the infinite sums defining the braiding. Similar arguments hold for the compatibilities with the new associativity constraints. 
The statements for the natural transformations is trivial since the tensor product itself is not changed. 
Let now $M$ be a cocommutative comonoid in $\B$, then 
\begin{align*}
    \sigma^{\lambda,f}_{M,M}\circ \Delta=\sigma_{M,M}\circ\Big(\sum_{k=0}^\infty\frac{(\lambda-1)^k}{2^kk!}\log(\sigma_{M,M}^2)^k\Big)\circ \Delta.
\end{align*}
Note that $\log(\sigma_{M,M}^2)=\sum_{k=1}^\infty\frac{(-1)^{k+1}}{k}(\sigma_{M,M}^2-\id)^k$ and that $(\sigma_{M,M}^2-\id)\circ\Delta=0$. This implies that $\sigma_{M,M}^{\lambda,f}\circ \Delta=\Delta$. Again, similar arguments can be applied to show that $M$ is coassociative. 
\end{proof}
The last theorem implies, in particular, that for every $\lambda$ we can use Theorem \ref{theorem-gt-deformation} to deform a quasi-symmetric monoidal category to obtain a polynomial curve of associativity constraints matching braidings.

More specifically, let $\C$ be a $\mathbb{K}$-linear quasi-symmetric category and consider the category $\C_\nu$ built as in section \ref{section-aux-categories} (which is again quasi-symmetric). For any algebraic morphism of semigroups $g : \mathbb{K} \to \overline{GT}(\mathbb{K})$ with $g(\nu)=(\nu, f(\nu))$, we can endow $\C_\nu$ with a braided monoidal structure via
\begin{align}
a^{\nu,g}_{X,Y,Z} &= a_{X,Y,Z} \circ f(\nu)(\sigma^2_{X,Y}\ten \id_Z, a^{-1}_{X,Y,Z}\circ(\id_X\ten\sigma^2_{Y,Z})\circ a_{X,Y,Z}) \label{eq:ass-constraint-nu-g}\\
\sigma^{\nu,g}_{X,Y} &= \sigma_{X,Y} \circ (\sigma^2_{X,Y})^{\frac{\nu-1}{2}}.\label{eq:braiding-nu-g}
\end{align}
We shall denote it by $\C_{\nu,g}$. 
Note that for every $s\in \mathbb{K}$ the functor $\ev_s\colon \C_{\nu,g} \to \C_{g(s)}$ is strongly monoidal,
where $\C_{g(s)}$ is the braided monoidal category obtained from $\C$ by using the element $g(s)\in \overline{GT}(\mathbb{K})$ as explained in section \ref{section-aux-categories}.

\begin{theorem}
\label{Thm: curveassociators}
Let $\B,\C$ be quasi-symmetric categories and let $F\colon \B\to \C$ be a filtered (strongly) comonoidal functor. For any algebraic morphism of semigroups $g\colon \mathbb{K}\to \overline{GT}(\mathbb{K})$, $F$ induces a (strongly) comonoidal functor $F_{\nu,g}\colon \B_{\nu,g}\to \C_{\nu,g},$
which coincides with $F$ on objects and with $F_\nu$ on morphisms. Furthermore, let $G\colon \B\to\C$ be another braided comonoidal functor, and let $n\colon F\implies G$ be a comonoidal natural transformation. Then there is an induced comonoidal natural transformation $n_{\nu,g}\colon F_{\nu,g}\implies G_{\nu,g}$.
Furthermore, for every $s\in \mathbb{K}$, we have the commutative diagram of functors 
\begin{center}
\begin{tikzcd}
{\B_{\nu,g}} \arrow[r, "{F_{\nu,g}}"] \arrow[d, "\ev_s"'] & {\C_{\nu,g}} \arrow[d, "\ev_s"] \\
\B_{g(s)} \arrow[r, "F_{g(s)}"]                           & \C_{g(s)}                     
\end{tikzcd}.
\end{center}
Moreover, if $M\in \B$ is a cocommutative comonoid, then $M$ is a cocommutative comonoid in $\B_{\nu,g}$. Finally, if $F$ is $M$-adapted, then $F_{\nu,g}$ is $M$-adapted. \\
The results stated above also hold in the setting of monoidal functors, monoids, and coadapted functors.
\end{theorem}
\begin{proof}
The proof consists only of applying the evaluation functors: by Theorem \ref{Thm: applGT}, we know that the statement is true for every evaluation and by using Lemma \ref{Lem: Evareenough} we get the claim. 
\end{proof}

\subsection{Construction of infinitesimal braidings}
\label{SubSec: Infbraid}
Let $\C$ be a quasi-symmetric category, $(\Phi,\frac{1}{2})$ be a Drinfeld associator,  and consider the algebraic morphism of semigroups $g\colon \mathbb{K}\to \overline{GT}(\mathbb{K})$ from Theorem \ref{Thm: nicecurve}. Consider the braided monoidal category $\C_{\nu,g}$ built in the previous section. We define the natural transformation\footnote{The naturality of $t^\sigma$ follows from the naturality of $\sigma^{\nu,g}$} $t^\sigma\colon \ten \implies \ten$ by
\begin{align}
\label{eq:inf-braiding-GT}
t^\sigma_{X,Y}:=\ev_0\bigg(\frac{d}{d\nu}(\sigma^{\nu,g}_{X,Y})^2\bigg)\in \Hom_\C(X\ten Y,X\ten Y).
\end{align}
\begin{theorem}
\label{thm: qasisygivesinfitesimal braiding}
The natural transformation $t^\sigma$ defined in Equation \eqref{eq:inf-braiding-GT} turns $\C_{g(0)}$ into a symmetric quasi-trivial Cartier category.
If $M$ is a cocommutative comonoid in $\C$, then $M$ is an infinitesimally cocommutative comonoid in $\C_{g(0)}$. Furthermore, if $\B$ is another quasi-symmetric category and $(F,F^2,F^0)\colon \B\to \C$ is a braided comonoidal functor, then $(F_{g(0)},F^2_{g(0)},F^0_{g(0)}):=(F,F^2,F^0)\colon \B_{g(0)}\to \C_{g(0)}$ is an infinitesimally braided comonoidal functor. Moreover, if $G\colon \B\to \C$ is another braided comonoidal functor and $n\colon F\implies G$ is a comonoidal natural transformation, then $n_{g(0)}:=n\colon F_{g(0)}\implies G_{g(0)}$ is a comonoidal natural transformation. \\
Moreover, the stated results also hold in the framework of monoids and monoidal functors.
\end{theorem}
\begin{proof}
Recall that, since the category $\C_{\nu,g}$ is quasi-symmetric, we get that $(\sigma^{\nu,g}_{X,Y})^2 -\id_{X \ten Y} \in \mathrm{F}^1\Hom_{\C_{\nu,g}}(X\ten Y,X\ten Y)$, which implies that $\frac{d}{d\nu}(\sigma^{\nu,g}_{X,Y})^2\in \mathrm{F}^1\Hom_{\C_{\nu,g}}(X\ten Y,X\ten Y)$, and since the evaluation functors are filtered we get that $t^\sigma_{X,Y}\in \mathrm{F}^1\Hom(X\ten Y,X\ten Y)$. Let us denote by $\sigma^0$ the braiding in $\C_{g(0)}$ which is given by $\sigma^0_{X,Y}=\ev_0( (\sigma^{\nu,g}\nu)_{X,Y})$. Then, using the Leibniz identity \eqref{eq:Leibniz} and the fact that $\sigma$ is quasi-symmetric, we have that
\begin{align*}
\sigma_{X,Y}^0\circ t^\sigma_{X,Y} &=\sigma_{X,Y}^0 \circ \ev_0\bigg( \frac{d}{d\nu}(\sigma^{\nu,g}_{X,Y})^2\bigg) = \ev_0\bigg( \sigma_{X,Y}^{\nu,g} \circ \frac{d}{d\nu}(\sigma^{\nu,g}_{X,Y})^2\bigg) \\
&=\ev_{0}\bigg(\frac{d}{d\nu}\big(\sigma^{\nu,g}_{X,Y}\circ (\sigma^{\nu,g}_{X,Y})^2  \big)  \bigg)
- \ev_{0}\bigg(\bigg(\frac{d}{d\nu}\sigma^{\nu,g}_{X,Y}\bigg)\circ (\sigma^{\nu,g}_{X,Y})^2\bigg)\\&
=\ev_{0}\bigg(\frac{d}{d\nu}\big((\sigma^{\nu,g}_{Y,X})^2\circ \sigma^{\nu,g}_{X,Y}\big)    \bigg)
-\ev_{0}\bigg(\frac{d}{d\nu}(\sigma^{\nu,g}_{X,Y})\bigg)\\&=
\ev_0 \bigg(\frac{d}{d\nu}(\sigma^{\nu,g}_{Y,X})^2\bigg)\circ \sigma_{X,Y}^0 + \ev_{0}\bigg(\frac{d}{d\nu}\sigma^{\nu,g}_{X,Y}\bigg)-\ev_{0}\bigg(\frac{d}{d\nu}\sigma^{\nu,g}_{X,Y}\bigg)\\&
=t^\sigma_{Y,X}\circ \sigma^0_{X,Y}.
\end{align*}
Using the hexagon axiom, we get 
\begin{align*}
(\sigma^{\nu,g}_{X\ten Y,Z})^2= (a_{X,Y,Z}^{\nu,g})^{-1}\circ (\id\ten \sigma^{\nu,g}_{Z,Y})\circ a_{X,Z,Y}^{\nu,g}\circ ((\sigma_{X,Z}^{\nu,g})^2\ten \id)\circ (a_{X,Z,Y}^{\nu,g})^{-1}\circ(\id\ten \sigma_{Y,Z}^{\nu,g})\circ a_{X,Y,Z}^{\nu,g}.
\end{align*}
Deriving this equation and evaluating at $\nu=0$ leads to the condition \eqref{eq:pre-cartier-two}, since $\sigma^0$ is symmetric and we may assume that the derivation at $0$ of the associativity constraints is $0$ and thus also for its inverse, because of part $(iv)$ of Theorem \ref{Cor: Firstordervanishes}. 

Next, given a cocommutative comonoid $(M,\Delta,\varepsilon)$ in $\C$, by Theorem \ref{Thm: curveassociators} we obtain that $M$ is also a cocommutative comonoid in $\C_{\nu,g}$ and thus we get 
\begin{align*}
    t^\sigma_{M,M}\circ \Delta= \ev_0\bigg(\frac{d}{d\nu}(\sigma^{\nu,g}_{M,M})^2 \circ \Delta\bigg)=\ev_0\bigg(\frac{d}{d\nu}\Delta\bigg)=0.
\end{align*}
Clearly, the analogue result works for commutative monoids. Finally, let $F\colon \B\to \C$ be a braided comonoidal functor. Then, deriving and evaluating at $0$ the following commutative diagram
\begin{center}
\begin{tikzcd}
{F_{\nu,g}(X\ten Y)} \arrow[d, "{F_{\nu,g}(\sigma^{\nu,g}_{X,Y})}"'] \arrow[rr, "{F_{\nu,g}^2(X,Y)}"] &  & {F_{\nu,g}(X)\ten F_{\nu,g}( Y)} \arrow[d, "{\sigma^{\nu,g}_{F_{\nu,g}(X), F_{\nu,g}(Y)}}"] \\
{F_{\nu,g}(Y\ten X)} \arrow[rr, "{F_{\nu,g}^2(Y,X)}"]                                         &  & {F_{\nu,g}(X\ten Y)}                                                              
\end{tikzcd}
\end{center} 
gives the claim of the second part of the statement. The proof of the remaining statements is straightforward.
\end{proof}
The previous result seemed to be known already to the authors of \cite{EK2} (see Proposition 2.8 therein). However, we could not find the precise statement and proof in the literature. 

The next theorem is the main result of this section and provides an equivalence of categories 
\begin{align*}
\{\text{quasi-symmetric categories} \}& \longleftrightarrow	\{\text{symmetric quasi-trivial Cartier categories} \}
\end{align*}
described by the following diagrams
\begin{equation*}
\begin{tikzcd}
{(\C,\sigma)} \arrow[rr, "\text{Thm. } \ref{thm: qasisygivesinfitesimal braiding} ", shift left] &  & {(\C_{g(0)},t^\sigma)} \arrow[ll, "\text{Thm. } \ref{Thm:Quasisym=infbraided} ", shift left] & {(\B, a^\Phi, \sigma^\Phi)} \arrow[rr, "\text{Thm. } \ref{Thm:Quasisym=infbraided} ", shift left] &  & {(\B,t)} \arrow[ll, "\text{Thm. } \ref{Thm: ApplDA} ", shift left]
\end{tikzcd}
\end{equation*}
which, as usual, depend on the choice of a Drinfeld associator $\big(\Phi,\frac{1}{2}\big)$.
\begin{theorem}
\label{Thm:Quasisym=infbraided}
Let $(\Phi, \frac{1}{2})$ be a Drinfeld associator and let $g:\K \to \overline{GT}(\mathbb{K})$ be the corresponding unique morphism of semigroups obtained by part $(iii)$ of Theorem \ref{Thm: nicecurve}.
\begin{enumerate}
\item Let $\B$ be a symmetric quasi-trivial Cartier category. Then  $(\B_{\Phi})_{g(0)}=\B$.
\item Let $\C$ be a quasi-symmetric category. Then $(\C_{g(0)})_{\Phi}=\C$.
\end{enumerate} 
\end{theorem}

\begin{proof}
$(i):$ Let $(\B,t)$ be symmetric quasi-trivial Cartier and consider the deformed category $\B_\Phi$ obtained by Theorem \ref{Thm: ApplDA}, whose braiding is $\sigma^\Phi_{X,Y}=\sigma_{X,Y}\circ e^{\frac{1}{2}t_{X,Y}}$. The fact that $\B$ is symmetric together with Equation \eqref{eq:cartier-category} implies $(\sigma^\Phi_{X,Y})^2=e^{t_{X,Y}}$. Let $\B_{\Phi,\nu,g}$ be the curve of braided monoidal categories whose braiding is, recalling Equation \eqref{eq:braiding-nu-g}, given by $\sigma^{\Phi,\nu,g}_{X,Y}=\sigma^\Phi_{X,Y}\circ ((\sigma_{X,Y}^\Phi)^2)^{\frac{(\nu-1)}{2}}=\sigma_{X,Y} \circ e^{\frac{\nu}{2}t_{X,Y}}$. Hence, in particular we have $(\sigma^{\Phi,\nu,g}_{X,Y})^2=e^{\nu t_{X,Y}}$ and therefore we get that $\sigma^{\Phi,g(0)}=\sigma$ and $t^{\sigma^\Phi}=t$. Moreover, recalling Equation \eqref{eq:ass-constraint-nu-g}, part $(i)$ of Theorem \ref{Thm: DrinfeldI}, and that $g(0)=(0,f)$, we have that the associativity constraint is given by 
\begin{align*}
a^{\Phi,g(0)}_{X,Y,Z} &= a_{X,Y,Z} \circ \Phi\big(t_{X,Y} \ten \id_Z, a^{-1}_{X,Y,Z} \circ (\id_X \ten t_{Y,Z} \circ a_{X,Y,Z})\big)\\
& \qquad \qquad \qquad \circ f\big(e^{t_{X,Y} }\ten \id_Z , (a^\Phi_{X,Y,Z})^{-1} \circ (\id_X \ten e^{t_{Y,Z}}) \circ a^{\Phi}_{X,Y,Z} \big)\\
&= a_{X,Y,Z} \circ (f \bullet_{\frac12} \Phi) (t_{X,Y} \ten \id_Z, a^{-1}_{X,Y,Z} \circ (\id_X \ten t_{Y,Z}) \circ a_{X,Y,Z}).
\end{align*}
To end the proof it suffices to note that, by Theorem \ref{Thm: DrinfeldI}, $g$ and $(\Phi,\frac{1}{2})$ were chosen accordingly in such a way that $f\bullet_{\frac 12}\Phi=1= \Phi(0,0)$. \\ $(ii):$ Let $(\C,\sigma)$ be a quasi-symmetric category and consider the symmetric Cartier category $\C_{g(0)}$ with infinitesimal braiding as in Equation \eqref{eq:inf-braiding-GT}. An immediate computation shows that $t^\sigma_{X,Y}=\log(\sigma^2_{X,Y})$. Therefore, the braiding of the deformed category $(\C_{g(0)})_\Phi$ is given by $\sigma^{g(0),\Phi}_{X,Y}=\sigma^{g(0)}_{X,Y}\circ e^{\frac{1}{2}\log(\sigma_{X,Y}^2)}= \sigma_{X,Y} \circ e^{-\frac{1}{2}\log(\sigma_{X,Y}^2)} \circ e^{\frac{1}{2}\log(\sigma_{X,Y}^2)} =\sigma_{X,Y}$.
Next, the associativity constraint is given (using again $g(0)=(0,f)$) by 
\begin{align*}
(a^{g(0),\Phi}_{X,Y,Z}) &= a^{g(0)}_{X,Y,Z} \circ \Phi\big(\log(\sigma^2_{X,Y} \ten \id_Z), (a^{g(0)}_{X,Y,Z})^{-1} \circ (\id_X \ten \log(\sigma^2_{Y,Z}))\circ a_{X,Y,Z}^{g(0)} \big)\\
&= a_{X,Y,Z} \circ f\big( \sigma^2_{X,Y} \ten \id_Z, a^{-1}_{X,Y,Z} \circ (\id_X \ten \sigma^2_{Y,Z}) \circ a_{X,Y,Z}\big) \\
& \quad \circ \Phi\big(\log(\sigma^2_{X,Y} \ten \id_Z), (a^{g(0)}_{X,Y,Z})^{-1} \circ (\id_X \ten \log(\sigma^2_{Y,Z}))\circ a_{X,Y,Z}^{g(0)} \big) \\
&= a_{X,Y,Z} \circ f\big( \sigma^2_{X,Y} \ten \id_Z, a^{-1}_{X,Y,Z} \circ (\id_X \ten \sigma^2_{Y,Z}) \circ a_{X,Y,Z}\big) \\
&\quad \circ \Phi\bigg(\log(\sigma^2_{X,Y} \ten \id_Z), \big(f\big( \sigma^2_{X,Y} \ten \id_Z, a^{-1}_{X,Y,Z} \circ (\id_X \ten \sigma^2_{Y,Z}) \circ a_{X,Y,Z}\big)\big)^{-1} \circ a^{-1}_{X,Y,Z}\circ\\
& \quad \circ (\id_X \ten \log(\sigma^2_{Y,Z})) \circ a_{X,Y,Z} \circ f\big( \sigma^2_{X,Y} \ten \id_Z, a^{-1}_{X,Y,Z} \circ (\id_X \ten \sigma^2_{Y,Z}) \circ a_{X,Y,Z}\big)\bigg) \\
&= a_{X,Y,Z} \circ f(e^A,e^B) \circ \Phi\big(A, f(e^A,e^B)^{-1} B f(e^A,e^B)\big)
\end{align*}
where we set $A= \log(\sigma^2_{X,Y}) \ten \id_Z$ and $B = a^{-1}_{X,Y,Z} \circ (\id_X \ten \log(\sigma^2_{Y,Z})) \circ a_{X,Y,Z}$.
Recall from the discussion after the statement of Corollary \ref{Cor: Firstordervanishes} that on grouplike elements in $\mathbb{K}\langle\!\langle A,B\rangle\!\rangle$ the map $ G_1(A,B)\diamond G_2(A,B)=G_1(A,B)G_2(A,G_1(a,b)^{-1}BG_1(A,B))$
defines the structure of a group with $1\in \mathbb{K}\langle\!\langle A,B\rangle\!\rangle$ being the unit. Let us compute 
\begin{align*}
\Phi(A,B)\diamond f(e^A,e^B)=\Phi(A,B)f(e^A,\Phi(A,B)^{-1}e^B\Phi(A,B))=(f\bullet_{\frac{1}{2}}\Phi)(A,B)=1,
\end{align*}
which implies that also $f(e^A,e^B)\diamond \Phi(A,B)=1$, since $\diamond$ defines a group structure. This means in particular that
$a^{g(0),\Phi}=a$ and the Theorem is proven. 
\end{proof}

\section{(Co)Poisson Hopf monoids and their (de)quantization}
\label{Sec: PHAandQuant}
\subsection{Drinfeld-Yetter modules for (co)Poisson Hopf monoids}
\label{sec:DY}

In this section, we introduce categories attached to (co)Poisson Hopf algebras, which we will call \emph{Drinfeld-Yetter modules}. 
As mentioned in the introduction, a pictorial definition is given in \cite[7.2]{PulmannSevv2}, where these objects are referred to as dimodules.
We are well aware that the name Drinfeld-Yetter modules is already used for Lie bialgebras (e.g., by A. Appel and V. Toledano Laredo \cite{ATL24}). In case the (co)Poisson Hopf algebra is the universal enveloping algebra (or enveloping coalgebra) of a Lie bialgebra, we will see in Section \ref{SubSec: DQLiebialgebras}
that the Drinfeld-Yetter category in our sense is isomorphic to the usual definition of the Drinfeld-Yetter modules of the underlying Lie bialgebra. We deeply believe that the Drinfeld-Yetter modules of a general (co)Poisson Hopf algebra are of independent interest, and we are planning to study them in future projects.

\subsubsection{(co)Poisson Hopf monoids}
\begin{definition}
A coPoisson Hopf monoid in a preadditive symmetric monoidal category $\C$ is a seventuple $(C, \mu,\eta, \Delta,  \varepsilon,S, \delta)$, where $(C, \mu,\eta, \Delta,  \varepsilon,S)$ is a cocommutative Hopf monoid  and $\delta: C \to C \ten C$ is a morphism, called the Poisson cobracket, satisfying the following conditions:
\begin{align}
 \delta + \sigma\circ\delta &= 0\\
\mathrm{Alt}\circ (\id \ten \delta) \circ \delta&=0 \\
(\Delta \ten \id) \circ \delta - (\id \ten \delta) \circ \Delta - (\sigma \ten \id)\circ(\delta\ten \id) \circ \Delta&=0\\
\delta \circ \mu - (\mu \ten \mu) \circ \sigma_{(1324)} \circ (\delta \ten \Delta) - (\mu \ten \mu) \circ \sigma_{(1324)} \circ (\Delta \ten \delta)&=0\\
\delta \circ \eta &=0 \\
(\varepsilon \ten \id) \circ \delta &=0
\end{align}
where $\mathrm{Alt} = \id +\sigma_{(312)}+\sigma_{(231)}$. Dually, a Poisson Hopf monoid in $\C$ is a seventuple $(P, \mu,\eta, \Delta,  \varepsilon,S, \{\cdot,\cdot\})$, where $(P, \mu,\eta, \Delta,  \varepsilon,S)$ is a commutative Hopf monoid  and $\{\cdot,\cdot\}\colon P\ten P\to P$ is a morphism, called the Poisson bracket, satisfying the following conditions:
\begin{align}
 \{\cdot,\cdot\} + \{\cdot,\cdot\}\circ \sigma &= 0\\
 \{\cdot,\cdot\} \circ (\{\cdot,\cdot\}\ten \id )\circ \mathrm{Alt}&=0 \label{eq:Jacobi}\\
\{\cdot,\cdot\}\circ (\mu \ten \id) - \mu\circ(\id\ten\{\cdot,\cdot\})  -\mu\circ  (\{\cdot,\cdot\}\ten\id)\circ(\id \ten \sigma) &=0\\
\Delta\circ\{\cdot,\cdot\} - (\{\cdot,\cdot\}\ten \mu) \circ \sigma_{(1324)} \circ (\Delta \ten \Delta) - (\mu\ten \{\cdot,\cdot\}) \circ \sigma_{(1324)} \circ (\Delta \ten \Delta)&=0\\
\varepsilon \circ\{\cdot,\cdot\}   &=0 \\
\{\cdot,\cdot\}\circ(\eta\ten \id) &=0.
\end{align}
\end{definition}
Note that if $(P, \mu,\eta, \Delta,  \varepsilon,S, \{\cdot,\cdot\})$ is a Poisson Hopf monoid in $\C$, then $(P^\vee, \mu^\vee,\eta^\vee, \Delta^\vee,  \varepsilon^\vee,S^\vee, \delta^\vee)$ with $\delta^\vee=\{\cdot,\cdot\}$ is a coPoisson Hopf monoid in $\C^\vee$ and vice versa.

\begin{remark}
\label{Rmk: firstorderapprox}

Let $(P, \mu,\eta, \Delta,  \varepsilon,S, \{\cdot,\cdot\})$ be a Poisson Hopf monoid in $\C$ and recall the category $\C_x$ introduced in section \ref{section-aux-categories}. Then it is easy to see that
\begin{align}
\label{eq:Hopf-algebra-from-Poisson}
P_x:=\bigg(P, \mu+\frac{x}{2}\{\cdot,\cdot\},\eta, \Delta,  \varepsilon,S-\frac{x}{2}\mu\circ(\{\cdot,\cdot\}\ten \id)\circ(S\ten \id\ten S)\circ\Delta^{(2)}\bigg)
\end{align}
is a Hopf monoid in $\C_x$. In other words, we can view a Poisson Hopf structure as a first-order approximation of deforming a commutative Hopf monoid. Note that this is not one-to-one, i.e. if one has a Hopf monoid of the form \eqref{eq:Hopf-algebra-from-Poisson} for some skew-symmetric map $\{\cdot,\cdot\}$, this does not imply that $\{\cdot,\cdot\}$ is a Poisson bracket turning $P$ into a Poisson Hopf monoid, because the Jacobi identity \eqref{eq:Jacobi} does not need to hold.
Similarly, if $(C, \mu,\eta,\Delta,\varepsilon,S,\delta)$ is a coPoisson Hopf monoid in $\C$, then 
\begin{equation}
C_x := \bigg(C, \mu,\eta, \Delta+\frac{x}{2}\delta,  \varepsilon,S-\frac{x}{2}\mu^{(2)}\circ(S\ten \id\ten S)\circ(\id\ten \delta)\circ\Delta\bigg)
\end{equation}
is a Hopf monoid in $\C_x$.
Up to now, this seems only to be a nice interpretation of (co)Poisson Hopf monoids, but we will use this observation to prove some statements later on. 
\end{remark}
\subsubsection{Construction of (co)Poisson Hopf monoids}
We now extend Theorem \ref{theorem-hopf-algebra} by providing a construction of (co)Poisson Hopf monoids out of (co)adapted functors.

\begin{theorem}
\label{Thm: constrPoissonHopf}
Let $\B$ be a $\mathbb{Q}$-linear symmetric Cartier category with infinitesimal braiding $t$ and let $\C$ be a $\mathbb{Q}$-linear symmetric monoidal category (which we think of as a Cartier category with respect to the trivial infinitesimal braiding). 
\begin{enumerate}
\item Let $(M,\Delta_M,\varepsilon_M)$ be an infinitesimally cocommutative comonoid in $\B$ and let a $M$-adapted infinitesimally braided comonoidal functor $(F,F^2,F^0)\colon \B\to \C$ be given. Then the morphism
\begin{align*}
\delta_{F(M\ten M)}=\frac{1}{2}(\id-\sigma)\circ F^2(M \ten M, M \ten M) \circ F(\beta^t_{M,M,M,M}) \circ F(\Delta_M \ten \Delta_M)
\end{align*}
endows the cocommutative Hopf monoid $F(M\ten M)$ from Theorem \ref{theorem-hopf-algebra} with the structure of a coPoisson Hopf monoid in $\C$, where $\beta^t_{M,M,M,M}\colon (M\ten M)\ten (M\ten M) \to (M\ten M)\ten (M\ten M)$ denotes the morphism applying $\sigma_{M,M}\circ t_{M,M}$ to the two objects in the middle.
\item Let $(A,\mu_A,\eta_A)$ be an infinitesimally commutative monoid in $\B$ and let $(F,F_2,F_0)\colon \B\to \C$ be an $A$-coadapted infinitesimally braided monoidal functor. Then the morphism 
\begin{align*}
\{\cdot,\cdot\}_{F(A\ten A)}=\frac{1}{2}F(\mu_A \ten \mu_A) \circ F(\beta^t_{A,A,A,A}) \circ F_2(A \ten A, A \ten A)\circ(\id-\sigma)
\end{align*}
endows the commutative Hopf monoid $F(A\ten A)$ from Theorem \ref{theorem-hopf-algebra} with the structure of a Poisson Hopf monoid.
\end{enumerate}
\end{theorem}

\begin{proof}
One can, in principle, prove this theorem by hand using long diagram chases. But we use a different reasoning. Consider the category $\B_\hbar$ built in section \ref{section-aux-categories}. This is a quasi-trivial symmetric Cartier category with infinitesimal braiding $\hbar t$. The comonoid $M$ is infinitesimally cocommutative in $\B_\hbar$ and $F_\hbar\colon \B_\hbar\to \C_\hbar$ is an infinitesimally braided comonoidal functor. We choose a Drinfeld associator $(\Phi,\frac12)$ and consider the braided comonoidal functor $F_{\hbar,\Phi}\colon \B_{\hbar,\Phi}\to \C_{\hbar,\Phi}$ which is still $M$-adapted by Theorem \ref{prop: Severa}. Therefore, we can apply Theorem \ref{theorem-hopf-algebra} to build a Hopf monoid structure on $F_{\hbar,\Phi}(M\ten M)$, which coincides with the Hopf monoid structure on $F(M\ten M)$ modulo $\hbar$. Recall that 
$F(M\ten M)$ is a cocommutative Hopf algebra, since $\B$ is symmetric. Moreover, 
\begin{align*}
\delta:= \frac{\Delta_{F_{\hbar,\Phi}(M \ten M)}-\Delta^{\cop}_{F_{\hbar,\Phi}(M \ten M)}}{\hbar} \ \mod \hbar
\end{align*}
defines a coPoisson Hopf structure on $F(M\ten M)$ by general arguments. Using property $(i)$ of Definition \ref{Def: DA} and that we scaled the infinitesimal braiding with $\hbar$, we conclude that the associativity constraints in $\B_{\hbar,\Phi}$ satisfy $(a_{{\hbar}}^\Phi)_{X,Y,Z}= a_{X,Y,Z}+\mathcal{O}(\hbar^2),$
which implies that 
\begin{align*}
\Delta_{F_{\hbar,\Phi}(M \ten M)}=\Delta_{F(M \ten M)} + \frac{\hbar}{2}F^2(M \ten M, M \ten M) \circ F(\beta^t_{M,M,M,M}) \circ F(\Delta_M \ten \Delta_M)+\mathcal{O}(\hbar^2).
\end{align*}
This shows that $\delta=\delta_{F(M\ten M)}$. 
Part $(ii)$ of the statement can be proved analogously. 
\end{proof}

\begin{remark}
    Note that the factors of $\frac12$ in Theorem \ref{Thm: constrPoissonHopf} are purely cosmetic. In fact, one could drop the assumption that the categories are $\mathbb{Q}$-linear and define the (co)Poisson structures without these factors. Nevertheless, we will see later that the factors do make a difference when looking at specific categories. 
\end{remark}

Next, we state the (co)Poisson counterpart of Proposition \ref{Prop:nattrafoHopf}:

\begin{proposition}
\label{Prop:nattrafoPoissHopf}
Let $F,G\colon \B\to\C$ be two adapted infinitesimally braided comonoidal functors with respect to the same infinitesimally cocommutative comonoid $M\in \B$ and let $n\colon F\implies G$  be a natural transformation of braided comonoidal functors. Then $n_{M\ten M}\colon F(M\ten M)\to G(M\ten M)$ is a morphism of coPoisson Hopf monoids in $\C$. \\
Moreover, the dual counterpart of statement (i.e. with monoids and coadapted functors) holds by passing to the opposite-reversed setting.
\end{proposition}  

\subsubsection{Drinfeld-Yetter modules for (co)Poisson Hopf monoids: definition and Cartier structure}

We now define a version of the Yetter-Drinfeld modules of (co)Poisson Hopf algebras. These categories will be symmetric Cartier categories. 

\begin{definition}
\label{def-DY}
Let $\C$ be a preadditive symmetric monoidal category.
\begin{enumerate}
\item Let $(C, \mu,\eta, \Delta,  \varepsilon,S, \delta)$ be a coPoisson Hopf monoid in $\C$. A Drinfeld-Yetter module over $C$ is a triple $(V,\mu_V,\delta_V)$, where $(V,\mu_V)$ is a left $C$-module and $\delta_V\colon V\to V\ten C$ is a morphism satisfying
\begin{align}
(\id_V \ten \delta ) \circ \delta_V&=  (\delta_V \ten \id) \circ \delta_V - \sigma_{(132)} \circ (\delta_V\ten \id) \circ \delta_V \label{Eq: DYI-3}\\
(\id\ten \Delta)\circ\delta_V &= (\id\ten \eta\ten\id + \id\ten\id\ten\eta)\circ\delta_V \label{Eq: DYI-4}\\
\delta_V\circ \mu_V&=(\mu_V\ten \mu)\circ\sigma_{(4132)}\circ(S^{-1}\ten\delta\ten \id)\circ(\Delta\ten \id) \nonumber\\&
+(\mu_V\ten \mu^{(2)})\circ\sigma_{(51324)}\circ(S^{-1}\ten\id_3)\circ(\Delta^{(2)}\ten\delta_V).\label{Eq: DYI-6}
\end{align}
A morphism of Drinfeld-Yetter modules $(V,\mu_V,\delta_V)$ and $(W,\mu_W,\delta_W)$ is a morphism of left $C$-modules $f\colon V\to W$ satisfying $(f\ten \id)\circ \delta_V=\delta_W\circ f $.
We denote this category by $\underline{\mathfrak{DY}}(C,\C)$. 
\item Let $(P, \mu,\eta, \Delta,  \varepsilon,S, \{\cdot, \cdot\})$ be a Poisson Hopf monoid in $\C$. A Drinfeld-Yetter module over $P$ is a triple $(V,\pi_V,\Delta_V)$, where $(V,\Delta_ V)$ is a right $P$-comodule and $\pi_V\colon P\ten V\to V$ is a morphism satisfying
\begin{align}
\pi_V\circ(\{\cdot,\cdot\}\ten \id)&=\pi_V\circ (\id\ten \pi_V)\circ ((\id-\sigma)\ten \id)\\
\pi_V \circ (\mu \ten \id) &= \pi_V\circ (\varepsilon\ten \id\ten \id + \id\ten\varepsilon\ten \id) \\
\Delta_V\circ \pi_V&=(\id\ten\mu)\circ(\id\ten\{\cdot,\cdot\}\ten \id)\circ\sigma_{(4213)}\circ(S^{-1}\ten \id_3)\circ(\Delta\ten\Delta_V) \nonumber \\
&+(\pi\ten\mu^{(2)})\circ\sigma_{(51324)}\circ(S^{-1}\ten\id_3)\circ(\Delta^{(2)}\ten\Delta_V).
\end{align}
A morphism of Drinfeld-Yetter modules $(V,\pi_V,\Delta_V)$ and $(W,\pi_W,\Delta_W)$ is a morphism of right $P$-comodules $f\colon V\to W$ satisfying $ f\circ\pi_V=\pi_W\circ (\id\ten f)$. We denote this category by $\overline{\mathfrak{DY}}(P,\C)$. 
\end{enumerate} 
\end{definition}

\begin{theorem}
\label{Thm: DYinf}
Let $\C$ be a preadditive symmetric monoidal category. 
\begin{enumerate}
\item Let $C$ be a coPoisson Hopf monoid in $\C$. Then $\underline{\mathfrak{DY}}(C,\C)$ is a symmetric Cartier category. More specifically, if $(V,\mu_V,\delta_V),(W,\mu_W,\delta_W)$ are two objects in $\underline{\mathfrak{DY}}(C,\C)$, the Drinfeld-Yetter module structure on the tensor product $V \ten W$ is 
\begin{align}
\mu_{V\ten W}& = (\mu_V \ten \mu_W) \circ \sigma_{(1324)} \circ (\Delta \ten \id \ten \id)\\
\delta_{V\ten W}&=\sigma_{(132)} \circ(\delta_V\ten\id) + \id\ten \delta_W.
\end{align}
The braiding of $\underline{\mathfrak{DY}}(C,\C)$ is the same as $\C$ and the infinitesimal braiding is 
\begin{align}
\label{Eq: InfBraidDY}
t_{V\ten W}=-(\id\ten \mu_W)\circ (\delta_V\ten\id)- (\mu_V\ten \id)\circ \sigma_{(231)}\circ(\id\ten \delta_W).
\end{align}
\item Let $P$ be a Poisson Hopf monoid in $\C$. Then $\overline{\mathfrak{DY}}(P,\C)$ is a symmetric Cartier category. More specifically, if $(V,\pi_V,\Delta_V),(W,\pi_W,\Delta_W)$ are two objects in $\overline{\mathfrak{DY}}(P,\C)$, the Drinfeld-Yetter module structure on the tensor product $V \ten W$ is 
\begin{align}
\pi_{V\ten W}& = \pi_V \ten \id +
(\id\ten \pi_W)\circ \sigma_{(213)}\\
\Delta_{V\ten W}&=(\id\ten \mu)\circ\sigma_{(1324)}\circ(\Delta_V\ten\Delta_W).
\end{align}
The braiding of $\overline{\mathfrak{DY}}(P,\C)$ is the same as $\C$ and the infinitesimal braiding is 
\begin{align}
\label{Eq: InfBraidDYtwo}
t_{V\ten W}=-(\id\ten \pi_W)\circ (\Delta_V\ten\id)- (\pi_V\ten \id)\circ \sigma_{(231)}\circ(\id\ten \Delta_W).
\end{align}
\end{enumerate}
\end{theorem}

\begin{proof}
Let $(C,\mu,\eta,\Delta,\varepsilon,S,\delta)$ be a coPoisson Hopf monoid and consider the Hopf monoid $C_x\in \C_x$ from Remark \ref{Rmk: firstorderapprox}. Consider the following full subcategory of $\underline{\Y\D}(C_x,\C_x)$:
\begin{align*}
\underline{\A}(C_x, \C):=\big\{(V,\tilde{\mu},\tilde{\Delta})\in \underline{\Y\D}(C_x,\C_x)\ | \
\tilde{\mu}=\mu + x\cdot 0 ,\ \tilde{\Delta}=\id\ten \eta +\mathcal{O}(x)\big\}
\end{align*}
which is closed under tensor products. 
Then we can embed $\underline{\mathfrak{DY}}(C,\C)$ into $\underline{\A}(C_x, \C)$ via 
\begin{align}
\label{eq:embedding}
\underline{\mathfrak{DY}}(C,\C) \to \underline{\A}(C_x, \C), \quad (V,\mu_V,\delta_V)
\mapsto (V,\mu_V,\id\ten \eta + x\delta_V).
\end{align}

In other words, we can interpret Drinfeld-Yetter modules over $C$ as certain Yetter-Drinfeld modules over $C_x$. The fact that $\underline{\A}(C_x, \C)$ is closed under tensor products --upon checking the compatibility of the embedding \eqref{eq:embedding} with tensor products-- implies that $\underline{\mathfrak{DY}}(C,\C) $ has the monoidal structure described in the statement (where conditions \eqref{Eq: DYI-4}-\eqref{Eq: DYI-6} follow immediately, whence condition \eqref{Eq: DYI-3} follows from an easy computation). 

Similarly, the braided structure of $\underline{\Y\D}(C_x,\C_x)$ endows $\underline{\mathfrak{DY}}(C,\C) $ with a symmetric Cartier structure: one first checks that
\begin{align*}
\sigma_{V,W}^{\Y\D}=\sigma_{V,W}\circ\bigg(\id + \frac{x}{2}t_{V,W}\bigg) \ \text{ and } \ (\sigma_{V,W}^{\Y\D})^{-1}=\sigma_{W,V}\circ\bigg(\id -\frac{x}{2}t_{W,V}\bigg),
\end{align*}
which implies that $\sigma_{V,W}\circ t_{V,W}=t_{W,V}\circ \sigma_{V,W}$.
Moreover, we see that $(\sigma_{V,W}^{\Y\D})^2=\id + xt_{V,W}$
and the fact that $\sigma^{\Y\D}$ fulfills the hexagon axioms implies that $t$ fulfills the infinitesimal hexagon relations (which is a similar argument as the one used in the proof of Theorem \ref{Thm:Quasisym=infbraided}). The only thing left to check is that $t_{V,W}$ fulfills $ \delta_{V\ten W}\circ t_{V,W} = (t_{V,W}\ten\id) \circ\delta_{V\ten W}$, which follows again from an easy computation. The proof of part $(ii)$ follows the same lines.
\end{proof}

\subsubsection{The adjoint and coadjoint Drinfeld-Yetter modules}
Let us now define the Poisson analogues of the adjoint and coadjoint Yetter-Drinfeld modules defined in section \ref{section-adj-coadj-YD}.

\begin{proposition}
Let $\C$ be a preadditive symmetric monoidal category. 
\begin{enumerate}
\item Let $(C, \mu,\eta, \Delta,  \varepsilon,S, \delta)$ be a coPoisson Hopf monoid in  $\C$. Then the triple $C_-=(C,\mu_-,\delta_-)$ is an object of $\underline{\mathfrak{DY}}(C,\C)$, where 
\begin{align}
\mu_-&=\mu \\
\delta_-&=(\id\ten\mu^{\op})\circ(\id\ten S^{-1}\ten \id)\circ\sigma_{(213)}\circ(\id\ten \delta)\circ\Delta.
\end{align}
Moreover, the coproduct $\Delta\colon C\to C\ten C$ is a morphism of Drinfeld-Yetter modules, turning $C_-$ into an infinitesimally cocommutative comonoid in $\underline{\mathfrak{DY}}(C,\C)$.  
\item Let $(P, \mu,\eta, \Delta,  \varepsilon,S, \{\cdot,\cdot\})$ be a Poisson Hopf monoid in  $\C$. Then the triple $P_+=(P,\pi_+,\Delta_+)$ is an object of $\overline{\mathfrak{DY}}(P,\C)$, where
\begin{align}
\pi_+&=\mu\circ(\{\cdot,\cdot\}\ten \id)\circ \sigma_{(132)}\circ(\id\ten S^{-1}\ten \id)\circ(\Delta^{\cop}\ten \id)\\
\Delta_+&=\Delta.
\end{align}
Moreover, the product $\mu\colon P\ten P\to P$ is a morphism of Drinfeld-Yetter modules, turning $P_+$ into an infinitesimally commutative monoid in $\overline{\mathfrak{DY}}(P,\C)$.  
\end{enumerate}
\end{proposition}

\begin{proof}
Let $C$ be a coPoisson Hopf monoid, $C_x$ be the Hopf monoid built as in Remark \ref{Rmk: firstorderapprox}, and
$(C_x)_-\in \underline{\Y\D}(C_x,\C_x)$ be the Yetter-Drinfeld module as in Proposition \ref{proposition-Hplusminus}. A straightforward computation shows that 
the datum of $(C_x)_-$ is given by the triple $(C,\mu_-, \id\ten \eta + x\delta_-)$. This implies that $C_-$ fulfills all the requirements to be a Drinfeld-Yetter module besides Equation \eqref{Eq: DYI-3}, which can be easily checked by hand.  The same arguments work for $P_+$. 
\end{proof}

\begin{remark}
\label{remark-secord}
The proofs of the last two theorems show that Remark \ref{Rmk: firstorderapprox} is not only a nice interpretation, but also sheds some light on the nature of the Drinfeld-Yetter modules as we defined them: they are first-order approximations of the Yetter-Drinfeld modules of a first-order deformation of a (co)commutative Hopf algebra. One could also see that in all the cases we had to prove only one property by hand, which was not captured by this interpretation. This additional property is, in fact, the obstruction to extending the first-order deformation to a second-order deformation, which cannot be seen by the first-order deformation theory. 
\end{remark}

\begin{theorem}
\label{thm: (co)adaptedPoiss}
Let $\C$ be a $\mathbb{K}$-linear good enough category, $C$ be a coPoisson Hopf monoid in $\C$, and $P$ be a Poisson Hopf monoid in $\C$. Then the assignments
\begin{align}
\F_-\colon \underline{\mathfrak{DY}}(C,\C) &\to \C, \quad  (V,\mu_V,\delta_V) \mapsto \mathrm{coker}(\mu_V-\varepsilon\ten\id), \quad f \mapsto f_{|\mathrm{coker}} \\
\F_+\colon \overline{\mathfrak{DY}}(P,\C) &\to \C, \quad (V,\pi_V,\Delta_V)\mapsto \mathrm{ker}(\Delta_V-\id \ten \eta), \quad f \mapsto f_{|\mathrm{ker}} 
\end{align}
are functorial. Furthermore, the canonical morphisms
\begin{align*}
&\F_-^{2}(V,W)\colon \mathrm{coker}(\mu_{V\ten W}-\varepsilon\ten\id)\to \mathrm{coker}(\mu_V-\varepsilon\ten\id)\ten \mathrm{coker}(\mu_W-\varepsilon\ten\id) \\&
\F_-^{0}\colon \mathrm{coker} (\mu_{I}-\varepsilon\ten \id)\to I\\
&\F_{+,2}(V,W)\colon \mathrm{ker}(\Delta_V-\id \ten \eta)\ten \mathrm{ker}(\Delta_W-\id \ten \eta)\to \mathrm{ker}(\Delta_{V\ten W}-\id \ten \eta) \\&
\F_{+,0}\colon I\to \mathrm{ker} (\Delta_{I}-\id \ten \eta)
\end{align*}
make $\F_-$ (resp. $\F_+$) an infinitesimally braided comonoidal (resp. monoidal) functor. Moreover, the functor $\F_-(C_-\ten - )$ (resp. $\F_+(-\ten P_+ )$) is naturally isomorphic to the forgetful functor, and hence $\F_-$ (resp. $\F_+$) is $C_-$-adapted (resp. $P_+$-coadapted) and we get $\F_{- }(C_-\ten C_-)\cong C$ as coPoisson Hopf monoids (resp. $\F_{+ }(P_+\ten P_+)\cong P$ as Poisson Hopf monoids). 
\end{theorem}
\begin{proof}
We show the statement in the Poisson Hopf case; the coPoisson Hopf one follows along the same lines. Consider the following braided strongly monoidal functor 
\begin{align*}
\mathcal{I}\colon \overline{\mathfrak{DY}}(P,\C) \to \overline{\Y\D}(P,\C), \quad (V,\pi_V,\Delta_V)\mapsto (V,\varepsilon\ten \id, \Delta_V) ,
\end{align*}
which is well-defined since for every commutative Hopf monoid, one can turn a right comodule into a Yetter-Drinfeld module by considering the trivial action. This functor is braided since, on Yetter-Drinfeld modules with trivial action, the braiding is simply the symmetric braiding of the underlying category. Moreover, $\mathcal{I}(P_+)=P_+$, since by Proposition \ref{proposition-Hplusminus} the action of $P_+\in \overline{\mathfrak{DY}}(P,\C)$ is trivial. Therefore, we have $\F_+=\mathcal{F}_+\circ \mathcal{I}$, hence $\F_+$ is braided monoidal and $\F_+(-\ten P_+)$ is isomorphic to the forgetful functor.
Moreover, by the definition of the infinitesimal braiding (see Equation \eqref{Eq: InfBraidDYtwo}, we get that the concatenation
\begin{align*}
\mathrm{ker}(\Delta_V-\id \ten \eta)\ten \mathrm{ker}(\Delta_W-\id \ten \eta)\to \mathrm{ker}(\Delta_{V\ten W}-\id \ten \eta)\overset{t_{V\ten W}}{\longrightarrow} \mathrm{ker}(\Delta_{V\ten W}-\id \ten \eta)
\end{align*}
vanishes, showing that $\F_+$ is infinitesimally braided (with respect to the trivial infinitesimal braiding in $\C$).
Moreover, $\F_+(P_+\ten P_+)\cong P$ as Hopf monoids, which follows from Theorem \ref{Thm: (co)adaptedHopffunctors}.
Using the explicit isomorphism $\F_+(P_+\ten P_+)\cong P$ coming from the fact that $\F_+(-\ten P_+)$ is naturally isomorphic to the forgetful functor, one sees that the corresponding Poisson structures coincide as well by the definition of $\pi_+$, which can be seen using the same reasoning as in the proof of Theorem \ref{Thm: constrPoissonHopf}. 
\end{proof}

\begin{remark}
The (co)Poisson counterpart of Proposition \ref{Prop: dualHopf} is also true: for a coPoisson Hopf monoid $C \in \C$ we have $\underline{\mathfrak{DY}}(C,\C)^\vee\cong \overline{\mathfrak{DY}}(C^\vee,\C^\vee)$.
Moreover, we have that $(C_-)^\vee=C_+$ and $(\F_-)^\vee=\F_+$. 
\end{remark}

\subsubsection{Construction of Drinfeld-Yetter modules}
Next, we state the (co)Poisson counterpart of Theorem \ref{Thm: Bordemann}, which allows one to construct Drinfeld-Yetter modules out of (co)adapted functors. We also define the (co)Poisson analogues of the functors $\Gamma,\Xi$ defined in Equations \eqref{eq:Gamma}-\eqref{eq:Xi}. 
\begin{theorem}
\label{Thm: Bordemann-Poisson}
Let $\B$ be a symmetric Cartier category and $\C$ be a symmetric monoidal category (which we think as a Cartier category with the trivial infinitesimal braiding).
\begin{enumerate}
\item
Let $(M,\Delta,\varepsilon)$ be an infinitesimally cocommutative comonoid in $\B$ and $(F,F^2,F^0): \B \to \C$ be an infinitesimally braided $M$-adapted functor. Then 
\begin{enumerate}
\item For any object $X$ in $\B$, the morphisms 
\begin{align*}
\mu_{F(M \ten X)} &=  F(r_M \ten \id_X) \circ F((\id_M \ten \varepsilon_M) \ten \id_X) \circ (\gamma_{M,X}^M)^{-1} \\
\delta_{F(M\ten X)}&=\frac{1}{2}F^2(M \ten X, M\ten M) \circ\Big(  F\big((\id\ten \sigma \circ t\ten \id) \circ (\id\ten \sigma)  \circ (\Delta^{(2)} \ten \id_X)\big)\\&
-F\big((\id\ten \sigma \ten \id) \circ (\id\ten \sigma \circ t)  \circ (\Delta^{(2)} \ten \id_X)\big)\Big)
\end{align*}
endow $F(M \ten X)$ with the structure of a Drinfeld-Yetter module over the Hopf monoid $F(M \ten M)$ of Theorem \ref{theorem-hopf-algebra}.
\item The assignment $\Upsilon_F\colon \B \to\underline{\mathfrak{DY}}\big({F(M \ten M)},\C\big)$, $X \mapsto F(M \ten X) $, $f \mapsto F(\id_M \ten f)$ is a strongly infinitesimally braided comonoidal functor.
\item  The comonoids $\Upsilon_F(M)$ and $F(M\ten M)_-$ coincide in $\underline{\mathfrak{DY}}({F(M \ten M)},\C)$.
\end{enumerate}
\item Let  $(A,\mu,\eta)$ be an infinitesimally commutative monoid in $\B$ and $(F,F_2,F_0): \B \to \C$ be an infinitesimally braided $A$-coadapted functor. Then 
\begin{enumerate}
\item For any object $X$ in $\B$, the morphisms 
\begin{align*}
\Delta_{F(X\ten A)} &= (\psi_{X,A}^A)^{-1}\circ F(\id\ten(\eta \ten \id)) \circ F(\id\ten \ell_A^{-1})\\
\pi_{F(X \ten A)} &= \frac{1}{2}\Big(  F\big((\id\ten\mu^{(2)})\circ (\sigma \ten \id)\circ (\id\ten \sigma \circ t\ten \id)\big)\\&
                -
                ((\id\ten\mu^{(2)})\circ (\sigma \circ t\ten \id)\circ(\id\ten \sigma \ten \id))\Big)\circ F_2(A\ten A, X\ten A) 
\end{align*}
endow $F(X\ten A)$ with the structure of a Drinfeld-Yetter module over the Hopf monoid $F(A \ten A)$ of Theorem \ref{theorem-hopf-algebra}.
\item The assignment $\Omega_F\colon \B \to\overline{\mathfrak{DY}}\big({F(A\ten A)},\C\big)$, $X \mapsto F(X \ten A) $, $f \mapsto F(f \ten \id_A)$ is a strongly infinitesimally braided monoidal functor.
\item  The monoids $\Omega_F(A)$ and $F(A\ten A)_+$ coincide in $\overline{\mathfrak{DY}}\big({F(A\ten A)},\C\big)$.
\end{enumerate}
\end{enumerate}
\end{theorem}

\begin{proof}
We use the same reasoning as in Theorem \ref{Thm: constrPoissonHopf}: we choose a Drinfeld associator $(\Phi, \frac12)$
and deform $\B$ into $\B_{\hbar,\Phi}$. Applying Theorem \ref{Thm: Bordemann}, we get that 
\begin{align*}
\Delta_{F_{\hbar,\Phi}(M\ten X)}=\Delta_{F(M\ten X)}+ \hbar\delta_{F(M\ten X)}+\mathcal{O}(\hbar^2)
=\id\ten \eta + \hbar\delta_{F(M\ten X)}+\mathcal{O}(\hbar^2),
\end{align*}
where we used that for symmetric braidings the coaction is trivial. All the statements follow now by zero and first-order comparisons. The same line of thought can be applied to the coadapted case. 
\end{proof}

\begin{theorem}
\label{thm: extBordemannPoissII} \label{Prop:natTrafPoiss}
Let $\B$ be a symmetric Cartier category and $\C$ be a symmetric monoidal category.
\begin{enumerate}
\item Let $P\in \C$ (resp. $C\in \C$) be a Poisson (resp. coPoisson) Hopf monoid, let $\C$ be good enough and let $\F_-\colon \underline{\mathfrak{DY}}(C,\C)\to \C$ and $\F_+\colon \overline{\mathfrak{DY}}(P,\C)\to \C$ be the functors defined in Theorem \ref{thm: (co)adaptedPoiss}.
Then
\begin{align}
\Upsilon_{\F_-}\colon \underline{\mathfrak{DY}}(C,\C)\to \underline{\mathfrak{DY}}(\F_-(C_-\otimes C_-),\C) \\
\Omega_{\F_+}\colon \overline{\mathfrak{DY}}(P,\C)\to \overline{\mathfrak{DY}}(\F_+(P_+\otimes P_+),\C)
\end{align} 
are equivalences of categories.
\item Let $F,G\colon\B\to\C$ be two infinitesimally braided (co)monoidal 
functors and let $M$ an infinitesimally (co)commutative (co)monoid, such that $F$ and $G$ are $M$-(co)adapted with respect to $M$. Additionally, let $n\colon F\implies G$ be a natural (co)monoidal transformation such that $n_{M\ten X}\colon F(X\ten M)\to G(X\ten M)$ are isomorphisms. Then by Proposition \ref{Prop:nattrafoPoissHopf} $n_{M\ten M}\colon F(M\ten M)\to G(M\ten M)$ is an isomorphism of (co)Poisson Hopf monoids, and therefore it induces an invertible infinitesimally braided strongly (co)monoidal functor $\mathcal{I}_{n_{M\ten M}}\colon \mathfrak{DY}(F(M\ten M),\C)\to\mathfrak{DY}(G(M\ten M),\C)$
with the bars either above or below on both sides, depending on whether $G$ and $F$ are adapted or coadapted. Moreover,  $n$ induces  natural isomorphisms $\mathcal{I}_{n_{M\ten M}}\circ\Upsilon_F\implies \Upsilon_G$ and $I_{n_{M\ten M}}\circ \Omega_F \implies \Omega_G$.
\item Let $\D$ be a symmetric Cartier category together with an  infinitesimally braided (co)monoidal functor $G\colon \D\to \B$, let $M\in \D$ be an infinitesimally (co)commutative (co)monoid and let $F\colon \B\to \C$ be an infinitesimally braided (co)monoidal functor which is $G(M)$-(co)adapted. Then there are natural (co)monoidal isomorphisms (one in the adapted case, and one in the coadapted one) $\Upsilon_{F\circ G}\implies \Upsilon_F\circ G$ and $\Omega_{F\circ G}\implies \Omega_F\circ G$.
\item 
Let $M\in \B$ be an infinitesimally cocommutative comonoid, let $\C$ be $c$-good enough and let $F\colon \B\to \C$ be an infinitesimally braided comonoidal $M$-adapted functor.
There is a natural comonoidal transformation $\mathfrak{n}^-\colon \F_-\circ \Upsilon_F\implies F$ such that for all $X\in \B$ the maps $\mathfrak{n}^-_{M\ten X}\colon (\F_-\circ \Upsilon_F)(M\ten X)\to F(M\ten X)$ are isomorphisms.
\item 
Let $A\in \B$ be an infinitesimally commutative monoid, let $\C$ be $k$-good enough and let $F\colon \B\to \C$ be an infinitesimally braided monoidal $A$-coadapted functor.
There is a natural monoidal transformation $\mathfrak{n}^+\colon  F\implies \F_+\circ \Omega_F $  such that for all $X\in \B$ the maps $\mathfrak{n}^+_{X\ten A}\colon F(X\ten A) \to (\F_+\circ \Omega_F)(X\ten A)$ are isomorphisms.
\end{enumerate}
\end{theorem}

\begin{proof}
The proof follows the same line as the one of Theorem \ref{Prop:natTraf}.
\end{proof}

\subsubsection{Functoriality}
\label{Subsubsec: FunctorialityII}
The following lines are the Poisson counterpart of Section \ref{Subsubsec: FunctorialityI}
for Drinfeld-Yetter modules of (co)Poisson Hopf algebras, i.e. we show that arrows \ref{lbl:DY} and \ref{lbl:gS} are functorial. Moreover, we discuss how to interpret Section \ref{SubSec: Infbraid} in this setting. To obtain proofs of the following results, it suffices to apply the reasoning as in Remark \ref{Rmk: firstorderapprox} to the results in Section \ref{Subsubsec: FunctorialityI}.  

We fix a good enough symmetric monoidal category $\C$ (which we often think as a symmetric Cartier category with respect to the trivial infinitesimal braiding) and denote by $\mathrm{CoPoissH}(\C)$ the category of coPoisson Hopf monoids in $\C$ and by $\mathrm{PoissH}(\C)$ the category of Poisson Hopf monoids in $\C$. We define the categories $\mathrm{CartCat}^3(\C)$ and $\mathrm{CartCat}_3(\C)$ (overlooking again set/class-theoretic issues) as follows: 

\begin{enumerate}
        \item Objects of $\mathrm{CartCat}^3(\C)$ are triples $(\B,M,F)$, where $\B$ is a symmetric Cartier category, $M\in \B$ is an i-cocommutative comonoid and $F\colon \B\to \C$ is an i-braided $M$-adapted comonoidal functor. Morphisms in $\mathrm{CartCat}^3(\C)$ are triples $(\phi, \psi,n)\colon (\B,M,F)\to (\B',M',F')$, where $\phi=(\phi,\phi^2,\phi^0):\B\to \B'$ is a i-braided comonoidal functor, $\psi:\phi(M)\to M'$ is a morphism of comonoids, and $n\colon F\implies F'\circ \phi$ is a natural  comonoidal transformation. 
         \item Objects of $\mathrm{CartCat}_3(\C)$ are triples $(\B,A,F)$, where $\B$ is a symmetric Cartier category, $A\in \B$ is a i-commutative monoid, and $F\colon \B\to \C$ is an i-braided $A$-coadapted monoidal functor. Morphisms in $\mathrm{CartCat}_3(\C)$ are triples $(\varphi,\chi,m)\colon (\B,A,F)\to (\B',A',F')$, where $\varphi = (\varphi, \varphi_2,\varphi_0):\B\to \B'$ is an i-braided monoidal functor, $\chi: A'\to \phi(A)$ is a morphism of monoids, and $m\colon F'\circ \phi\implies F$ is a natural monoidal transformation. 
    \end{enumerate}
\begin{theorem}
\label{theorem-Severa-coPoisson}
The assignments 
\begin{align*}
\check{\mathfrak{S}}_-\colon\mathrm{CartCat}^3(\C)&\to \mathrm{CoPoissH}(\C)\\
\check{\mathfrak{S}}_+\colon\mathrm{CartCat}_3(\C)&\to \mathrm{PoissH}(\C)
\end{align*}
defined on objects by $\check{\mathfrak{S}}_-(\B,M,F)=F(M\ten M)$ (resp. $\check{\mathfrak{S}}_+(\B,A,F)=F(A\ten A)$) by means of Theorem \ref{Thm: constrPoissonHopf} and on morphisms by 
\begin{align*}
\check{\mathfrak{S}}_-(\phi,\psi,n) &:= F'(\psi \ten \psi) \circ F' (\phi^2(M,M)) \circ n_{M \ten M} : F(M \ten M) \to F'(M'\ten M') \\
\check{\mathfrak{S}}_+(\varphi,\chi,m) &:=  m_{A \ten A}  \circ F' (\varphi_2(A,A)) \circ F'(\chi \ten \chi)  : F'(A' \ten A') \to F(A\ten A)
\end{align*}
define respectively a functor and a contravariant functor. 
\end{theorem}

Conversely, we define the functors $\mathfrak{B}_-\colon \mathrm{CoPoissH}(\C)\to\mathrm{CartCat}^3(\C)$ and $\mathfrak{B}_+\colon \mathrm{PoissH}(\C)\to\mathrm{CartCat}_3(\C)$ on objects by $\mathfrak{B}_-(C)=(\underline{\mathfrak{DY}}(C,\C),C_-,\mathfrak{F}_-)$ and $\mathfrak{B}_+(P)=(\overline{\mathfrak{DY}}(P,\C),P_+,\mathfrak{F}_+)$,
and define the assignments on morphisms as in Section \ref{Subsubsec: FunctorialityI}. Let us sketch how this works for coPoisson Hopf monoids: let $f\colon C\to C'$ be a morphism of coPoisson Hopf monoids, and let $(V,\mu_V,\delta_V)\in \underline{\mathfrak{DY}}(C,\C)$. We first define the map $\delta'_V=(\id\ten f)\circ \delta_V\colon V\to V\ten C'$. Next, we define the map $\delta_{C'\ten V}=\sigma_{(132)}\circ(\delta_-\ten\id)+ (\id_{H'\ten V}\circ \mu^{(2)})\circ \sigma_{(51324)} \circ(\Delta^{(2)}\ten \delta_V')$ and
one can check that, together with $\mu_{H'\ten V}=\mu\ten \id$, it endows $H'\ten V$ with the structure of a Drinfeld-Yetter module over $H'$. Moreover, both maps descend to coequalizers, and therefore endow $H'\ten_HV$ with the structure of an $H'$- Drinfeld-Yetter module. Finally, one constructs the 
natural transformation and the morphism in the same way as in Section \ref{Subsubsec: FunctorialityI}. We have the following

    \begin{theorem}
    \label{theorem-Brochier-coPoisson}
The assignments
\begin{align*}
\mathfrak{B}_-\colon \mathrm{CoPoissH}(\C)\to \mathrm{CartCat}^3(\C)\\
\mathfrak{B}_+\colon \mathrm{PoissH}(\C)\to \mathrm{CartCat}_3(\C)
\end{align*}
defined on objects by $\mathfrak{B}_-(C)=(\underline{\mathfrak{DY}}(C,\C),C_-,\mathfrak{F}_-)$
        and $\mathfrak{B}_+(P)=(\overline{\mathfrak{DY}}(P,\C),P_+,\mathfrak{F}_+)$ can be made functorial. Moreover, the compositions
        $\check{\mathfrak{S}}_-\circ \mathfrak{B}_-$ and $\check{\mathfrak{S}}_+\circ \mathfrak{B}_+$ are isomorphic to the identity functors. 
\end{theorem}
Note that we also have Poisson counterparts of the statements at the end of Section \ref{Subsubsec: FunctorialityI}, which we will not spell out here. 

We end this section with an interpretation of the results of Section \ref{SubSec: Infbraid} using the categories introduced above and in Section \ref{Subsubsec: FunctorialityI}. Let us fix a complete, separated, filtered, and good enough category $\C$. We define 
    $\mathrm{qbMonCat}^3(\C)$ as the subcategory of 
        $\mathrm{bMonCat}^3(\C)$ whose objects of the form $(\B,M,F)$, where 
        $\B$ is a quasi-symmetric category, and $F\colon \B\to \C$ is a filtered functor, and whose Morphisms consist of morphisms in $\mathrm{bMonCat}^3(\C)$ which are also filtered. 
Similarly, we define $\mathrm{qbMonCat}_3(\C)$, $\mathrm{qCartCat}^3(\C)$ and 
$\mathrm{qCartCat}^3(\C)$ by restricting to quasi-symmetric or quasi-trivial and to filtered functors. The discussion in Section $\ref{SubSec: Infbraid}$ can be now summarized in the following

\begin{theorem}
\label{thm: DA-GT}
Let $\C$ complete, separated, filtered, and good enough category, $\Phi$ be a Drinfeld associator, and $g: \mathbb{K} \to \overline{GT}(\mathbb{K})$ be the associated unique morphism of semigroups. Then the assignments 
\begin{align}
\mathrm{qCartCat}^3(\C) &\to \mathrm{qbMonCat}^3(\C), \quad (\B,M,F)\mapsto (\B_\Phi,M,F) , \quad (\phi,\psi,n) \mapsto (\phi,\psi,n) \label{eq:q-functor-1}\\
\mathrm{qCartCat}_3(\C) &\to \mathrm{qbMonCat}_3(\C) , \quad (\B,A,F)\mapsto (\B_\Phi,A,F)\label{eq:q-functor-2}, \quad (\varphi,\chi,m) \mapsto (\varphi,\chi,m)  \\
\mathrm{qbMonCat}^3(\C) & \to \mathrm{qCartCat}^3(\C) , \quad (\B,M,F)\mapsto (\B_{g(0)},M,F) \label{eq:q-functor-3} , \quad (\phi,\psi,n) \mapsto (\phi,\psi,n)  \\
\mathrm{qbMonCat}^3(\C) &\to \mathrm{qCartCat}_3(\C), \quad  (\B,A,F)\mapsto (\B_{g(0)},A,F) \label{eq:q-functor-4}, \quad (\varphi,\chi,m) \mapsto (\varphi,\chi,m) 
\end{align}
are functorial. Moreover, the pairs \eqref{eq:q-functor-1}-\eqref{eq:q-functor-3} and \eqref{eq:q-functor-2}-\eqref{eq:q-functor-4} consist of pairs of functors which are inverse to each other.
\end{theorem}

\begin{proof}
    This is just an application of Theorems \ref{prop: Severa}, \ref{thm: qasisygivesinfitesimal braiding}, and \ref{Thm:Quasisym=infbraided}. 
\end{proof}

\subsection{The categories $\M_\pm$ and $\mathfrak{M}_\pm$}
\label{Sec: DQHopf/PoissHopf}
Let us fix in this section a complete, and separated filtered good enough category $\C$. Recall also, that being good enough means that, as mentioned in Remark \ref{remark_good_enough}, that, in particular, we need that $\C$ is $c$-good enough in the \emph{minus} setting, while we need that $\C$ is $k$-good enough in the \emph{plus} setting. We introduce some special classes of Hopf monoids in $\C$ in the following

\begin{definition}
\label{Def: goodHopf}
${}$
\begin{enumerate}
\item A coPoisson Hopf monoid $C$ in $\C$ is called quantizable if its Poisson cobracket satisfies $\delta\in \mathrm{F}^1\Hom(C,C\ten C) $.
\item A Poisson Hopf monoid $P$ in $\C$ is called  coquantizable if its Poisson bracket satisfies 
$\{\cdot,\cdot\}\in \mathrm{F}^1\Hom(P\ten P,P) $.
\item A Hopf monoid $H$ in $\C$ is called dequantizable if its comultiplication satisfies 
$\Delta-\sigma\circ\Delta \in \mathrm{F}^1\Hom(H,H\ten H)$.
\item A Hopf monoid $K$ in $\C$ is called codequantizable if its multiplication satisfies $\mu-\mu\circ \sigma \in \mathrm{F}^1\Hom(K\ten K,K)$.
\end{enumerate}
We denote the respective categories by $\mathrm{qCoPoissH}(\C)$, 
$\mathrm{cqPoissH}(\C)$, $\mathrm{dqHopf}(\C)$, and $\mathrm{cdqHopf}(\C)$.
\end{definition}
To a quantizable coPoisson Hopf monoid $C$ (resp. coquantizable Poisson Hopf monoid $P$), we associate the Cartier categories  
\begin{align}
\mathfrak{M}_-(C,\C)&=\big\{(V,\mu_V,\delta_V)\in \underline{\mathfrak{DY}}(C,\C)\ | \ \delta_V\in \mathrm{F}^1\Hom(V,V\ten C)\big\} \\
\mathfrak{M}_+(P,\C)&=\big\{(V,\pi_V,\Delta_V)\in \overline{\mathfrak{DY}}(P,\C)\ | \ \pi_V\in \mathrm{F}^1\Hom(P\ten V,V)\big\}
\end{align}
which are full subcategories of $\underline{\mathfrak{DY}}(C,\C)$ and $\overline{\mathfrak{DY}}(P,\C)$ respectively.
Similarly, to a dequantizable Hopf monoid $H$ (resp. codequantizable Hopf monoid $K$), we associate the categories
\begin{align}
\M_-(H, \C)&=\big\{(V,\mu_V,\Delta_V)\in \underline{\Y\D}(H,\C)\ | \ \Delta_V-\id\ten \eta\in \mathrm{F}^1\Hom(V,V\ten H)\big\} \\
\M_+(K,\C)&=\big\{(V,\mu_V,\Delta_V)\in \overline{\Y\D}(K,\C)\ | \ \mu_V-\varepsilon\ten \id\in \mathrm{F}^1\Hom(K\ten V,V)\big\}
\end{align}
which are full subcategories of $\underline{\Y\D}(H,\C)$ and $\overline{\Y\D}(K,\C)$ respectively. The category $\M_-$ was introduced by A. Appel and V. Toledano Laredo \cite{ATL18} under the name of \emph{admissible Drinfeld-Yetter modules}.

\begin{theorem}
\label{Thm: admissibleSubcategories}
The categories $\mathfrak{M}_-(C,\C)$ and $\mathfrak{M}_+(P,\C)$ are quasi-trivial, while the categories $\M_-(H,\C)$ and $\M_+(K,\C)$ are quasi-symmetric. Moreover, we have that
\[ C_- \in \mathfrak{M}_-(C,\C), \quad P_+ \in \mathfrak{M}_+(P,\C), \quad H_- \in \M_-(H,\C), \quad K_+ \in \M_+(K,\C) .\]
\end{theorem}

\begin{proof}
We only show the statement for $\mathfrak{M}_-$, the rest follow along the same lines. First, we notice that for any objects $(V,\mu_V,\delta_V)$ and $(W,\mu_W,\delta_W)$ in $\mathfrak{M}_-(C,\C)$, the association
\[\Hom_{\C}(V,W)\to \Hom_{\C}(C\ten V,W)\oplus\Hom(V,W\ten C), \quad f\mapsto \big(f\circ\mu_V-\mu_W\circ(\id\ten f), (f\ten\id)\circ\delta_V-\delta_W\circ f\big)\]
is filtration preserving. This means that its kernel is complete and separated. But its kernel is precisely a morphism of Drinfeld-Yetter modules. This means that $\underline{\mathfrak{DY}}(C,\C)$ is a complete and separated filtered Cartier category, and thus also $\mathfrak{M}_-(C,\C)$ is, since it is a full subcategory. By the definition of the morphism 
$\delta_-\colon C_-\to C_-\ten C$ we get that $C_-\in \mathfrak{M}_-(C,\C)$. The infinitesimal braiding in $\mathfrak{M}_-(C,\C)$ is a sum of two morphisms which both contain one coaction, which are by definition in filtration degree $1$, so the statement is shown. 
\end{proof}

\begin{remark}
    Using the notions from above and from Sections \ref{Subsubsec: FunctorialityI} and \ref{Subsubsec: FunctorialityII}, one can check that the assignments 
\begin{align*}
\mathrm{qCoPoissH}(\C) &\to \mathrm{qCartCat}^3(\C), \qquad   C\mapsto (\mathfrak{M}_-(C,\C),C_-,\mathfrak{F}_-) \\
\mathrm{cqPoissH}(\C)&\to \mathrm{qCartCat}_3(\C) , \qquad P\mapsto (\mathfrak{M}_+(P,\C),P_+,\mathfrak{F}_+) \\
\mathrm{dqHopf}(\C) &\to \mathrm{qbMonCat}^3(\C), \quad  H\mapsto (\mathscr{M}_-(H,\C),H_-,\mathcal{F}_-) \\
\mathrm{cdqHopf}(\C) & \to \mathrm{qbMonCat}_3(\C), \quad K\mapsto (\mathscr{M}_+(K,\C),K_+,\mathcal{F}_+)
\end{align*}
are functorial in the same way as we described in Sections \ref{Subsubsec: FunctorialityI} and \ref{Subsubsec: FunctorialityII}. Moreover, one can check that the functors $\check{\mathfrak{S}}_-$, $\check{\mathfrak{S}}_+$, $\Se_-$ and $\Se_+$ respectively map in $\mathrm{qCoPoissH}(\C)$, $\mathrm{cqPoissH}(\C)$ , $\mathrm{dqHopf}(\C)$ or $\mathrm{cdqHopf}(\C)$ if restricted to the subcategories $\mathrm{qCartCat}^3(\C)$, $\mathrm{qCartCat}_3(\C)$, $\mathrm{qbMonCat}^3(\C)$ or $\mathrm{qbMonCat}_3(\C)$, respectively.  
\end{remark}

\subsection{Quantization of coPoisson Hopf algebras and coquantization of Poisson Hopf algebras}
\label{sec-quantdequant-functors}
This section contains our main results, which consist of establishing a categorical equivalence 
between the categories of (co)quantizable (co)Poisson Hopf monoids and (co)dequantizable Hopf monoids in a suitable category and relating their Drinfeld-Yetter and Yetter-Drinfeld categories. 
\subsubsection{The quantization and dequantization functors}
\begin{definition}
Let $\C$ be a symmetric monoidal category and $[\C]$ be the associated category constructed in section \ref{section-aux-categories}.
\begin{enumerate}
\item Let $C$ be a quantizable coPoisson Hopf monoid. A Hopf monoid $H$ is called a quantization of $C$ if there is an isomorphism $\Psi : C \to H$ such that $[\Psi]: [C] \to [H]$ is an isomorphism of Hopf monoids satisfying
\[ \delta_C - (\Psi^{-1}\ten \Psi^{-1})\circ (\Delta_H-\sigma\circ\Delta_H)\circ \Psi \in \mathrm{F}^2\Hom(C,C\ten C).\]
\item Let $P$ be a coquantizable Poisson Hopf monoid. A Hopf monoid $K$ is called a coquantization of $P$ if  there is an isomorphism $\Psi : P \to K$ such that $[\Psi]: [P] \to [K]$ is an isomorphism of Hopf monoids satisfying 
\[ \{\cdot,\cdot\}_P - \Psi^{-1}\circ (\mu_K-\mu_K\circ \sigma)\circ (\Psi\ten \Psi) \in \mathrm{F}^2\Hom(P\ten P,P). \]
\end{enumerate}
\end{definition}
Our aim is to establish equivalences of categories
\begin{equation*}
\begin{tikzcd}
\mathrm{qCoPoissH}(\C) \arrow[rr, "\mathcal{Q}_-", shift left] &  & \mathrm{dqHopf}(\C) \arrow[ll, "\mathcal{D}_-", shift left] & \mathrm{cqPoissH}(\C) \arrow[rr, "\mathcal{Q}_+", shift left] &  & \mathrm{cdqHopf}(\C) \arrow[ll, "\mathcal{D}_+", shift left]
\end{tikzcd}
\end{equation*}
such that $\mathcal{Q}_-(C)$ is a quantization of $C$ (resp. $\mathcal{Q}_+(P)$ is a coquantization of $P$.)
Let us choose for this purpose a Drinfeld associator $(\Phi,\frac{1}{2})$ and the unique algebraic morphism of semigroups $g\colon\mathbb{K}\to\overline{GT}(\mathbb{K})$ from Theorem \ref{Thm: nicecurve}. We will omit the reference to $\frac{1}{2}$ of the Drinfeld associator in the following. 
Consider $C\in \mathrm{qCoPoissH}({\C})$ and $P\in \mathrm{cqPoissH}(\C)$ and the infinitesimally braided (co)monoidal functors from Theorem \ref{thm: (co)adaptedPoiss} $\F_-\colon \mathfrak{M}_-(C)\to \C$ and $\F_+\colon \mathfrak{M}_+(P)\to \C$ (which we already restricted to $\mathfrak{M}_-(C)$ and $\mathfrak{M}_+(P)$, respectively). Recall that $\F_-$ is $C_-$-adapted  and that $\F_+$ is $P_+$-coadapted. 
Applying Proposition \ref{prop: Severa}, we obtain the braided comonoidal (resp. braided monoidal) functor
\begin{align}
\label{eq:functor.f.psi}
(\F_-)_{\Phi}\colon \big(\mathfrak{M}_-(C)\big)_\Phi\to \C \qquad \text{and} \qquad 
(\F_+)_{\Phi}\colon \big(\mathfrak{M}_+(P)\big)_\Phi\to \C \ .
\end{align}
We define now 
\begin{align}
\label{Eq:DefQ}
    \mathcal{Q}_-(C):=\gamma_{C_-}^{-}\big((\F_{-})_\Phi(C_-\ten C_-)\big) \qquad \text{and} \qquad \mathcal{Q}_+(P):=\gamma_{P_+}^{+}\big((\F_{+})_\Phi(P_+\ten P_+)\big)
\end{align}
where $\gamma_{C_-}^{-}\colon \F_{-}(C_-\ten C_-)\to C$ (resp. $\gamma_{P_+}^{+}\colon \F_{+}(P_+\ten P_+)\to P$) is the isomorphism induced by the natural transformation $\gamma^{-}\colon \F_-(C_-\ten - )\implies U$
(resp. $\gamma^{+}\colon \F_+(-\ten P_+)\implies U$) from Theorem \ref{thm: (co)adaptedPoiss}, where we denote by $U\colon \mathfrak{M}_-(C)\to \C$ (resp. $U\colon \mathfrak{M}_+(P)\to \C$) the forgetful functor. We define the Hopf monoid structure on $(\F_-)_{\Phi}(C_-\ten C_-)$ (resp. $(\F_+)_{\Phi}(P_+\ten P_+)$) by Theorem \ref{theorem-hopf-algebra} for the  adapted (resp. coadapted) case. 
\begin{remark}
At first glance, the assignments of $\mathcal{Q}_\pm$ seem to be a bit complicated. This is motivated by the fact that we want the quantizable coPoisson Hopf monoid $C$ (resp. coquantizable Poisson Hopf monoid $P$) and the Hopf monoid $\mathcal{Q}_-(C)$ (resp. $\mathcal{Q}_+(P)$) to have the same underlying object in $\C$.  
\end{remark}

\begin{theorem}
\label{thm: Qunatizationcodomain}
Under the assumptions and notations as above, we have the following.
\begin{enumerate}
\item  For $C\in \mathrm{qCoPoissH}(\C)$ we get that $\mathcal{Q}_-(C)\in \mathrm{dqHopf}(\C)$ and $Q_-(C)$ is a quantization of $C$. Moreover, the induced Hopf monoid structures satisfy $[C]=[\mathcal{Q}_-(C)]$.
\item  For $P\in \mathrm{cqPoissH}(\C)$ we get that $\mathcal{Q}_+(P)\in \mathrm{cdqHopf}(\C)$ and $\mathcal{Q}_+(P)$ is a coquantization of $P$.
Moreover, the induced Hopf monoid structures satisfy $[P]=[\mathcal{Q}_+(P)]$.
\end{enumerate}
\end{theorem}

\begin{proof}
We only prove the first statement, since the second follows along the same line. By property $(i)$ of Definition \ref{Def: DA}, we have that the associativity constraint of $(\mathfrak{M}_{-}(C))_\Phi$, satisfies
\begin{align*}
a^\Phi_{U,V,W}-a_{U,V,W} \in \mathrm{F}^2\Hom((U\ten V)\ten W, U\ten (V\ten W))  \end{align*}
whence the braiding satisfies
\begin{align}
\label{Eq:Expansionbraiding}
\sigma^\Phi_{V,W} - \sigma_{V,W} - \tfrac{1}{2}\bigg(\sigma_{V,W}\circ t_{V,W}\bigg)  \in \mathrm{F}^2\Hom(V\ten W, W\ten V).  
\end{align}
This implies in particular that the Hopf monoid structure of $\mathcal{Q}_-(C)$ is the same as the one of $\F_-(C\ten C)$ up to filtration degree $1$, and since $\F_-(C\ten C)$ is cocommutative, then $\mathcal{Q}_-(C)$ is dequantizable. Moreover, the first-order term can be computed via Equation \eqref{Eq:Expansionbraiding}, which gives, by the respective formula from Theorem \ref{Thm: constrPoissonHopf}, the coPoisson cobracket. 
\end{proof}

From now on, we shall denote the assignments as above by $\mathcal{Q}_-^\Phi$ and $\mathcal{Q}_+^\Phi$ to indicate that they depend on the choice of a Drinfeld associator. 

Next, let $H\in \mathrm{dqHopf}(\C)$ and $K\in \mathrm{cdqHopf}(\C)$ and consider
the functors $\mathcal{F}_-\colon \M_-(H)\to \C$ and $\mathcal{F}_+\colon \M_+(K)\to \C$ from Theorem \ref{Thm: (co)adaptedHopffunctors} (restricted to $\M_-(H)$ and $\M_+(K)$, respectively). Recall that $\mathcal{F}_-$ is comonoidal and $H_-$-adapted  and that $\mathcal{F}_+$ is monoidal and $K_+$-coadapted. 
Recalling Theorem \ref{Thm: admissibleSubcategories}, we can apply Theorem \ref{thm: qasisygivesinfitesimal braiding} to obtain functors
\begin{align}
(\mathcal{F}_-)_{g(0)}\colon \big(\M_-(H)\big)_{g(0)}\to \C \qquad \text{and} \qquad 
(\mathcal{F}_+)_{g(0)}\colon \big(\M_+(K)\big)_{g(0)}\to \C
\end{align}
which are respectively infinitesimally braided comonoidal $H_-$-adapted and infinitesimally braided monoidal $K_+$-adapted. We set
\begin{align}
\label{Eq:DefD}
    \mathcal{D}_-(H):=\zeta_{H_-}^{-}\big((\mathcal{F}_-)_{g(0)}(H_-\ten H_-)\big) \qquad \text{and} \qquad \mathcal{D}_+(K):=\zeta_{K_+}^{+}\big((\mathcal{F}_+)_{g(0)}(K_+\ten K_+)\big)
\end{align} 
where $\zeta_{H_-}^{-}\colon \mathcal{F}_{-}(H_-\ten H_-)\to H$ (resp. $\zeta_{K_+}^{+}\colon \mathcal{F}_{+}(K_+\ten K_+)\to K$) is the isomorphism induced by the natural transformation
$\zeta^{-}\colon \mathcal{F}_-(H_-\ten - )\implies U$ (resp. $\zeta^{+}\colon \mathcal{F}_+(-\ten K_+)\implies U$)  from Theorem \ref{Thm: (co)adaptedHopffunctors}, where we denote by $U\colon \M_-(H)\to \C$ (resp. $U\colon \M_+(K)\to \C$) the forgetful functor. We define the coPoisson (resp. Poisson) Hopf monoid structure on $(\mathcal{F}_-)_{g(0)}(H_-\ten H_-)$ (resp. $(\mathcal{F}_+)_{g(0)}(K_+\ten K_+)$) by   Theorem \ref{Thm: constrPoissonHopf} for the adapted (resp. coadapted) case.  

\begin{theorem}
\label{thm: dequantizationcodomain}
Under the assumptions and notations as above, we have the following.
\begin{enumerate}
\item  For $H\in \mathrm{dqHopf}(\C)$ we get that $\mathcal{D}_-(H)\in \mathrm{qCoPoissH}(\C)$. Moreover, the induced Hopf monoid structures satisfy $[H]=[\mathcal{D}_-(H)]$.
\item  For $K\in \mathrm{cdqHopf}(\C)$ we get that
$\mathcal{D}_+(K)\in \mathrm{cqPoissH}(\C)$.  Moreover, the induced Hopf monoid structures satisfy $[K]=[\mathcal{D}_+(K)]$. 
\end{enumerate}
\end{theorem}

\begin{proof}
    The statement follows again by the explicit formulas of the Hopf structures and a careful filtration degree counting as in Theorem \ref{thm: Qunatizationcodomain}. 
\end{proof}

Also in this case, we denote the above assignments by $\mathcal{D}_\pm^\Phi$ to indicate their dependencies on the unique algebraic morphism $g\colon \mathbb{K}\to \overline{GT}(\mathbb{K})$ depending on the Drinfeld associator $\Phi$.  

\begin{theorem}
\label{Thm: QandDQfunctors}
The following assignments are functors 
\begin{align}
\mathcal{Q}^\Phi_- &\colon \mathrm{qCoPoissH}(\C) \to \mathrm{dqHopf}(\C), \quad &C \mapsto \mathcal{Q}^\Phi_-(C), \quad &f \mapsto \mathcal{Q}^\Phi_-(f)=f\\
\mathcal{Q}^\Phi_+ &\colon \mathrm{cqPoissH}(\C) \to \mathrm{cdqHopf}(\C), \quad &P \mapsto \mathcal{Q}^\Phi_+(P), \quad &f \mapsto \mathcal{Q}^\Phi_+(f)=f\\
\mathcal{D}^\Phi_- &\colon \mathrm{dqHopf}(\C) \to \mathrm{qCoPoissH}(\C), \quad &H \mapsto \mathcal{D}^\Phi_-(H), \quad &f \mapsto \mathcal{D}^\Phi_-(f)=f\\
\mathcal{D}^\Phi_+ &\colon \mathrm{cdqHopf}(\C)  \to \mathrm{cqPoissH}(\C), \quad &K \mapsto \mathcal{D}^\Phi_+(K), \quad &f \mapsto \mathcal{D}^\Phi_+(f)=f.
\end{align}
\end{theorem}

\begin{proof}
By definition of all of the functors, the underlying objects in $\C$ are unchanged, so the assignments of the underlying morphisms are well-defined. 
The Hopf structures (resp. (co)Poisson Hopf structures) are built upon infinite sums of algebraic (i.e. concatenations) combinations of the old ones. This means in particular that every morphism respects these combinations. Alternatively, one can use the results contained in Sections \ref{Subsubsec: FunctorialityI} and \ref{Subsubsec: FunctorialityII} to see that all the involved assignments on the level of objects are naturally isomorphic to concatenations of functors and compute the assignment on the level of morphisms, which turns out to be the identity.
\end{proof}
\subsubsection{The quantization-dequantization equivalence}
Fix
$C\in \mathrm{qCoPoissH}(\C)$, $P\in \mathrm{cqPoissH}(\C)$, $H\in \mathrm{dqHopf}(\C)$ and $K\in \mathrm{cdqHopf}(\C)$, and consider the induced  strongly monoidal functors 
 \begin{align}
\Gamma_{(\F_{-})_\Phi}&\colon \big(\mathfrak{M}_-(C,\C)\big)_\Phi\to \underline{\Y\D}\big(\mathcal{Q}^\Phi_-(C),\C \big)\\
\Xi_{(\F_{+})_\Phi}&\colon \big(\mathfrak{M}_+(P,\C)\big)_\Phi\to \overline{\Y\D}\big(\mathcal{Q}^\Phi_+(P),\C\big) \\
\Upsilon_{(\mathcal{F}_{-})_{g(0)}}&\colon \big(\M_-(H,\C)\big)_{g(0)}\to \underline{\mathfrak{DY}} \big(\mathcal{D}^\Phi_-(H),\C\big) \\
\Omega_{(\mathcal{F}_{+})_{g(0)}}&\colon \big(\M_+(K,\C)\big)_{g(0)}\to \overline{\mathfrak{DY}}\big(\mathcal{D}^\Phi_+(K),\C\big)
 \end{align}
where we used the Hopf algebra isomorphisms from \ref{Eq:DefQ} and \ref{Eq:DefD}:
$\mathcal{Q}^\Phi_-(C)\cong (\F_-)_{\Phi}(C_-\ten C_-)$, $\mathcal{Q}^\Phi_+(P)\cong (\F_+)_{\Phi}(P_+\ten P_+)$, $\mathcal{D}^\Phi_-(H)\cong (\mathcal{F}_-)_{g(0)}(H_-\ten H_-)$, and $\mathcal{D}^\Phi_+(K)\cong (\mathcal{F}_+)_{g(0)}(K_+\ten K_+)$ in order to identify the respective target categories. One can show -- by carefully counting the filtration degrees of the formulas for actions, coactions, etc.  -- that one can restrict the functors to the target categories
\begin{align}
\Gamma_{(\F_{-})_\Phi}&\colon \big(\mathfrak{M}_-(C,\C)\big)_\Phi\to \M_-\big(\mathcal{Q}^\Phi_-(C),\C\big) \label{eq:equivalence_1}\\
\Xi_{(\F_{+})_\Phi}&\colon \big(\mathfrak{M}_+(P,\C)\big)_\Phi\to \M_+\big(\mathcal{Q}^\Phi_+(P),\C\big)\label{eq:equivalence_2} \\
\Upsilon_{(\mathcal{F}_{-})_{g(0)}}&\colon \big(\M_-(H,\C)\big)_{g(0)}\to \mathfrak{M}_-\big(\mathcal{D}^\Phi_-(H),\C\big) \label{eq:equivalence_3}\\
\Omega_{(\mathcal{F}_{+})_{g(0)}}&\colon \big(\M_+(K,\C)\big)_{g(0)}\to \mathfrak{M}_+\big(\mathcal{D}^\Phi_+(K),\C\big). \label{eq:equivalence_4}
\end{align}

This observation will help us to prove the following 
\begin{theorem}
\label{Q-DQ-Eq}
The functors $\mathcal{Q}_-^\Phi\circ \mathcal{D}_-^\Phi$, $\mathcal{Q}_+^\Phi\circ \mathcal{D}_+^\Phi$, $\mathcal{D}_-^\Phi\circ \mathcal{Q}_-^\Phi$ and 
$\mathcal{D}_+^\Phi\circ \mathcal{Q}_+^\Phi$ are all isomorphic to their respective identities. 
\end{theorem}

\begin{proof}
We prove the statement for the composition $\mathcal{D}_-^\Phi\circ\mathcal{Q}_-^\Phi$, since the other cases follow the same lines.
Let $C\in \mathrm{qCoPoissH}(\C)$ and recall that $\mathcal{Q}^\Phi_-(C)=\gamma_{C_-}^{-}\big((\F_{-})_\Phi(C_-\ten C_-)\big)$, where $(\F_{-})_\Phi\colon (\mathfrak{M}_-(C))_\Phi\to \C$ is the functor defined in Equation \eqref{eq:functor.f.psi} and $\gamma^-$ is the natural transformation defined by Theorem \ref{thm: (co)adaptedPoiss}. Using part $(iv)$ of Theorem \ref{Prop:natTraf}, we get a natural transformation $n^-\colon \mathcal{F}_-\circ \Gamma_{(\F_{-})_\Phi}\implies (\F_{-})_\Phi$, with the property that $n^-_{C_-\ten X}$ is an isomorphism for all objects $X$ in $\mathfrak{M}_-(C)_\Phi$. Due to the discussion above, we can restrict and corestrict the functor $\Gamma_{(\F_{-})_\Phi} $ to obtain 
$\Gamma_{(\F_{-})_\Phi} \colon \mathfrak{M}_-(C)_\Phi\to \M_{-}(\mathcal{Q}^\Phi_-(C))$ and, similarly, we consider the restriction of the functor $\mathcal{F}_-\colon  \M_{-}(\mathcal{Q}^\Phi_-(C))\to \C$. Applying Theorem \ref{thm: qasisygivesinfitesimal braiding} to both functors and the natural transformation, we obtain two $C_-$-adapted infinitesimally braided comonoidal functors $\big((\F_{-})_\Phi\big)_{g(0)} = \F_-$ and $(\mathcal{F}_-)_{g(0)}\circ \big(\Gamma_{(\F_{-})_\Phi}\big)_{g(0)}$ and a natural transformation $n_{g(0)}^-\colon (\mathcal{F}_-)_{g(0)}\circ \big(\Gamma_{(\F_{-})_\Phi}\big)_{g(0)}\implies \F_-$ with the property that $(n^-_{g(0)})_{C_-\ten X}$ are isomorphisms. We can now use Proposition \ref{Prop:nattrafoPoissHopf} to get that the resulting coPoisson Hopf algebras are isomorphic via 
\begin{align*}
((n^-_{g(0)})_{C_-\ten C_-})^{-1}\colon \F_-(C_-\ten C_-)&\to \bigg((\mathcal{F}_{-})_{g(0)}\circ
    \big(\Gamma_{(\F_{-})_\Phi}\big)_{g(0)}\bigg)(C_-\ten C_-)\\
    &\cong 
(\mathcal{F}_{-})_{g(0)}\bigg(\big(\Gamma_{(\F_{-})_\Phi}\big)_{g(0)}(C_-)\big)\ten \big(\Gamma_{(\F_{-})_\Phi}\big)_{g(0)}(C_-)\bigg)\\
    &\cong (\mathcal{F}_{-})_{g(0)}(\mathcal{Q}^\Phi_-(C)_-\ten \mathcal{Q}^\Phi_-(C)_-).
\end{align*}
To end the proof it suffices to observe that the left-hand side of this equation is isomorphic to
$C$, whence the right-hand side is isomorphic to $\mathcal{D}^\Phi_-(\mathcal{Q}^\Phi_-(C))$. Note that all the isomorphisms can be explicitly constructed, inducing the isomorphism $C\to \mathcal{D}^\Phi_-(\mathcal{Q}^\Phi_-(C))$ as the natural transformation $\id \implies\mathcal{D}_-^\Phi\circ \mathcal{Q}_-^\Phi$. 
\end{proof}

As a last part of this section, we want to compare the categories $\mathfrak{M}_\pm$ and $\M_{\pm}$ of the Hopf algebras and their quantizations and dequantizations. 

\begin{theorem}
\label{Thm: RepCatsEq}
For every choice of
$C\in \mathrm{qCoPoissH}(\C)$, $P\in \mathrm{cqPoissH}(\C)$, $H\in \mathrm{dqHopf}(\C)$ or $K\in \mathrm{cdqHopf}(\C)$,  the induced strongly monoidal functors \eqref{eq:equivalence_1}-\eqref{eq:equivalence_4}
are equivalences of categories. 
\end{theorem}

\begin{proof}
Let $C\in \mathrm{qCoPoissH}(\C)$. Then, following the same lines as the previous proof, attached to $(\F_{-})_\Phi\colon \mathfrak{M}_-(C)_{\Phi}\to \C$, we have the functor 
$\mathcal{F}_-\circ \Gamma_{(\F_{-})_\Phi}$ and  natural transformation $n^-\colon \mathcal{F}_-\circ \Gamma_{(\F_{-})_\Phi}\implies (\F_{-})_\Phi$, with the property that $n^-_{C_-\ten X}$ is an isomorphism for all $X$. Applying now Theorem \ref{thm: qasisygivesinfitesimal braiding} we obtain two $C_-$-adapted infinitesimally braided comonoidal functors $\big((\F_{-})_\Phi\big)_{g(0)} = \F_-$ and $(\mathcal{F}_-)_{g(0)}\circ \big(\Gamma_{(\F_{-})_\Phi}\big)_{g(0)}$ and a natural transformation $n^-_{g(0)}\colon (\mathcal{F}_-)_{g(0)}\circ \big(\Gamma_{(\F_{-})_\Phi}\big)_{g(0)}\implies \F_-$, still with the property that
$(n^-_{g(0)})_{C_-\ten X}$ are isomorphisms. By part $(ii)$ of Theorem \ref{thm: extBordemannPoissII}  we get a natural isomorphism

\begin{align*}
\Upsilon_{\F_-} \cong \mathcal{I}_{((n^-_{g(0)})_{C_-\ten C_-})^{-1}}\circ \Upsilon_{(\mathcal{F}_-)_{g(0)}\circ \big(\Gamma_{(\F_{-})_\Phi}\big)_{g(0)}}\cong \mathcal{I}_{((n^-_{g(0)})_{C_-\ten C_-})^{-1}}\circ \Upsilon_{(\mathcal{F}_-)_{g(0)}} \circ \big(\Gamma_{(\F_{-})_\Phi}\big)_{g(0)}
\end{align*}
where for the second isomorphism we used part $(iii)$ of Theorem
\ref{thm: extBordemannPoissII}. By part $(i)$ of Theorem \ref{thm: extBordemannPoissII}, we know that  $\Upsilon_{\F_-}$ is an equivalence of categories, which implies that  $(\Gamma_{(\F_{-})_\Phi})_{g(0)}$ has a left inverse, so does then 
$\Gamma_{(\F_{-})_\Phi} = \big((\Gamma_{(\F_{-})_\Phi})_{g(0)}\big)_\Phi$
and, with the same argument, one shows that it has a right inverse. In fact, one can already guess that it is given, up to canonical equivalences, by 
$(\Upsilon_{(\mathcal{F}_{-})_{g(0)}})_\Phi$. The remaining statements follow exactly the same lines. 
\end{proof}

\section{Applications: (de)quantization of Lie bialgebras and Tamarkin's proof of Deligne's conjecture}
\label{Sec:Appl}
The aim of this section is to stress that the results illustrated in the previous sections have rather significant applications. In fact, we are able to reprove the (de)quantization of Lie bialgebras from Etingof and Kazhdan \cite{EK1} and, more recently, from \v{S}evera \cite{Sev16}. Moreover, we will also discuss a \emph{dual} version of this (de)quantization procedure, where the universal enveloping coalgebra is the center of interest. 

Another application is Tamarkin's celebrated proof of Deligne's conjecture \cite{tamarkin}, in which the key technique is the dequantization functor. While Tamarkin only relied on its existence, we provide a very explicit construction.
Another advantage of our construction is that it is independent (contrary to Tamarkin's approach) of formal power series.

Since most of these applications and techniques are well-known, we only sketch most of the proofs. 
\subsection{(De)quantization of Lie bialgebras}
\subsubsection{(De)quantization of Lie bialgebras \`a la Etingof-Kazhdan}
\label{SubSec: DQLiebialgebras}
We now introduce Lie bialgebras \cite{Dri82} and Drinfeld-Yetter modules \cite{ATL18}, also called bimodules in \cite{EK2} and \cite{Sev16}.
\begin{definition}
A Lie bialgebra is a triple $(\b, [\cdot,\cdot],\delta)$, where $(\b, [\cdot,\cdot])$ is a Lie algebra, $(\b,\delta)$ is a Lie coalgebra and the cocycle condition $\delta([x,y]) - x \cdot \delta(y) +  y \cdot \delta(x) =0$ is satisfied for all $x,y \in \b$.
\end{definition}
Lie bialgebras form a symmetric monoidal category, which we denote by $\L\B\A$. Note that, even though we shall denote it by $\ten$, here the tensor product is the direct sum of vector spaces. The notion of a Lie bialgebra extends trivially to the setting of any preadditive symmetric monoidal category.

Given a Lie bialgebra $(\mathfrak{b},[\cdot,\cdot],\delta)$, we can consider its universal enveloping algebra $\U(\b)$ and define the map
$\delta_{\U(\b)}\colon \U(\b)\to\U(\b)\ten \U(\b)$ by 
     \begin{align}
     \label{Eq: LiebialggivescPH}
         \delta_{\U(\b)}(x_1\cdots x_k)=\sum_{i=1}^k\Delta(x_1\cdots x_{i-1})\delta(x_i)\Delta(x_{i+1}\cdots x_k).
     \end{align}
It is well-known \cite{Dri86} that this endows $\U(\b)$ with the structure of a coPoisson Hopf algebra.  
Let us now consider the category $\mathrm{tfMod}_{\mathbb{K }[[\hbar]]}$, which is a filtered (by $\hbar$-degrees, see Example \ref{Ex: PSareFilt}) complete, separated, and $k$-good enough monoidal category. Nevertheless, it is not $c$-good enough. In what follows, we shall denote the completed tensor product by $\ten$.

Let us consider a topologically free Lie bialgebra $(\b_\hbar,[\cdot,\cdot]_\hbar,\delta_\hbar)$, i.e. a Lie bialgebra in $\mathrm{tfMod}_{\mathbb{K}[[\hbar]]}$. 
This means in particular that 
$\b_\hbar\cong \b[[\hbar]]$ for some vector space $\b$
and that $(\b,[\cdot,\cdot]_0,\delta_0)$ is a Lie bialgebra in $\mathrm{Vect}$.
Note that it is not a priori clear that the universal enveloping algebra exists in the category of topologically free algebras, but one can construct it in the category of complete and separated $\mathbb{K}[[\hbar]]$-modules and can show that it is topologically free by using a version of the Poincar\'e-Birkhoff-Witt (PBW) Theorem. Moreover, one can show that $\U(\b_\hbar)\cong \U(\b)[[\hbar]]$ as topologically free coalgebras, and one can define the coPoisson Hopf structure induced by $\delta_\hbar$ also via formula  \eqref{Eq: LiebialggivescPH} to get a topologically free coPoisson Hopf algebra  $\U(\b_\hbar)$. If we now assume that $\delta_0=0$, we get 
that the coPoisson structure $\delta_{\U(\b_\hbar)}$ is in filtration degree $1$ which means that $\delta_{\U(\b_\hbar)}\in \hbar\Hom(\U(\b_\hbar),\U(\b_\hbar)\ten\U(\b_\hbar))$. 
This means that the coPoisson Hopf algebra in this case is quantizable in the sense of Definition \ref{Def: goodHopf} in the category $\mathrm{tfMod}_{\mathbb{K}[[\hbar]]}$. 
Let us denote the category of topologically free Lie bialgebras such that $\delta_0=0$ by $\L\B\A_0(\mathrm{tfMod}_{\mathbb{K}[[\hbar]]})$, and its objects by $(\b_\hbar^-,[\cdot,\cdot]_\hbar,\delta_\hbar)$. Taking the universal enveloping algebra establishes a functor $\U: \L\B\A_0(\mathrm{tfMod}_{\mathbb{K}[[\hbar]]})\to \mathrm{qCoPoissH}(\mathrm{tfMod}_{\mathbb{K}[[\hbar]]})$, which is not an equivalence of categories. However, if we adequately restrict the target category, it does. For this purpose, we introduce the following

\begin{definition}
A topologically free coalgebra $C$ is called quasi-conilpotent if $C/\hbar C$ is conilpotent. A topologically free Hopf algebra is called quasi-conilpotent if its underlying coalgebra is quasi-conilpotent. 
\end{definition}
This definition will lead us to a generalization of the well-known Milnor-Moore Theorem, which is again an implication by the PBW-Theorem. 

\begin{lemma}
\label{Lem: Milnor-Moore}
Let $H$ be a quasi-conilpotent topologically free Hopf algebra. 
\begin{enumerate}
\item If $H$ is dequantizable, then $ H/\hbar H \cong \U(\g) $, where $\g$ is the Lie algebra of all primitive elements of $H/\hbar H$. 
\item If $H$ is cocommutative, then its primitive elements $\mathrm{Prim}(H)$ form a topologically free Lie algebra and  $H\cong \U(\mathrm{Prim}(H))$ as topologically free Hopf algebras.
\end{enumerate}
\end{lemma}

\begin{proof}
$H$  being dequantizable in $\mathrm{tfMod}_\mathbb{K[[\hbar]]}$ implies in particular that $H/\hbar H$ is cocommutative. Moreover, it is conilpotent by hypothesis. Using the Milnor-Moore Theorem we get the claim. 
For part $(ii)$ we notice that the primitives are the kernel of the map  
\begin{align*}
H \to H\ten H, \quad  x\mapsto \Delta(x)-1\ten x-x\ten 1 
\end{align*}
and since the category of topologically free modules is $k$-good enough, $\mathrm{Prim}(H)$ is topologically free as well. 
Using part one of the statement, we get $H\cong \U(\g)[[\hbar]]$ as topologically free modules. Using the classical PBW-theorem, one sees that 
$\U(\g)[[\hbar]]$ is cofree as a topologically free coalgebra and the claim follows.   
\end{proof}
Note that part $(i)$ of Lemma \ref{Lem: Milnor-Moore} shows that quasi-conilpotent dequantizable Hopf algebras are the same as what is called in the literature \emph{quantized universal enveloping algebras}   (see e.g. \cite{EK1}). 

A coPoisson Hopf structure on $H$ always induces a Lie bialgebra structure on the primitive elements. 
This means that there is an equivalence of categories
\begin{center}
\begin{tikzcd}
{\L\B\A_0(\mathrm{tfMod}_{\mathbb{K}[[\hbar]]})} \arrow[rr, "\U", shift left] &  & {\mathrm{qCoPoissH}_{qcn}(\mathrm{tfMod}_{\mathbb{K}[[\hbar]]})} \arrow[ll, "\mathrm{Prim}", shift left]
\end{tikzcd}
\end{center}
where the subscript \emph{qcn} means to restrict to quasi-conilpotent. This, together with Theorem \ref{Q-DQ-Eq}, implies that we have the following concatenation of equivalences of categories 
\begin{center}
\begin{tikzcd}
{\L\B\A_0(\mathrm{tfMod}_{\mathbb{K}[[\hbar]]})} \arrow[rr, "\U", shift left] &  & {\mathrm{qCoPoissH}_{qcn}(\mathrm{tfMod}_{\mathbb{K}[[\hbar]]})} \arrow[ll, "\mathrm{Prim}", shift left] \arrow[rr, "\mathcal{Q}^\Phi_-", shift left] &  & {\mathrm{dqHopf}_{qcn}(\mathrm{tfMod}_{\mathbb{K}[[\hbar]]})} \arrow[ll, "\mathcal{D}^\Phi_-", shift left].
\end{tikzcd}
\end{center}
The diagram above describes a new explicit way to construct the dequantization functor of quantized universal enveloping algebras. This result was previously achieved by Etingof and Kazhdan \cite{EK2} and later by Enriquez and Etingof \cite{EnrEt} by using a version of the \emph{Hensel's lemma} \cite{Zar}. 

\begin{remark}
\label{Rem: tModcMod}
As cited before, the category $\mathrm{tfMod}_{\mathbb{K}[[\hbar]]} $ does not have cokernels. However, it is a full subcategory of $\mathrm{csMod}_{\mathbb{K}[[\hbar]]}$,  which has cokernels, i.e. it is $c$-good enough. Therefore, we can apply the functors $\mathcal{Q}_-^\Phi$ and $\mathcal{D}_-^\Phi$ in this category, and using the fact that the functors do not change the underlying objects, we can restrict the domains and codomains to topologically free $\mathbb{K}[[\hbar]]$-modules. 
Moreover, quantization and dequantization do not change the property of a topologically free Hopf algebra to be quasi-conilpotent, since if we apply one of these assignments to a topologically free quantizable coPoisson Hopf algebra (resp. dequantizable Hopf algebra) $H$, we get 
        \begin{align}
            \mathcal{Q}^\Phi_-(H)/\hbar \mathcal{Q}^\Phi_-(H)\cong H/\hbar H \qquad \text{ \big(resp. } \mathcal{D}^\Phi_-(H)/\hbar \mathcal{D}^\Phi_-(H)\cong H/\hbar H \big)
        \end{align}
    using Theorem \ref{thm: Qunatizationcodomain} (i) (resp. Theorem \ref{thm: dequantizationcodomain} (i)). 
\end{remark}

Next, for a fixed topologically free Lie bialgebra $\b_\hbar$, we introduce the following
\begin{definition}
A (left-right) Drinfeld-Yetter module over $\b_\hbar$ is a triple $(V, \pi,\rho)$, where $V$ is a topologically free $\mathbb{K}[[\hbar]]$-module together with a structure of left Lie module $(V,\pi)$ and of right Lie comodule $(V, \rho)$ over $\b_\hbar$, satisfying the following compatibility condition:
\begin{align}
\rho \circ \pi = (\pi \ten \id) \circ (\id \ten &\rho) + (\id \ten [\cdot, \cdot]) \circ \tau_{(12)} \circ (\id \ten \rho) + (\pi \ten \id) \circ \tau_{(23)} \circ (\delta \ten \id). \label{eq:DY-three}
\end{align}
\end{definition}
We shall denote the category of Drinfeld-Yetter modules over $\b_\hbar$ by $\D \Y(\b_\hbar)$. It is well-known that $\D \Y(\b_\hbar)$ is a symmetric Cartier category: if $(V, \pi_V, \rho_V)$ and $(W, \pi_W, \rho_W)$ are two objects in $\D \Y(\b_\hbar)$, then the Drinfeld-Yetter module structure of $V \ten W$ is 
\begin{align}
\label{Eq: tensorprodaction}
\pi_{V \ten W} &= (\pi_V \ten \id) + (\id_V \ten \pi_W) \circ (\tau_{\b,V} \ten \id_W) \\
\label{Eq: tensorprodcoaction}
\rho_{V \ten W} &= (\id_V \ten \rho_W) + (\id_V \ten \tau_{\b,W}) \circ (\rho_V \ten \id_W).
\end{align}
The braiding of $\D \Y(\b_\hbar)$ is the tensor flip, and the infinitesimal braiding is 
\begin{equation}
\label{Eq: Infbaiding}
t_{V,W} =  -r_{V,W} -  \tau_{W,V} \circ r_{W,V} \circ \tau_{V,W}
\end{equation}
where $r_{V,W} = (\id_V \ten \pi_W) \circ (\rho_V \ten \id_W)$. 

\begin{lemma}
\label{lemma_isom_DY_to_DY}
The categories $\underline{\mathfrak{DY}}(\U(\b_\hbar))$ and $\D\Y(\b_\hbar)$ are isomorphic symmetric Cartier categories. 
\end{lemma}

\begin{proof}
Let us denote by $i\colon \b_\hbar \to \U(\b_\hbar)$ the canonical inclusion of the Lie algebra in its universal enveloping algebra. Consider the functor
\begin{align*}
\D\Y(\b_\hbar) \to \underline{\mathfrak{DY}}(\U(\b_\hbar)), \quad  (V,\pi_V,\rho_V)\mapsto (V,\mu_V, (\id\ten i)\circ \rho_V), \qquad f \mapsto f,
\end{align*}
where $\mu_V$ is the unique algebra action $\mu_V\colon \U(\b_\hbar)\ten V\to V$ defined via the universal property of $\U(\b_\hbar)$ and the Lie algebra action $\pi_V$. An easy computation shows that $(V,\mu_V, (\id\ten i)\circ \rho_V)$ is an object in $\underline{\mathfrak{DY}}(\U(\b_\hbar))$. On the other hand, for a given $(V,\mu_V,\delta_V)\in  \underline{\mathfrak{DY}}(\U(\b_\hbar))$, one can easily show (by using Equation \ref{Eq: DYI-4}) that the map  $\delta_V$ corestricts $\delta_V\colon V\to V\ten \mathrm{Prim}(\U(\b_\hbar))$. This means that there is a unique map $\pi_V\colon V\to V\ten \b$, such that $\rho_V=(\id\ten i)\circ \pi_V$, with which we define the assignment 
\begin{align*}
\underline{\mathfrak{DY}}(\U(\b_\hbar)) \to \D\Y(\b_\hbar), \quad (V,\mu_V,\delta_V)\mapsto (V, \mu_V\circ (i\ten \id),\pi_V) \qquad f \mapsto f.
\end{align*}
Again, one easily checks that  $(V, \mu_V\circ (i\ten \id),\pi_V)$ is in fact a Drinfeld-Yetter module. Moreover, it is easy to see that both of the assignments are infinitesimally braided strongly monoidal functors. 
\end{proof}
Having in mind this equivalence, we now want to achieve a dequantization result for Yetter-Drinfeld modules.
For a Lie bialgebra $\b_\hbar^-\in\L\B\A_0(\mathrm{tfMod}_{\mathbb{K}[[\hbar]]})$, we define the full subcategory 
    \begin{align*}
        \mathfrak{D}_-(\b_\hbar^-):=\{(V,\pi_V,\rho_V)\in \D\Y(\b_\hbar^-)\ | \ \rho_V=0+\mathcal{O}(\hbar)\}.  
    \end{align*}
Using Theorem \ref{Thm: RepCatsEq} and the same arguments as in Remark \ref{Rem: tModcMod}, we obtain the following equivalence for every $\b_\hbar^-\in \L\B\A_0(\mathrm{tfMod}_{\mathbb{K}[[\hbar]]})$:

\begin{center}
\begin{tikzcd}
\mathfrak{D}_-(\b_\hbar^-) \arrow[rr, , "Lemma \ \ref{lemma_isom_DY_to_DY}", shift left] &  & \mathfrak{M}_-\big(\U(\b_\hbar^-)\big) \arrow[ll, shift left] \arrow[rr, "Eq. \eqref{eq:equivalence_1}", shift left] &  & \M_-\big(\mathcal{Q}_-^\Phi(\U(\b^-_\hbar))\big)_{g(0)} \arrow[ll,"Eq. \eqref{eq:equivalence_3}", shift left],
\end{tikzcd}
\end{center}
where we see $\U(\b_\hbar^-)$ now as a quantizable coPoisson Hopf algebra. 
This equivalence was also proven by Appel and Toledano Laredo \cite{ATL18}, but with completely different techniques. 

\subsubsection{Universal enveloping coalgebras and (de)coquantization of Lie bialgebras}\label{SubSec: UECoalgebras}
The usual way to quantize a Lie bialgebra is to deform its corresponding universal enveloping algebra. However, a (conilpotent) Lie bialgebra also has the universal enveloping coalgebra attached to it. In fact, in this section, we see that their (de)quantization can be interesting. Since the universal enveloping coalgebra is not widely known, we decided to include its construction here. We will always work in the context of conilpotent Lie coalgebras, since the notions are more traceable and the similarities to the universal enveloping algebra are more evident. All the unproved statements in the following can be easily shown by dualizing the well-known corresponding results regarding universal enveloping algebras.

\begin{definition} 
We call a Lie coalgebra $(\mathfrak{c}, \delta)$ conilpotent if there is an exhaustive filtration
\begin{align*}
\{0\}\subseteq \mathrm{F}^1\mathfrak{c}\subseteq \mathrm{F}^2\mathfrak{c}\subseteq \dots =\mathfrak{c} \qquad \text{such that} \qquad
\delta(\mathrm{F}^k\mathfrak{c})\subseteq 
\bigoplus_{i+j=k} \mathrm{F}^i\mathfrak{c} \otimes   \mathrm{F}^j \mathfrak{c}. 
\end{align*}
\end{definition}

\begin{remark}
Note that a conilpotent Lie coalgebra is always a union of finite-dimensional conilpotent Lie coalgebras, and finite-dimensional Lie coalgebras are conilpotent if their linear dual Lie algebra is nilpotent in the usual sense. 
\end{remark}

We denote by $\LC \A$ (resp. $\C \A$) the category of Lie coalgebras (resp. coassociative counital coalgebras).\\
For any coalgebra $(C, \Delta, \varepsilon)$, the map $\delta_\Delta \coloneqq \Delta - \tau \circ \Delta : C \to C \ten C$ makes $(C, \delta_\Delta)$ into a Lie coalgebra. This correspondence gives a functor
\begin{align*}
    \L: \C\A &\to \LC\A , \quad
    (C, \Delta, \varepsilon) \mapsto (C, \delta_\Delta)
\end{align*}
which obviously restricts to the conilpotent subcategories. Our aim is to define a right adjoint functor to the conilpotent one, which we shall call the \emph{universal enveloping coalgebra} functor.

First, we consider the cofree conilpotent coalgebra $(T^c( \mathfrak{c}),\Delta)$ together with the canonical projection 
$p\colon T^c(\mathfrak{c})\to \mathfrak{c}$. Define the maps
\begin{align}
\phi&= (p\otimes  p)\circ \delta_\Delta - \delta\circ p \ \colon T^c(\mathfrak{c})\to \mathfrak{c} \ten \mathfrak{c}	\label{eq:cogenerating-one}\\
\psi &=(\id\otimes  \phi \otimes  \id) \circ  \Delta^{(2)} \colon T^c(\mathfrak{c})\to 
T^c (\mathfrak{c}) \otimes  \mathfrak{c} \ten \mathfrak{c} \otimes  T^c (\mathfrak{c}). \label{eq:cogenerating-two}
\end{align}  

\begin{lemma}
\label{lemma-univ-envelop-coalg}
The kernel of $\psi$, denoted by $\U^c(\mathfrak{c})$, is the biggest conilpotent subcoalgebra of $T^c( \mathfrak{c})$ contained in $\ker \phi$. The map $p|_{\U^c(\mathfrak{c})}\colon 
\L(\U^c(\mathfrak{c}))\to \mathfrak{c}$ is a morphism of Lie coalgebras, which is universal in the sense that
for any other conilpotent coalgebra
$(C,\Delta_C)$ together with a map of Lie coalgebras $\pi: \L(C)\to \mathfrak{c}$, there exists a unique 
coalgebra morphism $\Pi\colon C\to \U^c(\mathfrak{c})$ such that the following diagram commutes
\begin{center}
\begin{tikzcd}
\L(C) \arrow[rr, "\L(\Pi)"] \arrow[rd, "\pi"'] &              & \L(\U^c(\mathfrak{c})) \arrow[ld, "p|_{\U^c(\mathfrak{c})}"] \\     & \mathfrak{c} &     
\end{tikzcd}.
\end{center}
\end{lemma}

\begin{proof}
Let us first check that $\U^c(\mathfrak{c})$ is a coalgebra: it is easy to see that 
\begin{align*}
	(\psi \otimes  \id )\circ \Delta =(\id^{\otimes  2}\otimes  \Delta) \circ \psi \quad \text{and} \quad (\id\otimes  \psi )\circ \Delta =(\Delta\otimes  \id^{\otimes  2}) \circ \psi
\end{align*} 
which implies that $\Delta(\U^c(\mathfrak{c}))\subseteq \U^c(\mathfrak{c})\otimes  \U^c(\mathfrak{c})$. 
Moreover, for any $x\in \U^c(\mathfrak{c})$ we have
	\begin{align*}
		\phi (x) =\big((\varepsilon \otimes  \id\otimes  \varepsilon) \circ \psi\big)(x)=0.
	\end{align*}
This means in particular $\U^c(\mathfrak{c})\subseteq \ker \phi$. Note that every other 
coalgebra being contained in $\ker \phi$ is automatically in $\ker \psi=\U^c(\mathfrak{c})$ and 
hence $\U^c(\mathfrak{c})$ is the biggest subcoalgebra with this property. 
Moreover, since $\U^c(\mathfrak{c})\subseteq \ker \phi$, the map $p$ is a Lie coalgebra morphism by 
construction. Let now $(C,\Delta_C)$ be a conilpotent coalgebra with Lie coalgebra map 
$\pi \colon \L(C)\to \mathfrak{c}$. Since $T^c(\mathfrak{c})$ is cofree, there exists a unique map 
$\Pi\colon C\to T^c(\mathfrak{c})$ such that $p\circ \Pi =\pi$. Due to the fact that the image of a 
coalgebra map is always a subcoalgebra, we just need to check that $\mathrm{im}(\Pi)\subseteq \ker \phi$, but this follows by
	\begin{align*}
		(p\otimes   p) \circ \L(\Delta)\circ \Pi &= (p\otimes   p)\circ (\Pi\otimes  \Pi) \circ \L(\Delta_C)
		=(\pi\otimes  \pi) \circ \L(\Delta_C) 
		= \delta \circ \pi =\delta \circ p\circ \Pi. 
	\end{align*}
\end{proof}

From now on, we will refer to $\U^c(\mathfrak{c})$ as the universal enveloping coalgebra of the Lie coalgebra $(\mathfrak{c}, \delta)$. 
Recall  that $T^c(\mathfrak{c})$ is a Hopf algebra
with commutative multiplication
\begin{align*}
(x_1\otimes  \cdots \ten x_i)\bullet_{Sh} (x_{i+1}\otimes  \cdots \otimes  x_k):=
\sum_{\sigma\in \mathrm{Sh}(i,k-i)} x_{\sigma^{-1}(1)}\otimes  \cdots \otimes   x_{\sigma^{-1}(k)},
\end{align*}
called the shuffle product, 
and antipode $S(x_1\otimes  \dots \otimes   x_k)=(-1)^{k}x_k\otimes   x_{k-1}\otimes  \dots \otimes  x_1$, where $\mathrm{Sh}(i,k-i)$ denotes the set of all $(i,k-i)$-shuffle permutations.
This structure induces a Hopf algebra structure on $\U^c(\mathfrak{c})$: in fact, one can check that the composition 
	\begin{align}
\label{eq:multiplication-univ-env-coalg}
\mu : \U^c(\mathfrak{c})\otimes \U^c(\mathfrak{c}) 
\hookrightarrow
T^c(\mathfrak{c})\otimes  T^c(\mathfrak{c})
\xrightarrow{\bullet_{Sh}} T^c(\mathfrak{c})
	\end{align}
is a coalgebra map whose image is contained in $\ker \psi$ and, similarly, one can check that $S$ restricts to $\U^c(\mathfrak{c})$. 
Such maps turn $\U^c(\mathfrak{c})$ into a commutative Hopf algebra.

\begin{remark}
Note that our definition of the functor $\mathcal{U}^c$ is right adjoint to $\L$ only if one restricts to the category of conilpotent coalgebras. Hence, to define it in this way only makes sense if the underlying Lie coalgebra is conilpotent as well. There are different constructions for which the universal enveloping coalgebra makes sense and behaves reasonably well with respect to comodules. This is explained in much more detail in \cite{Mich}. We however are not going in this direction, since we need the conilpotent version of the universal enveloping coalgebra. 
\end{remark}

Assume now that we have a Lie bialgebra $(\b,[\cdot,\cdot],\delta)$ whose underlying Lie coalgebra is conilpotent. We can endow its universal enveloping coalgebra with a Poisson Hopf algebra structure by 
    \begin{align}
    \label{Eq: Poissonbracket}
        \{\xi_1\cdots\xi_k,\eta_1\cdots \eta_\ell\}=\sum_{i,j}[\xi_i,\eta_j]\xi_1\overset{\overset{i}\wedge}{\cdots} \xi_k\cdot\eta_1\overset{\overset{j}\wedge}{\cdots}\eta_\ell. 
    \end{align}
Here we used that the universal enveloping coalgebra is free as a commutative algebra, see Appendix \ref{App: CoPBW}. 
The following theorem is a dual version of the Milnor-Moore Theorem. 
\begin{theorem}
\label{Thm: coMM}
Let $H$ be a commutative Hopf algebra.
\begin{enumerate}
\item The quotient $\mathrm{coPrim}(H):=\mathrm{coker}(m\colon \ker\varepsilon^{\ten 2}\to \ker\varepsilon)$ has a natural structure of Lie coalgebra. If moreover $H$ is conilpotent, then so is $\mathrm{coPrim}(H)$ as a Lie coalgebra, and we have $H\cong \U^c(\mathrm{coPrim}(H))$ as Hopf algebras. 
\item If $H$ is additionally a Poisson Hopf algebra, then $\mathrm{coPrim}(H)$ possesses a Lie bialgebra structure induced by the Poisson bracket. If moreover $H$ is conilpotent, then 
$H\cong \U^c(\mathrm{coPrim}(H))$ as Poisson Hopf algebras. 
\end{enumerate}
\end{theorem}
The last theorem says that there is an equivalence between the categories of conilpotent Poisson Hopf algebras and conilpotent Lie bialgebras, and the equivalence is established by the functors $\U^c$ and $\mathrm{coPrim}$ (in particular, the fact that $\mathrm{coPrim}\circ\U^c $ is isomorphic to the identity can be seen using the dual of the PBW Theorem, see Appendix \ref{App: CoPBW}). 
We now want to generalize these statements to the case of topologically free Lie bialgebras. In this case, we need to restrict ourselves to \emph{quasi-conilpotent Lie coalgebras:}
\begin{definition}
A topologically free Lie coalgebra $(\mathfrak{c}_\hbar\cong \mathfrak{c}[[\hbar]],\delta_\hbar)$ is called quasi-conilpotent if $(\mathfrak{c},\delta_0)$ is conilpotent. 
\end{definition}
Let us define the universal enveloping coalgebra in this setting in the same way as before. For a quasi-conilpotent topologically free Lie coalgebra $(\mathfrak{c}_\hbar\cong \mathfrak{c}[[\hbar]],\delta_\hbar)$, we define $T^c(\mathfrak{c}_\hbar):=T^c(\mathfrak{c})[[\hbar]].$
One can check that this is the cofree quasi-conilpotent coalgebra cogenerated by $\mathfrak{c}_\hbar$. We can now adapt the above construction of $\U^c$: we define 
\begin{align}
\phi&= (p\otimes  p)\circ \delta_\Delta - \delta\circ p \ \colon T^c(\mathfrak{c})[[\hbar]]\to (\mathfrak{c} \ten \mathfrak{c})[[\hbar]]	\label{eq:cogenerating-oneFP}\\
\psi &=(\id\otimes  \phi \otimes  \id) \circ  \Delta^{(2)} \colon T^c(\mathfrak{c})[[\hbar ]]\to 
(T^c (\mathfrak{c}) \otimes  \mathfrak{c} \ten \mathfrak{c} \otimes  T^c (\mathfrak{c}))[[\hbar]]. \label{eq:cogenerating-twoFP}
\end{align}  
And as in the case of $\U^c$, one can show that $\U^c(\mathfrak{c}_h):=\ker\psi$
is a Hopf algebra. Moreover, it is topologically free, since it is 
a kernel. We will not go too much into detail, but it is in fact 
the universal enveloping coalgebra in the category of quasi-conilpotent coalgebras. 

\begin{lemma}
\label{Lem: topfreecomHop}
Let $H$ be a topologically free quasi-conilpotent Hopf algebra.
\begin{enumerate}
\item If $H$ is codequantizable, then there exists a conilpotent Lie coalgebra $\mathfrak{c}$ (over $\mathbb{K}$), such that $H/\hbar H\cong \U^c(\mathfrak{c})$ as Hopf algebras. 
\item If $H$ is commutative, then its coprimitive elements form a topologically free quasi-conilpotent Lie coalgebra and we have $H\cong \U^c(\mathrm{coPrim}(H)). $
\item If $H=\U^c(\mathfrak{c}_\hbar)$ for some quasi-conilpotent Lie coalgebra $\mathfrak{c}_\hbar$, then $\mathrm{coPrim}(H)\cong \mathfrak{c}_\hbar$. 
\end{enumerate}
\end{lemma}

\begin{proof}
    The first statement follows from Theorem \ref{Thm: coMM}, while the second part is an adaptation of the proof for topologically free quasi-conilpotent commutative Hopf algebras. The third part is a consequence of the dual of the PBW Theorem for topologically free Lie coalgebras, see Appendix \ref{App: CoPBW}.
\end{proof}

Let us denote by $\L\B\A^0(\mathrm{tfMod}_{\mathbb{K}[[\hbar]]})$
the category of topologically free quasi-conilpotent Lie bialgebras for which  the Lie bracket vanishes for $\hbar = 0$. We shall denote its objects by $(\mathfrak{b}_\hbar^+,[\cdot,\cdot]_\hbar,\delta_\hbar)$.
Note that in this case we do not need to pass to the category $\mathrm{csMod}_{\mathbb{K}[[\hbar]]}$, since topologically free modules have kernels. 
In particular, since we showed that $\U^c(\b)$ can be written as a kernel, we get that if $\b$ is topologically free, then $\U^c(\b)$ is topologically free as well.  
Using the PBW isomorphism from Appendix \ref{App: CoPBW}, for any $\mathfrak{b}_\hbar^+\in \L\B\A^0(\mathrm{tfMod}_{\mathbb{K}[[\hbar]]})$ we can endow $\U^c(\mathfrak{b}_\hbar^+)$  with a coquantizable Poisson Hopf structure via formula \eqref{Eq: Poissonbracket}. One can check (using Lemma \ref{Lem: topfreecomHop}) that this establishes an equivalence of categories 
\begin{center}
	\begin{tikzcd}
		\L\B\A^0(\mathrm{tfMod}_{\mathbb{K}[[\hbar]]})\arrow[rr, shift left, "\U^c"] &  &
		\mathrm{cqPoissH}_{qcn}(\mathrm{tfMod}_{\mathbb{K}[[\hbar]]})\arrow[ll, shift left, "\mathrm{coPrim}"].
        \end{tikzcd}
\end{center}

Therefore, applying Theorem \ref{Q-DQ-Eq} we get the equivalences of categories 
\begin{center}
	\begin{tikzcd}
		\L\B\A^0(\mathrm{tfMod}_{\mathbb{K}[[\hbar]]})\arrow[rr, shift left, "\U^c"] &  &
		\mathrm{cqPoissH}_{qcn}(\mathrm{tfMod}_{\mathbb{K}[[\hbar]]})\arrow[ll, shift left, "\mathrm{coPrim}"] 
        \arrow[rr, shift left, "\mathcal{Q}^\Phi_+"]  & & 
        \mathrm{cdqHopf}_{qcn}(\mathrm{tfMod}_{\mathbb{K}[[\hbar]]})\arrow[ll, shift left, "\mathcal{D}^\Phi_+"],
        \end{tikzcd}
\end{center}
where we again used the index $qcn$ for quasi-conilpotent and the arguments why quantization and dequantization corestrict to these subcategories is the same as in the previous section. This picture describes the quantization-dequantization correspondence for universal enveloping coalgebras. 
Next, we prove the corresponding result for modules. In the case of conilpotent Lie coalgebras we need to adjust the Drinfeld-Yetter modules.
\begin{definition}
Let $(\mathfrak{c},\delta)$ be a Lie coalgebra. A right Lie $\mathfrak{c}$-comodule $(V, \rho)$ is said to be conilpotent if for every $v\in V$ there exists $n\in \mathbb{N}$, such that $\rho^n(v)=0$. 
\end{definition}
The following result is the dual version of Lemma \ref{lemma_isom_DY_to_DY}:
\begin{lemma}
\label{lemma_isom_DY_to_DY_dual}
The categories $\overline{\mathfrak{DY}}(\U^c(\b_\hbar))$ and $\D\Y(\b_\hbar)_{cn}$ are isomorphic symmetric Cartier categories.
\end{lemma}

Let us now consider $(\b_\hbar ^+,[\cdot,\cdot]_\hbar,\delta_\hbar)\in \L\B\A^0(\mathrm{tfMod}_{\mathbb{K}[[\hbar]]})$ and define 
    \begin{align*}
        \mathfrak{D}_+(\b_\hbar^+):=\{(V,\pi_V,\rho_V) \in \D\Y(\b_\hbar^+)\ | \ \rho_V|_{\hbar=0} \text{ conilpotent, }  \pi_V=0+\mathcal{O}(\hbar)\}.
    \end{align*}
Using the same arguments as in the previous section, we can apply Theorem \ref{Thm: RepCatsEq} to get the equivalences of categories: 

\begin{center}
	\begin{tikzcd}
		\mathfrak{D}_+(\b_\hbar^+)\arrow[rr,"Lemma \  \ref{lemma_isom_DY_to_DY_dual}", shift left] &  &
		\mathfrak{M}_+\big(\U^c(\b_\hbar^+)\big)\arrow[ll, shift left] 
        \arrow[rr, "Eq. \eqref{eq:equivalence_2}",shift left]  & & 
        \M_+\big(\mathcal{Q}_+^\Phi(\U^c(\b_\hbar^+))\big)_{g(0)}\arrow[ll, "Eq. \eqref{eq:equivalence_4}", shift left].
        \end{tikzcd}
\end{center}

\begin{remark}
Note that both quantization and codequantization procedures described above are a particular case of a more global picture. Indeed, with the same techniques one can quantize (resp. codequantize) a topological coPoisson (resp. Poisson) Hopf algebra whose Poisson (cobracket) (resp. bracket) has vanishing constant term (with respect to $\hbar$). The same reasoning applies to the corresponding modules.
\end{remark}

\subsection{Tamarkin's proof of Deligne's conjecture}
\label{SubSec: Tamarkin}
In \cite{tamarkin} D. Tamarkin proposed a proof of M. Kontsevich's Formality Theorem \cite{Kon97} by proving that the Hochschild cochain complex of an algebra possesses a so-called $G_\infty$-algebra structure (here $G$ stands for \emph{Gerstenhaber}, and $G_\infty$ means Gerstenhaber up to higher coherent homotopies). The question of whether such a structure exists on the Hochschild complex was first asked by P. Deligne in 1993 and is usually referred to as \emph{Deligne's conjecture}.

The idea of Tamarkin was to use a dequantization functor. Without going into too much detail, we want to show in this section how to construct this remarkable $G_\infty$-structure. For a detailed discussion, we refer the reader to Tamarkin's original work \cite{tamarkin}, or to the excellent exposition of V. Hinich \cite{hinich}. 

Tamarkin relied on the classical quantization of Lie bialgebras by Etingof and Kazhdan to show that there is an equivalence of two operads, which allowed him to perform a dequantization of a universal enveloping coalgebra. 
The advantage of our methods consists in providing an explicit dequantization of the universal enveloping coalgebra. Let us illustrate this by presenting Tamarkin's result in the following lines.

\subsubsection{Filtered vector spaces}
\begin{definition}
A vector space $V$ is said to be filtered if it possesses an increasing filtration 
 \begin{align*}
0=V_{(-1)}\subseteq V_{(0)}\subseteq V_{(1)}\subseteq\dots \subseteq V \qquad \text{such that} \qquad V=\bigcup_{n\geq 0}V_ {(n)}.
\end{align*}
A morphism of filtered vector spaces is a linear map $f:V \to W$ which respects the filtration. 
\end{definition}
It is well-known that the tensor product of two filtered vector spaces is canonically filtered by
\begin{align*}
(V\ten W)_{(n)}= \sum_{i+j=n} V_{(i)}\ten W_{(j)}
\end{align*}
turning the category of filtered vector spaces into a monoidal category, which we denote by $\mathrm{FVect}$.
Note that, given two filtered vector spaces $V,W$, the vector space $\mathrm{Hom}(V,W)$ is canonically filtered by $\mathrm{F}^k\mathrm{Hom}(V,W)=\{f\in \Hom(V,W)\mid f(V_{(n)})\subseteq W_{(n-k)} \text{ for all } n\in \mathbb{N}\}.$
Moreover, this filtration is separated and complete.

Let us now consider the category of cochain complexes in $\mathrm{FVect}$ with (degree 0) cochain morphisms as morphisms, i.e. 
an object $(V,d)$ is a $\mathbb{Z}$-graded vector space $V=\bigoplus_{i\in \mathbb{Z}}V^i$ together with a degree $+1$ differential $d\colon V\to V$, such that every degree $V^i$ is filtered and $d$ respects the filtration. 
Note that the filtration of the morphisms restricted to cochain maps is still complete and separated. 
This also forms a braided monoidal category: the tensor product of $(V,d_V)$ and $(W,d_W)$ is given by 
    \begin{align*}
        (V\ten W)^i=\bigoplus_{k+l=i}V^k\ten W^l,
    \end{align*}
where the tensor product is understood in the category of filtered vector spaces. Moreover, the differential $d_{V\ten W}$ is given by $d_{V\ten W}\big|_{V^k\ten W^l}=d_V\ten \id +(-1)^k (\id \ten d_W)$
and the braiding is given by $\sigma_{V,W}\big|_{V^k\ten W^l} =(-1)^{kl}\tau_{V,W},$
where $\tau$ is the usual tensor flip. We shall denote this complete and separated filtered braided monoidal category $\mathrm{CFVect}$. 
\subsubsection{The Hochschild complex and its Bar resolution}
Let $A$ be an associative algebra. 
Recall the Hochschild complex of $A$, given by $C^k(A):=\Hom(A^{\ten k},A),$
where $C^{0}(A)=A$. It is well-known that the operations
\begin{align*}
\delta\colon C^\bullet(A)\to C^{\bullet+1}(A) \qquad \text{and} \qquad \cup\colon C^\bullet(A)\ten C^\bullet(A)\to C^{\bullet+\bullet}(A)
\end{align*} 
respectively defined by 
\begin{align*}
        \delta\phi(a_1,\dots,a_{k+1})=& a_1\phi(a_2,\dots,a_{k+1})+
\sum_{i=1}^k (-1)^i\phi(a_1,\dots,a_ia_{i+1},\dots,a_{k+1})\\&
+(-1)^{k+1}\phi(a_1,\dots,a_k)a_{k+1}
    \end{align*} 
and
\begin{align*} \phi\cup\psi(a_1,&\dots,a_{k+\ell})=\phi(a_1,\dots,a_k)\psi(a_{k+1},\dots,a_{k+\ell})
   \end{align*}
turn $C^\bullet(A)$ into a differential graded 
associative algebra. Moreover, we have the so-called braces: for 
$\phi\in C^k$ and $\phi_i\in C^{k_i}(A)$ for $i\in {1,\dots, n}$ with 
$n\leq k$ we define
    \begin{align*}
        \phi & \{\phi_1,\dots,\phi_n\}(a_1,\dots,a_m):=\\&
        \sum\epsilon(i_1,\dots,i_n)\phi(a_1,\dots,\phi_1(a_{i_1+1},\dots, a_{i_1+k_1}),
        a_{i_1+k_1+1},\dots, \phi_n(a_{i_n},\dots,a_{i_n+k_n}),a_{i_n+k_n+1},\dots,a_m)
    \end{align*}
where $\epsilon(i_1,\dots,i_n):=(-1)^{\sum_{j=1}^n i_j(k_j-1)}$ and $\phi\{\phi_1,\dots,\phi_l\}=0$ for $n>k$. 
Next, consider the graded cofree conilpotent counital coalgebra cogenerated by $C^\bullet(A)[1]$, denoted
by $(T^c(C^\bullet(A))[1],\Delta)$, where we extend the cup product and the Hochschild differential as coderivations, giving a dg coalgebra structure $(T^c(C^\bullet(A)[1]),\Delta,\partial=\delta+\cup)$. We define now the sequence of maps \[m_{k,\ell}\colon T^{c,k}(C^\bullet(A)[1])\ten T^{c,\ell}(C^\bullet(A)[1])\to C^{\bullet}(A)[1],\] where $m_{0,0}=0$, $m_{0,1}=m_{1,0}=\id$ and $m_{1,k}(\phi\ten(\phi_1\ten\dots\ten\phi_k)) =\phi\{\phi_1,\dots,\phi_k\}$.
One can check that such maps are of degree $0$ and since $T^c(C^\bullet(A)[1])$ is cofree we get a coalgebra morphism 
$m\colon T^c(C^\bullet(A)[1])\ten T^c(C^\bullet(A)[1]) \to T^c(C^\bullet(A)[1])$. One can also check that $(T^c(C^\bullet(A)[1]),\Delta,m,\partial)$
is a dg-bialgebra, and since it is conilpotent, there exists a matching antipode $S\colon T^c(C^\bullet(A)[1])\to T^c(C^\bullet(A)[1])$ turning $(T^c(C^\bullet(A)[1]),\Delta,m,S,\partial)$
into a dg-Hopf algebra. 

\begin{remark}
    Note that the construction above does not depend on the exact form of the Hochschild cochains, but only on their brace algebra structure. In fact,
    one can construct this Hopf algebra for any brace algebra. 
\end{remark}

We now introduce the filtration 
on $T^c(C^\bullet(A))$ by tensor powers, i.e. 
    \begin{align*}
    T^c(C^\bullet(A)[1])_{(n)}=\bigoplus_{i=0}^n T^{c,i}(C^\bullet(A)).
    \end{align*}

By construction of $m$, one can check that for $X\ten Y\in T^{c,n}(C^\bullet(A)[1]) \ten T^{c,m}(C^\bullet(A)[1])$ one has
$m(X\ten Y)= X\bullet_{sh}Y \mod T^c(C^\bullet(A)[1])_{(n+m-1)}$,
where $\bullet_{sh}$ is the graded shuffle product, which is graded commutative (i.e. $\sigma$-commutative). This means 
that  $(T^c(C^\bullet(A)[1]),\Delta,m,S,\partial)$ is a codequantizable Hopf algebra (in the sense of Definition \ref{Def: goodHopf}) in the category $\mathrm{CFVect}$. 

\subsubsection{Tamarkin's Theorem}

Let $\Phi$ be a Drinfeld associator. Then, using Theorem 
\ref{Thm: QandDQfunctors}, we can assign a Poisson Hopf algebra to $(T^c(C^\bullet(A)[1]),\Delta,m,S,\partial)$ by applying the codequantization functor $\mathcal{D}_+^\Phi$. We shall denote 
it by $(T^c(C^\bullet(A)[1]),\hat{\Delta},\hat{m},\hat{S},\{\cdot,\cdot\}_{DQ},\partial)$. Note again that the underlying object (which in this case is the filtered cochain complex $(T^c(C^\bullet(A)[1]),\partial)$) is not changed by the functor $\mathcal{D}_+^\Phi$. 

\begin{remark}
\label{Gerstenhaberbracket}
It is simple to compute that the Poisson bracket evaluated on two elements of $C(A)[1]\subseteq T^c(C(A)[1])$ gives their Gerstenhaber bracket, i.e. $\{\phi,\psi\}_{DQ}=\phi\{\psi\}\mp \psi\{\phi\}=:[\phi,\psi]_G$. 
\end{remark}

Let us now briefly sketch why this implies that the conilpotent cofree Lie coalgebra cogenerated by $C^\bullet(A)[1]$ carries the structure of a Lie bialgebra. Recall that for every vector space $V$ there is a canonical isomorphism $\U^c(\mathrm{coLie}(V))\cong T^c(V)$ as commutative algebras, where $T^c(V)$ is endowed with the shuffle product. We are using now the commutativity of $(T^c(C^\bullet(A)[1])),\hat{m})$  and the classical map 
    \begin{align*}
        \iota\colon \mathrm{coLie}(C^\bullet(A)[1])\to 
        \U^c(\mathrm{coLie}(C^\bullet(A)[1]))\cong T^c(C^\bullet(A)[1]))
    \end{align*}
where the last identification is understood as an isomorphism of vector spaces. Therefore, we get an algebra morphism 
from the free unital commutative algebra $S (\mathrm{coLie}(C^\bullet(A)[1]))\to T^c(C^\bullet(A)[1]))_0,$
which is in fact an isomorphism. Moreover, we use the (inverse of the) dual PBW isomorphism $S (\mathrm{coLie}(C^\bullet(A)[1]))\to \U^c(\mathrm{coLie}(C^\bullet(A)[1]))$ (see Appendix \ref{App: CoPBW}) to obtain  an isomorphism of commutative algebras $ \U^c(\mathrm{coLie}(C^\bullet(A)[1])) \to T^c(C^\bullet(A)[1])). $
Using this isomorphism,  we can assume that we have a Hopf algebra structure $\U^c(\mathrm{coLie}(C^\bullet(A)[1]))=
    (\U^c(\mathrm{coLie}(C^\bullet(A)[1])),\tilde{m},\tilde{\Delta}, \tilde{S}, \tilde{\partial})$
such that $m(0)=\bullet_{sh}$. This means that we get a dg Lie bialgebra structure on $\mathrm{coLie}(C^\bullet(A)[1]))\cong\mathrm{coPrim}(\U^c(\mathrm{coLie}(C^\bullet(A)[1]))). $
Note that, by construction, we have $$\tilde\Delta=\Delta+ \mathrm{F}^1\Hom(T^c(C^\bullet(A)[1]),T^c(C^\bullet(A)[1])\ten T^c(C^\bullet(A)[1])).$$ This implies that the Lie cobracket of $\mathrm{coLie}(C^\bullet(A)[1]))$ differs from the cofree one by a morphism in the filtration degree
$\mathrm{F}^1\Hom(\mathrm{coLie}(C^\bullet(A)[1]),\mathrm{coLie}(C^\bullet(A)[1])\ten \mathrm{coLie}(C^\bullet(A)[1]))$ and therefore, we get that the Lie coalgebra morphism induced by $\mathrm{coLie}(C^\bullet(A)[1])\to C^\bullet(A)[1]$ (using the cofreeness) it is an isomorphism. This means, in 
particular, that we obtain a dg Lie bialgebra structure on the 
cofree Lie coalgebra 
$\mathrm{coLie}(C^\bullet(A)[1]))$. Note that every identification discussed above is natural, meaning that the 
bracket and the differential are also natural. In short, we have the following
\begin{theorem}
Let $A$ be an associative algebra. Then there is a natural dg Lie bialgebra structure on the cofree Lie coalgebra $\mathrm{coLie}(C^\bullet(A)[1])$.
\end{theorem}

Note that it is easy to see, by tracing back the isomorphisms and using Remark \ref{Gerstenhaberbracket}, that the Lie bracket restricted to $C^\bullet(A)[1]\subseteq \mathrm{coLie}(C^\bullet(A)[1])$ is given by the Gerstenhaber bracket. 
We also mention that a dg Lie bialgebra structure on the cofree Lie coalgebra extending the cofree Lie cobracket is called by Hinich a $\tilde{B}$-algebra structure. Such a structure is a special case of a $G_\infty$-algebra. We refer the reader to \cite{hinich} for details. 

This means in particular, that we can deduce Tamarkin's theorem: 

\begin{theorem}[Tamarkin \cite{tamarkin}]
\label{Thm: Tamarkin}
Any associative algebra $A$ induces a natural 
structure of a $G_\infty$-algebra on $C^\bullet(A)$ extending the differential graded Lie algebra consisting of the Hochschild differential and the Gerstenhaber bracket.
\end{theorem}

Note that Tamarkin proved that this theorem rather directly implies Kontsevich's formality theorem \cite{Kon97}, highlighting its fundamental importance. We hope that this simplification of one aspect of Tamarkin's proof, which is induced by our treatment of Poisson Hopf algebras, will help to understand the remarkable $G_\infty$ structure on the Hochschild complex better and, with this, also formality theory. 

\section{Outlook}
\label{section-Outlook}
By formulating the quantization-dequantization correspondence in this general setting, our results open the door to several further developments.
Let us briefly point out some of them: 
\begin{enumerate}
\item It is already formulated in \cite{Sev16}, that one can quantize the smooth functions on a Poisson Lie-group by dualizing the framework of adapted functors. In this paper, we provided a precise framework for showing this statement. In fact, the smooth functions on a Lie group are a Hopf monoid in the (abelian) category of nuclear Fr\'echet spaces. Nevertheless, for a real Lie bialgebra $\mathfrak{g}$ integrating to a Poisson Lie-group $G$ with dual Poisson Lie group $G^*$ we can now attach four (co)quantizations to it, namely those of 
\begin{align*}
\U(\g)\ \quad -\quad  \ \U^c(\g)\ \quad -\quad  \ C^\infty(G)\ \quad -\quad  \ C^\infty(G^*).
\end{align*}
The exact relation between them is unknown to us, but we believe that, with our approach, one is able to obtain a clear picture in which frame these (co)quantizations fit into.
\item In \cite{PulSev}, the authors provided a quantization technique of infinitesimally braided Poisson Hopf algebras to braided Hopf algebras. We are certain that one can modify our construction in an obvious way (in Theorem \ref{Thm: constrPoissonHopf}, we drop the condition that the target category has trivial infinitesimal braiding) to arrive to the same results and also to find the dequantization using our techniques. In fact, the authors' construction relies on nerves of Hopf algebras which are subcategories (consisting of tensor powers of the corresponding (co)adjoint (co)monoids) of the Drinfeld-Yetter modules. 

\item A bit more subtle is the quantization-dequantization correspondence of Lie bialgebroids \cite{Xu}, or more generally, their algebraic counterparts modeled on Lie-Rinehart algebras.
\item As already seen in Section \ref{Sec:Appl}, the      
dequantization of Tamarkin's Hopf algebra leads to a $G_\infty$- algebra structure on the Hochschild cochains. The Hochschild cochains and the Hochschild chains form together what is known as a homotopy (pre)calculus (see \cite{Tsygan} for the precise definition), which means that the Hochschild chains are a special kind of module over the Hochschild cochains (seen as a $G_\infty$-algebra). We believe that the dequantization functor for Yetter-Drinfeld modules, or a variation thereof, can be used to show its existence in a very clear and systematic way.  
\end{enumerate}

\begin{appendix}
    \section{The dual of the Poincar\'e-Birkhoff-Witt Theorem in the conilpotent case}
    \label{App: CoPBW}
    In Section \ref{SubSec: UECoalgebras} we briefly recalled the construction of the conilpotent universal enveloping algebra of a conilpotent Lie coalgebra $(\mathfrak{c},\delta)$. The aim of this appendix is to sketch a proof of the 
    dual version of the PBW theorem in the conilpotent case, i.e. 
    that the morphism of commutative algebras
        \begin{align*}
            \mathrm{pbw}^*\colon \U^c(\mathfrak{c})\hookrightarrow T^c(\mathfrak{c})\twoheadrightarrow S(\mathfrak c) 
        \end{align*}
    is an isomorphism, where the second arrow is the projection to completely symmetric tensors. We are aware that this theorem is most likely well-known, but we did not find a precise statement of it in the literature. For example, in \cite{Mich2} the dual version of the PBW theorem is proven, but not in the conilpotent case. 
    We decided to follow the (dualized) arguments from \cite{QFSI}, since in there the classical proof of the PBW Theorem can be easily generalized to topologically free Lie algebras, or Lie coalgebras in our case. 
    
Consider the decreasing filtration   
\begin{align*}
T^c(\mathfrak{c})^{(k)}:=\bigoplus_{i=k}^\infty \mathfrak{c}^{\ten i}, \qquad \bigcap_{k\in \mathbb{N}}T^c(\mathfrak{c})^{(k)}=\{0\}
\end{align*}
and the corresponding induced filtrations on $\U^c(\mathfrak{c})$ and $S(\mathfrak{c})$
\begin{align*}
\U^c(\mathfrak{c})^{(k)}=\U^c(\mathfrak{c})\cap T^c(\mathfrak{c})^{(k)} 
\quad \text{ and } \quad
S(\mathfrak{c})^{(k)}=\bigoplus_{i=k}^\infty S^i(\mathfrak{c})
\end{align*}
which turn $\mathrm{pbw}^*$ into a map of filtered vector spaces. Consider the maps 
        \begin{align*}
            \mathrm{pbw}_k^*\colon \U^c(\mathfrak{c})/\U^c(\mathfrak{c})^{(k)}\to S(\mathfrak{c})/S(\mathfrak{c})^{(k)},
        \end{align*}
which are clearly injective for $k=0,1,2$. Using that $\U^c(\mathfrak{c})\subseteq T^c(\mathfrak{c})$, one can show that if $\mathrm{pbw}^*_k$ is injective, so is $\mathrm{pbw}_{k+1}^*$; this implies that $\mathrm{pbw}^*$ is injective as well. Note that, at this point, we did not use the conilpotency of the Lie coalgebra yet. For the surjectivity, we observe that the canonical map $p\colon \U^c(\mathfrak{c})\to \mathfrak{c}$ is given by $p=\mathrm{pr}_1\circ \mathrm{pbw}^*$,
where $\mathrm{pr}_1\colon S(\mathfrak{c})\to \mathfrak{c}$ is the canonical projection to symmetric tensors of length one. This means that, since $\mathrm{pbw}^*$ is a morphism of algebras, in order to show that $\mathrm{pbw}^*$ is surjective, it is enough to prove that $p$ is surjective.  
There are many ways to do so; in what follows, we shall use the following strategy: we construct a conilpotent coalgebra $C$ together with a surjective Lie coalgebra morphism $\L(C)\to \mathfrak{c}. $ Then, using the universal property of the universal enveloping coalgebra (see Lemma \ref{lemma-univ-envelop-coalg}), we shall conclude that also $p\colon \U^c(\mathfrak{c})\to \mathfrak{c}$ is surjective. 
In order to define such a coalgebra $C$, we dualize the construction in the proof of \cite[Chapter 1.3.7]{QFSI}. As a vector space, we set $C = S(\mathfrak{c})$. We denote by $\mu_k\colon T^k(\mathfrak{c})\to S^k(\mathfrak{c})$ the $k$-fold product. Next, we want to define a map $\kappa\colon S(\mathfrak{c})\to \mathfrak{c}\ten S(\mathfrak{c}).$
In order to do so, we first define 
\begin{align*}
R_\delta\colon T(\mathfrak{c}) \to \mathfrak{c}\ten S(\mathfrak{c}), \quad  x_1\ten \dots\ten x_k \mapsto 
\frac{1}{(k+1)!}\sum_{i=1}^k(k+1-i) x_i'
\ten x_1\cdots x_i''\cdots x_k,
\end{align*}
    where we used the notation $\delta(x)=x'\ten x''$. We define now 
    $\kappa$ inductively by splitting it into 
        \begin{align*}
            \kappa_\ell\colon S^k(\mathfrak{c})\to\mathfrak{c}\ten S^{k+l}(\mathfrak{c})
        \end{align*}
    and set $\kappa_{-1}(x_1\cdots x_k)=\sum_{i=1}^k x_i\ten x_1\cdots
    x_{i-1}\cdot x_{i+1} \cdots x_{k}$ and define $\kappa_{\ell+1}\colon S^k(\mathfrak{c})\to \mathfrak{c}\ten S^{k+\ell+1}(\mathfrak{c})$ by
        \begin{align}
        \label{eq:kappa}
            \kappa_{\ell+1}=R_\delta \circ (\id\ten\mathrm{pr}_1)\circ  \kappa^{k+\ell},
        \end{align}
    where $\kappa^{k+\ell}\colon S(\mathfrak{c})\to \mathfrak{c}^{\ten k+\ell}\ten S(\mathfrak{c})$ is the $(k+\ell)$-fold application of $\kappa$ and $\mathrm{pr}_1\colon S(\mathfrak{c})\to \mathfrak{c}$ is again the projection to symmetric tensors of length one. 
    One can check that the right hand side of Equation \eqref{eq:kappa} depends only on $\kappa_{k+i}$ for $i\leq \ell$. 
    So we write $\kappa=\sum_{\ell=-1}^\infty\kappa_\ell$. Note that $\kappa$ is well-defined because $\mathfrak{c}$ is conilpotent and, moreover, that $\kappa_\ell$ consists, up to combinatorial factors, of $\ell+1$-fold applications of the Lie cobracket $\delta$. 
    We define now the coproduct $\Delta\colon S(\mathfrak{c})\to S(\mathfrak{c})\ten S(\mathfrak{c})$ by 
        \begin{align*}
            \Delta=\sum_{\ell=0}^\infty \frac{1}{\ell!}(\mu_{\ell}\ten \id)\circ \kappa^\ell,
        \end{align*}
    where we set  $\kappa^0(x):=1\ten x$ for all $x\in S(\mathfrak{c})$. Again, by the conilpotency of $\mathfrak{c}$, this map is well defined. 
    Moreover, if we denote by $\varepsilon \colon S(\mathfrak{c})\to \mathbb{K}$ the projection to symmetric tensors of length $0$, we immediately get that $(\varepsilon \ten \id)\circ \Delta=\id$. In order to show that $(\id\ten \varepsilon)\circ \Delta=\id$, one defines the maps 
    \begin{align*}
        \phi_\ell\colon S^k(\mathfrak{c}) \to S^{k+\ell}(\mathfrak{c}),  \quad  x \mapsto \mathrm{pr}_{k+l}\Big(\big((\id\ten \varepsilon)\circ \Delta\big)(x) \Big)  ,
    \end{align*}
    where $\mathrm{pr}_{k+\ell}\colon S(\mathfrak{c})\to S^{k+\ell}(\mathfrak{c})$ is the canonical projection. 
    One can easily check, using the anti-symmetry of $\delta$, that $\phi_0=\id$ and $\phi_1=0$. Inductively, one shows that $\phi_\ell=0$ for $\ell\geq 1$. 
    \begin{lemma}
    \label{Lem:PBW}
        $(S(\mathfrak{c}),\Delta,\varepsilon)$ is a conilpotent coassociative coalgebra. Moreover, the canonical projection $\mathrm{pr}_1\colon S(\mathfrak{c})\to \mathfrak{c}$ induces a morphism of Lie coalgebras 
         $\mathrm{pr}_1\colon \mathscr{L}(S(\mathfrak{c}))\to \mathfrak{c}$. 
    \end{lemma}
    \begin{proof}(Sketch)
     The conilpotency of $\Delta$ follows from the conilpotency of 
     $\mathfrak{c}$. By the discussion before the Lemma, $\varepsilon$ is a counit. Using the explicit formulas of $\Delta$, one sees that 
        \begin{align*}
            \Delta\colon S^k(\mathfrak{c})\to \bigoplus_{i+j\geq k }
            S^i(\mathfrak{c})\ten S^j(\mathfrak{c}),
        \end{align*}
    so $(\mathrm{pr}_1\ten \mathrm{pr}_1)\circ \Delta$ vanishes when restricted to $S^k(\mathfrak{c})$ for $k\geq 3$. Using again the explicit formulas, one can see that
        \begin{align}
        \label{Eq:defcoprodorder0}
            \Delta(x_1\cdots x_k)-\sum_{i=0}^k\sum_{\sigma\in \mathrm{Sh}(i,k-i)} x_{\sigma(1)}\cdots x_{\sigma(i)}\ten x_{\sigma(i+1)}\cdots x_{\sigma(k)} \in \bigoplus_{i+j\geq k+1 }
            S^i(\mathfrak{c})\ten S^j(\mathfrak{c}).
        \end{align}
     This implies in particular that $(\mathrm{pr}_1\ten \mathrm{pr}_1)\circ (\Delta-\tau\circ\Delta)$
    vanishes on $S^k(\mathfrak{c})$ for $k\geq 2$. One can check that 
    for $x\in\mathfrak{c}$ one has
        \begin{align*}
         \Delta(x)-x\ten 1 - 1\ten x- \tfrac{1}{2}\delta(x) \in \bigoplus_{i+j\geq 3 }
            S^i(\mathfrak{c})\ten S^j(\mathfrak{c}),   
        \end{align*}
    which implies that $\mathrm{pr}_1\ten \mathrm{pr}_1$ maps the cocommutator to the cobracket. The only thing left to show is the coassociativity of $\Delta$, which we prove via induction in the following sense: define the maps
        \begin{align*}
            \mathrm{Ass}_\ell\colon S^k(\mathfrak{c})\to 
            \bigoplus_{p+q+r=k+\ell}S^p(\mathfrak{c})\ten S^q(\mathfrak{c})\ten S^r(\mathfrak{c})
        \end{align*}
    by $(\Delta\ten \id)\circ \Delta- (\id\ten \Delta)\circ \Delta$ followed by the respective projections. By definition, it is clear that $\mathrm{Ass}_\ell$ vanishes if and only if $\Delta$ is coassociative. 
    Note that by Equation \ref{Eq:defcoprodorder0}, we can immediately conclude that 
    $\mathrm{Ass}_0=0$. Let us now assume that $\mathrm{Ass}_\ell=0$ for all $0\leq \ell \leq N+1$ and define the maps $ \phi_{N+1}\colon S^k(\mathfrak{c})\to T^{k+N+1} (\mathfrak c)$ by $\phi_{N+1}=\mathrm{pr}_1^{\ten k+N+1}\circ \Delta^{{k+N}}$
    where $\Delta^{{k+N}}\colon S(\mathfrak{c})\to T^{k+N+1}(S(\mathfrak{c}))$ is defined recursively by $\Delta^1=\Delta$ and 
    $\Delta^{k+1}=(\id^{\ten k}\ten \Delta)\circ \Delta^{k}$. 
    Note that the sum $\phi=\sum_{\ell=0}^\infty \phi_\ell$ is well defined and injective: its left inverse are the maps $M=\sum_{\ell=0} \frac{1}{\ell!}\mu_\ell\colon T(\mathfrak{c})\to S(\mathfrak{c})$. This means in particular that the total symmetrization of $\phi_{N+1}$ vanishes. 
    
    Let us now define the maps
        \begin{align*}
            \Phi_{N+1}\colon S^k(\mathfrak{c}) \to  T^{k+N+1}(\mathfrak{c}), \quad  x\mapsto
            (\id-\tau)\ten \id^{\ten k+N-1}(\phi_{k+N+1}(x))-\delta\ten\id^{k+N-1}(\phi_{N}(x)).
        \end{align*}
One can check the following three properties of $\Phi_{N+1}$: 
\begin{enumerate}
            \item it is anti-symmetric in the first two tensor factors.
            \item It is totally symmetric in the remaining tensor factors.
            \item The sum of the cyclic permutation of the first three tensor factors vanishes. 
\end{enumerate}
    This can be checked along the same lines as in the proof of \cite[Lemma 1.3.7.5]{QFSI} using the induction hypothesis and the coJacobi identity.
    Moreover, again dualizing the reasoning of the same proof, one concludes that 
    $\Phi_{N+1}=0$. 
    Using this, one shows that the maps 
        \begin{align*}
            (\phi \ten \id  )\circ \Delta \quad\text{ and } \quad\sum_{\ell=0}^\infty \kappa^\ell \qquad \colon S^k(\mathfrak{c})\to T(\mathfrak{c})\ten S(\mathfrak{c})
        \end{align*}
    coincide after projecting to $\bigoplus_{i+j=k+N+1}\mathfrak{c}^{\ten i}\ten S^j(\mathfrak{c})$, which one does by using the maps $\Phi_{N+1}$, again in a dual manner as in \cite{QFSI}. Recall that the left inverse to $\phi$ is given by $M=\sum_{\ell=0}^\infty \frac{1}{\ell!}\mu_\ell$. Now, we compute
        \begin{align*}
            (\Delta\ten\id)\circ \Delta&=(M\ten M\ten \id)\circ (\phi\ten \phi\ten \id)\circ 
             (\Delta\ten\id)\circ \Delta\\&
             = (M\ten M\ten \id)\circ (\Delta^c\ten\id)\circ (\phi\ten \id) 
             \circ \Delta\\&
             =(M\ten M\ten \id)\circ (\Delta^c\ten\id)\circ \sum_{\ell=0}^\ell\kappa_\ell=(\id\ten\Delta)\circ\Delta,
        \end{align*}
    where we denoted by $\Delta^c$ the classical deconcatenation coproduct of $T(\mathfrak{c})$. The last equation is a straightforward computation,
    and the equalities $\phi\ten \phi\circ \Delta=\Delta^c\circ \phi$ and
    $ (\phi \ten \id  )\circ \Delta= \sum_{\ell=0}^\infty \kappa^\ell$ hold by the induction hypothesis after projecting to the respective tensor degrees. 
    \end{proof}
    Using now the discussion right before Lemma \ref{Lem:PBW}, we can thus deduce the following
    \begin{theorem}[PBW$^*$-Theorem]
    \label{Thm: coPBW}
        Let $\mathfrak{c}$ be a conilpotent Lie coalgebra, then $\mathrm{pbw}^*\colon \U^c(\mathfrak{c})\to S(\mathfrak c)$
    is a (filtered) isomorphism of algebras.
    \end{theorem}
It is now completely clear, by the structure of the proof, that for a quasi-conilpotent topologically free Lie coalgebra $(\mathfrak{c}_\hbar\cong \mathfrak{c}[[\hbar]],\delta_\hbar)$ the canonical map $\U^c(\mathfrak{c}_\hbar)\to S(\mathfrak{c})[[\hbar]]$ is an isomorphism as well. 
Moreover, it also holds in a broader categorical context; for instance in the framework of differential graded Lie coalgebras, see Section \ref{SubSec: Tamarkin}. 
\end{appendix}

\textbf{Contacts:}\\
~\\
andrea.rivezzi@matfyz.cuni.cz\\
jonaschristoph.schnitzer@unipv.it\\

\end{document}